\documentclass{amsart}

\usepackage{amssymb}
\usepackage{amsthm}
\usepackage{amsmath}
\usepackage{enumerate}

\usepackage[foot]{amsaddr}

\usepackage{hyperref}

\theoremstyle{plain}
\newtheorem{proposition}{Proposition}[section]
\newtheorem{theorem}[proposition]{Theorem}
\newtheorem{lemma}[proposition]{Lemma}
\newtheorem{corollary}[proposition]{Corollary}
\newtheorem{observation}[proposition]{Observation}

\theoremstyle{definition}
\newtheorem{example}[proposition]{Example}
\newtheorem{definition}[proposition]{Definition}
\theoremstyle{remark}
\newtheorem{remark}[proposition]{Remark}

\newtheorem*{question}{Question}

\DeclareMathOperator{\Aff}{Aff}
\DeclareMathOperator{\Aut}{Aut}

\DeclareMathOperator{\supp}{supp}

\DeclareMathOperator{\GL}{GL}

\DeclareMathOperator{\PGL}{PGL}

\DeclareMathOperator{\Hol}{Hol}
\DeclareMathOperator{\End}{End}

\DeclareMathOperator{\Span}{Span} 
 
\DeclareMathOperator{\Euc}{Euc} 
\DeclareMathOperator{\id}{id}

\DeclareMathOperator{\Bc}{\mathcal{B}}
\DeclareMathOperator{\Cc}{\mathcal{C}}
\DeclareMathOperator{\Dc}{\mathcal{D}}

\DeclareMathOperator{\Kc}{\mathcal{K}}

\DeclareMathOperator{\Vc}{\mathcal{V}}

\DeclareMathOperator{\Bb}{\mathbb{B}}
\DeclareMathOperator{\Cb}{\mathbb{C}}
\DeclareMathOperator{\Db}{\mathbb{D}}

\DeclareMathOperator{\Kb}{\mathbb{K}}

\DeclareMathOperator{\Nb}{\mathbb{N}}
\DeclareMathOperator{\Pb}{\mathbb{P}}
\DeclareMathOperator{\Rb}{\mathbb{R}}
\DeclareMathOperator{\Xb}{\mathbb{X}}
\DeclareMathOperator{\Yb}{\mathbb{Y}}

\newcommand{\abs}[1]{\left|#1\right|}

\newcommand{\norm}[1]{\left\|#1\right\|}
\newcommand{\wt}[1]{\widetilde{#1}}
\newcommand{\wh}[1]{\widehat{#1}}
\newcommand{\ip}[1]{\left\langle #1\right\rangle}

\begin{document}

\title{Subelliptic estimates from Gromov hyperbolicity}
\author{Andrew Zimmer}\address{Department of Mathematics, Louisiana State University, Baton Rouge, LA, USA}\curraddr{Department of Mathematics, University of Wisconsin-Madison, Madison, WI, USA}
\email{amzimmer2@wisc.edu}
\date{\today}
\subjclass[2010]{}

\begin{abstract} In this paper we prove: if the complete K\"ahler-Einstein metric on a bounded convex domain (with no boundary regularity assumptions) is Gromov hyperbolic, then the $\bar{\partial}$-Neumann problem satisfies a subelliptic estimate. This is accomplished by constructing bounded plurisubharmonic function whose Hessian grows at a certain rate (which implies a subelliptic estimate by work of Catlin and Straube). We also provide a characterization of Gromov hyperbolicity in terms of orbit of the domain under the group of affine transformations. This characterization allows us to construct many examples. For instance, if the Hilbert metric on a bounded convex domain is Gromov hyperbolic, then the K\"ahler-Einstein metric is as well. \end{abstract}

\maketitle

\tableofcontents

\section{Introduction}

Suppose that $\Omega \subset \Cb^d$ is a bounded pseudoconvex domain. Then a \emph{subelliptic estimate of order $\epsilon>0$ holds on $\Omega$} if there exists a constant $C>0$ such that 
\begin{align*}
\norm{u}_{\epsilon} \leq C ( \|\bar{\partial} u\|_0 + \|\bar{\partial}^* u\|_0)
\end{align*}
for all $u \in L^2_{(0,q)}(\Omega) \cap { \rm dom}(\bar{\partial}) \cap { \rm dom}(\bar{\partial}^*)$ and $1 \leq q \leq d$. Here $\norm{\cdot}_s$ denotes the $L^2$-Sobolev space norm of order $s$ on $(0,q)$-forms on $\Omega$,  $\bar{\partial}^*$ denotes the adjoint of $\bar{\partial}$ with respect to the $L^2$ inner product, and  $L^2_{(0,q)}(\Omega)$ denotes the space of $(0,q)$-forms with square integrable coefficients.

 In the case when $\Omega$ is smoothly bounded, subelliptic estimates have been extensively studied, culminating in Catlin's~\cite{C1987,C1983} deep work which asserts that a subelliptic estimate holds on a smoothly bounded pseudoconvex domain if and only if the boundary has finite type in the sense of D'Angelo. For more background, see the survey papers~\cite{BS1999,CD2010}.

In this paper we consider domains with non-smooth boundary. Previously, Henkin-Iordan-Kohn~\cite{HIK1996} established subelliptic estimates on strongly pseudoconvex domains with piecewise smooth boundary and Michel-Shaw~\cite{MS1998} established subelliptic estimates on strongly pseudoconvex domains with Lipschitz boundary. Straube~\cite{S1997} established subelliptic estimates on pseudoconvex domains with piecewise smooth boundary of finite type. Straube~\cite{S1997} and Harrington~\cite{H2007} have also established sufficient conditions for subelliptic estimates in terms of the existence of functions with large Hessians near the boundary. 

We will focus our attention on convex domains. For smoothly bounded convex domains, subelliptic estimates have been previously studied by Forn\ae ss-Sibony \cite{FS1989} and McNeal~\cite{M1994,M2002,NPT2013}. For bounded convex domains with non-smooth boundary, Fu-Straube~\cite{FS1998} established necessary and sufficient conditions for compactness of the $\bar{\partial}$-Neumann problem. Convexity is a strong geometric assumption, but we will show that this special case already contains interesting examples with non-smooth boundary. 

In the non-smooth setting, it seems difficult to develop boundary invariants that will imply or be implied by subelliptic estimates. Instead, we consider conditions on the interior geometry of a domain. In particular, every bounded pseudoconvex domain $\Omega$ has a canonical geometry: the complete K\"ahler-Einstein metric $g_\Omega$ with Ricci curvature $-1$ constructed by Cheng-Yau~\cite{CY1980} when $\partial \Omega$ is $C^2$ and Mok-Yau~\cite{MY1983} in general. Let $d_\Omega$ denote the distance on $\Omega$ induced by this K\"ahler metric. In~\cite{Z2016}, we proved that when $\Omega$ is a smoothly bounded convex domain, then $\partial \Omega$ has finite type if and only if the metric space $(\Omega, d_\Omega)$ is Gromov hyperbolic. 

Combining this with Catlin's results yields the following: when $\Omega$ is a smoothly bounded convex domain a subelliptic estimate holds if and only if $(\Omega, d_\Omega)$ is Gromov hyperbolic. The first main result of this paper shows that one direction of the above equivalence holds without any boundary regularity.

\begin{theorem}\label{thm:main} Suppose $\Omega \subset \Cb^d$ is a bounded convex domain and $(\Omega, d_\Omega)$ is Gromov hyperbolic. Then $\Omega$ satisfies a subelliptic estimate.
\end{theorem}

\begin{remark}\label{rmk:main_thm} \ \begin{enumerate}
\item Unfortunately the converse to Theorem~\ref{thm:main} is false, see Section~\ref{sec:converse_example}. 
\item A bounded convex domain has (at least) two other natural metrics: the Kobayashi metric and the Bergman metric. By a result of Frankel~\cite{F1991} these are both bi-Lipschitz to the K\"ahler-Einstein metric and hence if one is Gromov hyperbolic, then they all are.
\end{enumerate}
\end{remark}

We prove Theorem~\ref{thm:main} by constructing a bounded plurisubharmonic function whose Levi form grows at a certain rate. Such functions imply subelliptic estimates by results of Catlin~\cite{C1987} and Straube~\cite{S1997}. Catlin's construction of these functions in the finite type case is very involved and so finding alternative approaches to constructing these functions seems desirable. The main idea in our construction is to show that analytic and algebraic arguments of McNeal~\cite{M1992, M1994} in the case of convex domains of finite type can be recast as geometric arguments in terms of the intrinsic metrics. Another key part in our construction is proving that convex domains whose K\"ahler-Einstein metric is Gromov hyperbolic must be $m$-convex, see Section~\ref{sec:m_convex_vs_gromov}. An outline of the proof of Theorem~\ref{thm:main} and a more detailed discussion of prior work can be found in Section~\ref{sec:prior_outline_subelliptic}.

One motivation for Theorem~\ref{thm:main} comes from the deep connections between potential theory and negative curvature, see for instance \cite{A1983,S1983,AS1985,A1987,A1990}. In particular, techniques from Gromov hyperbolic metric spaces have been used to develop new insights into potential theory on bounded domains in $\Rb^d$, see for instance~\cite[Section 8]{A1987}. Based on these results, it seems natural to explore connections between other analytic problems and Gromov hyperbolicity. 

Theorem~\ref{thm:main} is a consequence of the following more general result. 

\begin{theorem}\label{thm:intersection}(see Section~\ref{sec:pf_of_thm_intersection}) Suppose $\Omega_1,\dots, \Omega_m \subset \Cb^d$ are bounded convex domains and each $(\Omega_j, d_{\Omega_j})$ is Gromov hyperbolic. If $\Omega := \cap_{j=1}^m \Omega_j$ is non-empty, then $\Omega$ satisfies a subelliptic estimate.  \end{theorem}

Our second main result is a necessary and sufficient condition for $(\Omega, d_\Omega)$ to be Gromov hyperbolic. To state the precise result, we need the following definitions. 

\begin{definition} \ 
\begin{enumerate}
\item A domain $\Omega \subset \Cb^d$ has \emph{simple boundary} if every holomorphic map $\Db \rightarrow \partial \Omega$ is constant. 
\item A convex domain $\Omega \subset \Cb^d$ is called \emph{$\Cb$-properly convex} if $\Omega$ does not contain any entire complex affine lines. 
\item Let $\Xb_d$ denote the set of all $\Cb$-properly convex domains in $\Cb^d$ endowed with the local Hausdorff topology (see Section~\ref{sec:space_of_convex_domains} for details). 
\item Let $\Aff(\Cb^d)$ denote the group of complex affine automorphisms of $\Cb^d$. 
\end{enumerate}
\end{definition}

The group $\Aff(\Cb^d)$ acts on $\Xb_d$ and our characterization of Gromov hyperbolicity is in terms of the orbit of a domain under this action.

\begin{theorem}\label{thm:main_equivalence}(see Section~\ref{sec:pf_of_thm_main_equivalence})  Suppose $\Omega \subset \Cb^d$ is a bounded convex domain. Then $(\Omega, d_\Omega)$ is Gromov hyperbolic if and only if every domain in
\begin{align*}
\overline{\Aff(\Cb^d) \cdot \Omega} \cap \Xb_d
\end{align*}
has simple boundary.
\end{theorem}

\begin{remark} Theorem~\ref{thm:main_equivalence} is motivated by results of Karlsson-Noskov~\cite{KN2002} and Benoist~\cite{B2009} on the Hilbert metric, see Section~\ref{sec:intro_to_hilbert_metric} for details. \end{remark}

Theorem~\ref{thm:main_equivalence} may seem like a very abstract characterization, but in many concrete cases one can use it to quickly determine if $(\Omega, d_\Omega)$ is Gromov hyperbolic or not. For instance, suppose $\Omega \subset\Cb^d$ is a bounded convex domain with $C^\infty$ boundary. If $\partial \Omega$ has finite type in the sense of D'Angelo, then the rescaling method of Bedford-Pinchuk~\cite{BP1994} implies that every domain in $\overline{\Aff(\Cb^d) \cdot \Omega} \cap \Xb_d$ coincides, up to an affine transformation, either with $\Omega$ or a domain of the form 
\begin{align*}
\{ z \in \Cb^d : {\rm Im}(z_1) > P(z_2,\dots, z_d)\}
\end{align*}
where $P$ is a ``non-degenerate'' real valued polynomial. This implies that every domain in $\overline{\Aff(\Cb^d) \cdot \Omega} \cap \Xb_d$ has simple boundary. Conversely, if $\Omega$ has a point $\xi \in \partial \Omega$ with infinite type in the sense of D'Angelo, then there exists a sequence of affine maps $A_n$ such that $A_n(\xi) = \xi$ and $A_n\Omega$ converges to a $\Cb$-properly convex domain whose boundary contains an analytic disk through $\xi$, see~\cite[Lemma 6.1]{Z2016}. This discussion implies the following corollary. 

\begin{corollary}\cite[Theorem 1.1]{Z2016}\label{cor:finite_type} Suppose $\Omega \subset \Cb^d$ is a bounded convex domain with $C^\infty$ boundary. Then $(\Omega, d_\Omega)$ is Gromov hyperbolic if and only if $\partial \Omega$ has finite type in the sense of D'Angelo. 
\end{corollary}

Using Theorem~\ref{thm:main} and Theorem~\ref{thm:main_equivalence}, we can construct examples of domains which satisfy a subelliptic estimate and have interesting boundaries. 

\begin{example}\label{ex:local_cone_point}(see Section~\ref{sec:ex_local_cone_point}) For any $d \geq 2$, there exists a bounded convex domain $\Omega \subset \Cb^d$ with the following properties:
\begin{enumerate}
\item there exists a boundary point $\xi \in \partial \Omega$ where $\Omega$ is locally a cone (that is, there exists a convex cone $\Cc \subset \Cb^d$ based at $\xi$ and a neighborhood $U$ of $\xi$ such that $\Cc \cap U = \Omega \cap U$) and
\item a subelliptic estimate holds on $\Omega$.
\end{enumerate}
\end{example}

\begin{example}\label{ex:non_strongly_pseudoconvex}(see Section~\ref{sec:squeezing}) For any $d \geq 2$, there exists a bounded convex domain $\Omega \subset \Cb^d$ with the following properties:
\begin{enumerate}
\item $\partial \Omega$ is $C^2$,
\item $\Omega$ is not strongly pseudoconvex, and
\item a subelliptic estimate of order $\epsilon$ holds on $\Omega$ for every $\epsilon \in (0,1/2)$. 
\end{enumerate}
\end{example}

\begin{example}\label{ex:ae_flat}(see Section~\ref{sec:ex_ae_flat}) For any $d \geq 2$ there exists a bounded convex domain $\Omega \subset \Cb^d$ with the following properties: 
\begin{enumerate}
\item $\partial \Omega$ is $C^{1,\alpha}$ for some $\alpha > 0$ (but not $C^{1,1}$),
\item the curvature of $\partial \Omega$ is concentrated on a set of measure zero (see Definition~\ref{defn:curvature_on_meas_zero}), and
\item a subelliptic estimate holds on $\Omega$.
\end{enumerate}
Informally, Condition (2) says that $\partial\Omega$ is strongly convex on a set of measure zero. 
\end{example}

We can also use Theorem~\ref{thm:main_equivalence} to relate the geometry of the classical Hilbert metric to the geometry of the K\"ahler-Einstein metric. This relationship will be one of our primary mechanisms for constructing interesting examples.

A convex domain $C \subset \Rb^d$ is called \emph{$\Rb$-properly convex} if it does not contain an entire affine real line. Every $\Rb$-properly convex domain $C \subset \Rb^d$ has a natural interior geometry: the Hilbert distance which we denote by $H_{C}$. Recently, Benoist~\cite{B2003} proved that the Hilbert distance on a bounded convex domain is Gromov hyperbolic  if and only if the boundary of the domain is quasi-symmetric (see Definition~\ref{defn:quasi_symmetric}). 

Using Theorem~\ref{thm:main_equivalence} and work of Karlsson-Noskov~\cite{KN2002} on the Hilbert metric we will establish the following. 

\begin{corollary}\label{cor:Hilbert_metric}(see Section~\ref{sec:pf_of_cor_hilbert_metric}) Suppose $\Omega \subset \Cb^d$ is a bounded convex domain. If $(\Omega, H_\Omega)$ is Gromov hyperbolic, then $(\Omega, d_\Omega)$ is Gromov hyperbolic.
\end{corollary}

Corollary~\ref{cor:Hilbert_metric} is somewhat surprising since the metric spaces $(\Omega, H_\Omega)$ and $(\Omega, d_\Omega)$ can be very different. For instance, if $D \subset \Cb$ is a convex polygon, then $(D,d_D)$ is isometric to the real hyperbolic plane, while $(D,H_D)$ is quasi-isometric to the Euclidean plane~\cite{B2009} (notice that this shows that the converse of Corollary~\ref{cor:Hilbert_metric} is false).

Using Corollary~\ref{cor:Hilbert_metric} and Benoist's characterization of Gromov hyperbolicity for the Hilbert distance, we have the following example.

\begin{example} Suppose $\Omega \subset \Cb^d$ is a bounded convex domain with quasi-symmetric boundary (see Definition~\ref{defn:quasi_symmetric}). Then $(\Omega, d_\Omega)$ is Gromov hyperbolic and hence a subelliptic estimate holds on $\Omega$. 
\end{example}

We can also use the proof of Theorem~\ref{thm:main_equivalence} to characterize the tube domains where the K\"ahler-Einstein metric is Gromov hyperbolic. A domain $\Omega \subset \Cb^d$ is called a \emph{tube domain} if there exists a domain $C \subset \Rb^d$ such that $\Omega = C + i\Rb^d$. Bremermann~\cite{B1957} showed that a tube domain $\Omega = C+i\Rb^d$ is pseudoconvex if and only if $C$ is convex. Further, when $C$ is convex the domain $\Omega = C+i\Rb^d$ is $\Cb$-properly convex if and only if $C$ is $\Rb$-properly convex. Using the proof of Theorem~\ref{thm:main_equivalence} we prove the following. 

\begin{corollary}\label{cor:tube_domains}(see Section~\ref{sec:tube_domains}) Suppose $d \geq 2$, $C \subset \Rb^d$ is an $\Rb$-properly convex domain, and $\Omega = C+i\Rb^d$. Then the following are equivalent:
\begin{enumerate}
\item $(\Omega,d_\Omega)$ is Gromov hyperbolic, 
\item $(C,H_C)$ is Gromov hyperbolic and $C$ is unbounded. 
\end{enumerate}
\end{corollary}

\begin{remark} Pflug and Zwonek previously established some necessary conditions for the K\"ahler-Einstein metric on a tube domain to be Gromov hyperbolic~\cite{PZ2018}. \end{remark}

If $(X,d)$ is a Gromov hyperbolic metric space,  $X$ has a natural compactification, denoted by $\overline{X}^G$, called the \emph{Gromov compactification}. The \emph{Gromov boundary of $X$} is $\partial_G X := \overline{X}^G \setminus X$. See Section~\ref{sec:GH_basics} for a precise definition. 

In joint work with Bracci and Gaussier, we showed when $\Omega$ is convex and $(\Omega, d_\Omega)$ is Gromov hyperbolic, the Gromov compactification coincides with the ``Euclidean end compactification.'' 

\begin{definition}\label{defn:end_comp} Given a domain $\Omega \subset \Cb^d$, let $\overline{\Omega}^{\rm End}$ denote  the end compactification of $\overline{\Omega}$ (in the sense of Freudenthal, see~\cite{P1992}). Then define $\partial_{\rm End}\Omega := \overline{\Omega}^{\rm End} \setminus \Omega$. \end{definition}

\begin{theorem}\cite[Theorem 1.4]{BGZ2018b}\label{thm:compactification} Suppose $\Omega \subset \Cb^d$ is a $\Cb$-properly convex domain and $(\Omega, d_\Omega)$ is Gromov hyperbolic. Then the identity map $\Omega \rightarrow \Omega$ extends to a homeomorphism 
\begin{align*}
\overline{\Omega}^{\rm End} \rightarrow \overline{\Omega}^G.
\end{align*}
 \end{theorem}
 
 \begin{remark} To be precise, Theorem 1.4 in~\cite{BGZ2018b} assumes that the Kobayashi distance $K_\Omega$ is Gromov hyperbolic and shows that $\overline{\Omega}^{\rm End}$ is homeomorphic to the Gromov compactification of $(\Omega, K_\Omega)$. However, as mentioned earlier, the Kobayashi and K\"ahler-Einstein metrics are bi-Lipschitz on any $\Cb$-properly convex domain~\cite{F1991} and the Gromov boundary is a quasi-isometric invariant. 
 \end{remark}

Using Theorem~\ref{thm:compactification} and facts about the geometry of Gromov hyperbolic metric spaces, one can establish the following results about the behavior of holomorphic maps.

\begin{corollary}\cite[Corollary 1.6]{BGZ2018b} Suppose $\Omega_1, \Omega_2 \subset \Cb^d$ are $\Cb$-properly convex domains and $f : \Omega_1 \rightarrow \Omega_2$ is a biholomorphism. If $(\Omega_1, d_{\Omega_1})$ (and hence also $(\Omega_2, d_{\Omega_2})$) is Gromov hyperbolic, then $f$ extends to a homeomorphism $\overline{\Omega_1}^{\rm End} \rightarrow \overline{\Omega_2}^{\rm End}$. \end{corollary}

\begin{corollary}\cite[Corollary 1.7]{BGZ2018b} Suppose $\Omega \subset \Cb^d$ is a $\Cb$-properly convex domain and $(\Omega, d_\Omega)$ is Gromov hyperbolic. If $f : \Omega \rightarrow \Omega$ is holomorphic, then either 
\begin{enumerate}
\item $f$ has a fixed point in $\Omega$, or
\item there exists $\xi \in \partial_{\rm End}\Omega$ such that 
\begin{align*}
\lim_{n \rightarrow \infty} f^n(z) = \xi
\end{align*}
for all $z \in \Omega$.
\end{enumerate}
\end{corollary}

Theorem~\ref{thm:main_equivalence} provides new examples with non-smooth boundary for which these corollaries apply. 

\subsection{Outline of Paper} Throughout the paper we will consider the Kobayashi metric instead of the K\"ahler-Einstein metric. As mentioned in the introduction, Frankel~\cite{F1991} proved that the two metrics are bi-Lipschitz on any $\Cb$-properly convex domain. Hence, if one is Gromov hyperbolic, then so is the other. In the convex setting, the Kobayashi metric is slightly easier to work with because there are very precise estimates, see for instance Lemmas~\ref{lem:kob_inf_bound} and~\ref{lem:hyperplanes} below. However, for general pseudoconvex domains it is not known whether or not the Kobayashi metric is complete, so it seems reasonable to state all the results in the introduction in terms of the K\"ahler-Einstein metric. 

The paper has four main parts: 
\begin{enumerate} 
\item Sections~\ref{sec:background} through ~\ref{sec:normalizing_maps} are mostly expository and devoted to some preliminary material. 
\item Sections~\ref{sec:prior_work_outline_Gromov} through~\ref{sec:pf_of_thm_main_equivalence} are devoted to the proof of Theorem~\ref{thm:main_equivalence}. In Section~\ref{sec:prior_work_outline_Gromov} we recall some prior work and give an outline of the proof of Theorem~\ref{thm:main_equivalence}. 
\item Sections~\ref{sec:prior_outline_subelliptic} through~\ref{sec:optimal_estimate} are devoted to the proof Theorem~\ref{thm:intersection}. In Section~\ref{sec:prior_outline_subelliptic} we recall some prior work and give an outline of the proof of Theorem~\ref{thm:intersection}. 
\item In Sections~\ref{sec:intro_to_hilbert_metric} through~\ref{sec:misc_examples}, we construct a number of examples. 
\end{enumerate}

 \subsection*{Acknowledgements} I would like to thank the referee for their very careful reading of the paper and their very insightful comments. This material is based upon work supported by the National Science Foundation under grants DMS-1760233, DMS-2104381, and DMS-2105580.

\part{Preliminaries}

\section{Background material}\label{sec:background}

\subsection{Notation}

\begin{enumerate}
\item For $z \in \Cb^d$ let $\norm{z}$ be the standard Euclidean norm and $d_{\Euc}(z_1, z_2) = \norm{z_1-z_2}$ be the standard Euclidean distance. 
\item For $z_0 \in \Cb^d$ and $r > 0$ let 
\begin{align*}
\Bb_d(z_0;r) = \left\{ z \in \Cb^d : \norm{z-z_0} < r\right\}.
\end{align*}
Then let $\Bb_d  = \Bb_d(0;1)$ and $\Db = \Bb_1$. 
\item Throughout the paper we will let $\Cb^d \cup \{\infty\}$ denote the one-point compactification of $\Cb^d$. 
\item Given an open set $\Omega \subset \Cb^d$, $z \in \Omega$, and $v \in \Cb^d \setminus \{0\}$ let 
\begin{align*}
\delta_{\Omega}(z)= \inf \{ d_{\Euc}(z,w) : w\in \partial \Omega \}
\end{align*}
and 
\begin{align*}
\delta_{\Omega}(z;v)= \inf \{ d_{\Euc}(z,w) : w\in \partial \Omega \cap (z+\Cb \cdot v) \}.
\end{align*}
\end{enumerate}

\subsection{Gromov hyperbolicity}\label{sec:GH_basics}

In this subsection we recall the definition of Gromov hyperbolic metric spaces and
state some of their basic properties, additional information can be found in~\cite{BH1999} or~\cite{DSU2017}.

Given a metric space $(X,d)$ define the \emph{Gromov product} of $x,y,z \in X$ to be 
\begin{align*}
(x|y)_z = \frac{1}{2}\left( d(x,z) + d(z,y) - d(x,y) \right). 
\end{align*}

\begin{definition}\label{defn:GH} \
\begin{enumerate}
\item A metric space $(X,d)$ is \emph{$\delta$-hyperbolic} if 
\begin{align*}
(x|z)_w \geq \min\{ (x|y)_w, (y|z)_w\} - \delta
\end{align*}
for all $x,y,z,w \in X$. 
\item A metric space is called \emph{Gromov hyperbolic} if it is $\delta$-hyperbolic for some $\delta \geq 0$. 
\end{enumerate}
\end{definition}

For proper geodesic metric spaces, Gromov hyperbolicity can also be defined in terms of the shape of geodesic triangles.  

When $(X,d)$ is a metric space and $I \subset\Rb$ is an interval, a curve $\sigma: I \rightarrow X$ is a \emph{geodesic} if 
\begin{align*}
d(\sigma(t_1),\sigma(t_2)) = \abs{t_1-t_2}
\end{align*}
 for all $t_1, t_2 \in I$.  We say that $(X,d)$ is \emph{geodesic} if every two points in $X$ can be joined by a geodesic and \emph{proper} if bounded closed sets are compact. 

A \emph{geodesic triangle} in a metric space is a choice of three (not necessarily distinct) points in $X$ and geodesic segments connecting these points. A geodesic triangle is said to be \emph{$\delta$-thin} if any point on any of the sides of the triangle is within distance $\delta$ of the union of the other two sides. 

\begin{theorem}\label{thm:thin_triangle} If $(X,d)$ is a proper geodesic metric space, then $(X,d)$ is Gromov hyperbolic if and only if there exists some $\delta \geq 0$ such that every geodesic triangle is $\delta$-thin. 
\end{theorem}

\begin{proof} See for instance~\cite[Chapter III.H, Proposition 1.22]{BH1999}. \end{proof}

A proper geodesic Gromov hyperbolic metric space $(X,d)$ also has a natural boundary which can be described as follows. Two geodesic rays $\sigma_1,\sigma_2:[0,\infty) \rightarrow X$ are  \emph{asymptotic} if 
\begin{align*}
\sup_{t \geq 0} d(\sigma_1(t),\sigma_2(t)) < \infty.
\end{align*}
Then the \emph{Gromov boundary}, denoted by $\partial_G X$, is the set of equivalence classes of asymptotic geodesic rays in $X$. 

The set $\overline{X}^G = X \cup \partial_G X$ has a natural topology making it a compactification of $X$ (see for instance~\cite[Chapter III.H.3]{BH1999}). To understand this topology we introduce the following notation: given a geodesic ray $\sigma:[0,\infty) \rightarrow X$ let $[\sigma]$ denote the equivalence class of $\sigma$ and given a geodesic segment $\sigma:[0,R] \rightarrow X$ define $[\sigma]:=\sigma(R)$. Now fix a point $x_0 \in X$, then the topology on $\overline{X}^G$ can be described as follows: $\xi_n \rightarrow \xi$ if and only if for every choice of geodesics $\sigma_n$ with $\sigma_n(0)=x_0$ and $[\sigma_n]=\xi_n$ every subsequence of $(\sigma_n)_{n \geq 0}$ has a subsequence which converges locally uniformly to a geodesic $\sigma$ with $[\sigma]=\xi$. One can also show that this topology does not depend on the choice of $x_0$ (again see ~\cite[Chapter III.H.3]{BH1999}).

\begin{remark} In some special cases, for instance when $X$ is simply connected complete negatively curved Riemannian manifold, for every $\xi \in \overline{X}^G$ there exists a unique geodesic $\sigma_\xi$ with $\sigma_{\xi}(0)=x_0$ and $[\sigma_{\xi}]=\xi$. In this case, $\xi_n \rightarrow \xi$ if and only the geodesics $\sigma_{\xi_n}$ converge locally uniformly to $\sigma_{\xi}$. 
\end{remark}

Next we recall the Morse Lemma for quasi-geodesics. 

\begin{definition} Suppose $(X,d)$ is a metric space, $I \subset \Rb$ is an interval, $\alpha \geq 1$, and $\beta \geq 0$. Then a map $\sigma : I \rightarrow X$ is a \emph{$(\alpha,\beta)$-quasi-geodesic} if
\begin{align*}
\frac{1}{\alpha} \abs{t-s}-\beta \leq d(\sigma(s), \sigma(t)) \leq \alpha \abs{t-s}+\beta
\end{align*}
for all $s,t \in I$. 
\end{definition}

Quasi-geodesics in a Gromov hyperbolic metric space have the following remarkable property. 

\begin{theorem}[Morse Lemma]\label{thm:morse_lemma} For any $\delta>0$, $\alpha \geq 1$, and $\beta \geq 0$ there exists $M =M(\delta,\alpha,\beta)> 0$ with the following property: if $(X,d)$ is a proper geodesic $\delta$-hyperbolic metric space and $\sigma_1:[a_1,b_1] \rightarrow X$, $\sigma_2:[a_2,b_2] \rightarrow X$ are $(\alpha,\beta)$-quasi-geodesics with $\sigma_1(a_1)=\sigma_2(a_2)$, $\sigma_1(b_1)=\sigma_2(b_2)$, then 
\begin{align*}
\max\left\{ \max_{t \in [a_1,b_1]} d(\sigma_1(t), \sigma_2), \max_{t \in [a_2,b_2]} d(\sigma_2(t), \sigma_1) \right\} \leq M.
\end{align*}
\end{theorem}

\begin{proof} For a proof see for instance~\cite[Chapter III.H, Theorem 1.7]{BH1999}. \end{proof}

\subsection{The Kobayashi metric}

 In this expository section we recall the definition of the Kobayashi metric and then state some of its properties.  
  
 Given a domain $\Omega \subset \Cb^d$ the \emph{(infinitesimal) Kobayashi metric} is the pseudo-Finsler metric
\begin{align*}
k_{\Omega}(x;v) = \inf \left\{ \abs{\xi} : f \in \Hol(\Db, \Omega), \ f(0) = x, \ d(f)_0(\xi) = v \right\}.
\end{align*}
By a result of Royden~\cite[Proposition 3]{R1971} the Kobayashi metric is an upper semicontinuous function on $\Omega \times \Cb^d$. In particular, if $\sigma:[a,b] \rightarrow \Omega$ is an absolutely continuous curve (as a map $[a,b] \rightarrow \Cb^d$), then the function 
\begin{align*}
t \in [a,b] \rightarrow k_\Omega(\sigma(t); \sigma^\prime(t))
\end{align*}
is integrable and we can define the \emph{length of $\sigma$} to  be
\begin{align*}
\ell_\Omega(\sigma)= \int_a^b k_\Omega(\sigma(t); \sigma^\prime(t)) dt.
\end{align*}
One can then define the \emph{Kobayashi pseudo-distance} to be
\begin{multline*}
 K_\Omega(x,y) = \inf \left\{\ell_\Omega(\sigma) : \sigma\colon[a,b]
 \rightarrow \Omega \text{ is abs. cont., } \sigma(a)=x, \text{ and } \sigma(b)=y\right\}.
\end{multline*}
This definition is equivalent to the standard definition using analytic chains by a result of Venturini~\cite[Theorem 3.1]{V1989}.

When $\Omega \subset\Cb^d$ is bounded, it is easy to show that $K_\Omega$ is a non-degenerate distance on $\Omega$. For general domains determining whether or not $K_\Omega$ is non-degenerate is very difficult, but in the special case of convex domains we have the following result of Barth. 

\begin{theorem}[Barth~\cite{B1980}]\label{thm:barth}
Suppose $\Omega$ is a convex domain. Then the following are equivalent:
\begin{enumerate}
\item $\Omega$ is $\Cb$-proper, 
\item $\Omega$ is biholomorphic to a bounded domain, 
\item $K_\Omega$ is a non-degenerate distance on $\Omega$, 
\item $(\Omega, K_\Omega)$ is a proper geodesic metric space. 
\end{enumerate}
\end{theorem}

Since every $\Cb$-properly convex domain is biholomorphic to a bounded domain, the results of Cheng-Yau~\cite{CY1980} and Mok-Yau~\cite{MY1983} imply that every such domain has a unique complete K\"ahler-Einstein metric with Ricci curvature $-1$. 

\begin{definition} When $\Omega \subset \Cb^d$ is a $\Cb$-properly convex domain, let $g_\Omega$ be the unique complete K\"ahler-Einstein metric on $\Omega$ with Ricci curvature $-1$ and let $d_\Omega$ be the associated distance. 
\end{definition}

As mentioned in Remark~\ref{rmk:main_thm}, we have the following uniform relationship between the Kobayashi and K\"ahler-Einstein metrics. 

\begin{theorem}[Frankel~\cite{F1991}] For any $d \in\Nb$, there exists $A > 1$ such that: if $\Omega \subset \Cb^d$ is a $\Cb$-properly convex domain, then
\begin{align*}
\frac{1}{A} k_\Omega(z;v) \leq \sqrt{g_\Omega(v,v)} \leq A k_\Omega(z;v)
\end{align*}
for all $z \in \Omega$ and $v \in T_z \Omega$. 
\end{theorem}

We will also use the following standard estimates on the the Kobayashi distance and metric.

\begin{lemma}[Graham \cite{G1991}]\label{lem:kob_inf_bound}
Suppose $\Omega \subset \Cb^d$ is a convex domain. If $z \in \Omega$ and $v \in \Cb^d$ is non-zero, then 
\begin{align*}
\frac{\norm{v}}{2\delta_{\Omega}(z;v)} \leq k_{\Omega}(z;v) \leq \frac{\norm{v}}{\delta_{\Omega}(z;v)}.
\end{align*}
\end{lemma}

A proof of Lemma~\ref{lem:kob_inf_bound} can also be found in \cite[Theorem 2.2]{F1991}.

\begin{lemma}\label{lem:hyperplanes}
Suppose $\Omega \subset \Cb^d$ is a convex domain and $H \subset \Cb^d$ is a complex hyperplane such that $H \cap \Omega = \emptyset$. Then for any $z_1, z_2 \in \Omega$ we have 
\begin{align*}
K_{\Omega}(z_1, z_2) \geq \frac{1}{2}\abs{ \log \frac{ d_{\Euc}(H,z_1)}{d_{\Euc}(H,z_2)} }.
\end{align*}
\end{lemma}

A proof of Lemma~\ref{lem:hyperplanes} can be found in \cite[Lemma 4.2]{Z2017b}. 

\begin{lemma}\label{lem:hyperplanes_along_lines}
Suppose $\Omega \subset \Cb^d$ is a convex domain, $z_1,z_2 \in \Omega$, and $L$ is the complex affine line containing $z_1,z_2$.  Then 
\begin{align*}
K_{\Omega}(z_1, z_2) \geq \sup_{\xi \in L \setminus \Omega \cap L} \frac{1}{2} \abs{\log \frac{ \norm{z_1-\xi}}{\norm{z_2-\xi}}}.
\end{align*}
\end{lemma}

A proof of Lemma~\ref{lem:hyperplanes_along_lines} can be found in \cite[Lemma 2.6]{Z2016}, but it also follows easily from Lemma~\ref{lem:hyperplanes}.

Using Lemma~\ref{lem:kob_inf_bound} and Lemma~\ref{lem:hyperplanes} it is possible to prove the following. 

\begin{proposition}[{\cite[Theorem 3.1]{Z2016}}]\label{prop:quasi_geodesic} Suppose $\Omega \subset \Cb^d$ is a $\Cb$-properly convex domain. For any $z_0 \in \Omega$ and $R > 0$, there exist $\alpha \geq 1$ and $\beta \geq 0$ such that: if $\xi \in \partial \Omega \cap \Bb_d(0;R)$, then the curve $\sigma_\xi: [0,\infty) \rightarrow \Omega$ given by 
\begin{align*}
\sigma_\xi(t) = \xi + e^{-2t}\left( z_0-\xi \right)
\end{align*}
is an $(\alpha,\beta)$-quasi-geodesic. 
\end{proposition}

\subsection{Geometric properties of convex domains} In this section we recall some basic geometric properties of convex domains. 

First, we have the following result about the complex geometry of the boundary of a convex domain. 

\begin{proposition}\label{prop:affine_vs_holomorphic_disks} Suppose $\Omega \subset \Cb^d$ is a convex domain. Then every holomorphic map $\Db \rightarrow \partial\Omega$ is constant if and only if every complex affine map $\Db \rightarrow \partial\Omega$ is constant. 
\end{proposition}

\begin{proof} See for instance~\cite[Theorem 1.1]{FS1998}. \end{proof}

We will also use the following observation about the asymptotic geometry of a convex domain. 

\begin{observation}\label{obs:asymptotic_cone_1} Suppose $\Omega \subset \Cb^d$ is a convex domain and $v \in \Cb^d$ is non-zero. Then the following are equivalent:
\begin{enumerate}
\item there exists a sequence $z_n \in \Omega$ such that $\norm{z_n} \rightarrow \infty$ and 
\begin{align*}
\lim_{n \rightarrow \infty} \frac{z_n}{\norm{z_n}} = \frac{v}{\norm{v}},
\end{align*}
\item $z_0 + \Rb_{\geq 0} v \subset \Omega$ for some $z_0 \in \Omega$,
\item $z + \Rb_{\geq 0} v \subset \Omega$ for all $z \in \Omega$. 
\end{enumerate}
\end{observation}

\begin{proof} Clearly $(3) \Rightarrow (2) \Rightarrow (1)$. To prove $(1) \Rightarrow (3)$: suppose that $z_n \in \Omega$, $\norm{z_n} \rightarrow \infty$, and 
\begin{align*}
\lim_{n \rightarrow \infty} \frac{z_n}{\norm{z_n}} = \frac{v}{\norm{v}}.
\end{align*}
Fix some $z \in \Omega$. Then by convexity $[z,z_n] \subset \Omega$ for every $n \in \Nb$. So $z + \Rb_{\geq 0} v \subset \overline{\Omega}$. Then since $\Omega$ is open and convex, we see that $z + \Rb_{\geq 0} v \subset \Omega$.
\end{proof}

This observation motivates the following standard definition. 

\begin{definition}\label{defn:asym_cone} Suppose $\Omega \subset \Cb^d$ is a convex domain. The \emph{asymptotic cone of $\Omega$}, denoted by ${ \rm AC}(\Omega)$, is the set of  vectors $v \in \Cb^d$ such that $z + \Rb_{\geq 0} v \subset \Omega$ for some (hence all) $z \in \Omega$.
\end{definition}

As the name suggests we have the following. 

\begin{observation} Suppose $\Omega \subset \Cb^d$ is a convex domain. Then ${ \rm AC}(\Omega)$ is a convex cone based at $0$.
\end{observation}

\begin{proof}
This is an immediate consequence of convexity. 
\end{proof}

Finally, we have the following connection between the asymptotic cone and the end compactification (recall, from Definition~\ref{defn:end_comp}, that $\overline{\Omega}^{\End}$ denotes the end compactification of $\overline{\Omega}$). 

\begin{observation}\label{obs:asymptotic_cone_3}
Suppose $\Omega \subset \Cb^d$ is a convex domain. Then either
\begin{enumerate}
\item $\Omega$ is bounded and $\overline{\Omega}^{\End} = \overline{\Omega}$,
\item $\overline{\Omega}^{\End} \setminus \overline{\Omega}$ is a single point, or
\item $\overline{\Omega}^{\End} \setminus \overline{\Omega}$ is two points and ${ \rm AC}(\Omega)= \Rb \cdot v$ for some non-zero $v \in \Cb^d$. 
\end{enumerate}
\end{observation}

\begin{proof}
This is an immediate consequence of Observation~\ref{obs:asymptotic_cone_1}. 
\end{proof}

\section{The space of convex domains}\label{sec:space_of_convex_domains}

Following work of Frankel~\cite{F1989b, F1991}, in this section we describe some facts about the space of convex domains and the action of the affine group on this space. 

\begin{definition} Let $\Xb_d$ be the set of all non-empty $\Cb$-properly convex domains in $\Cb^d$ and let $\Xb_{d,0}$ be the set of pairs $(\Omega,z)$ where $\Omega \in \Xb_d$ and $z \in \Omega$. \end{definition}

\begin{remark} The motivation for only considering $\Cb$-properly convex domains comes from Theorem~\ref{thm:barth}.
\end{remark}

We now describe a natural topology on the sets $\Xb_d$ and $\Xb_{d,0}$. Given two non-empty compact sets $A,B \subset \Cb^d$, the \emph{Hausdorff distance} between them is
\begin{align*}
d_{H}(A,B) = \max\left\{ \max_{a \in A} \min_{b \in B} \norm{a-b}, \max_{b \in B} \min_{a \in A} \norm{b-a} \right\}.
\end{align*}
We also define 
\begin{align*}
d_{H}(A,\emptyset) = d_{H}(\emptyset,A) = \begin{cases} \infty & \text{if } A \neq \emptyset \\ 0 & \text{if } A = \emptyset \end{cases}.
\end{align*}

The Hausdorff distance is a complete metric on the set of non-empty compact subsets in $\Cb^d$. To consider general closed sets, we introduce the \emph{local Hausdorff pseudo-distances} between two non-empty closed sets $A,B \subset \Cb^d$ by defining 
\begin{align*}
d_{H}^{(R)}(A,B) = d_{H}\left(A \cap \overline{\Bb_d(0;R)},B\cap \overline{\Bb_d(0;R)}\right)
\end{align*}
for $R > 0$. Since an open convex set is determined by its closure, we can define a topology on $\Xb_{d}$ and $\Xb_{d,0}$ using these pseudo-distances:
\begin{enumerate}
\item A sequence $\Omega_n \in \Xb_d$ converges to $\Omega \in \Xb_d$ if there exists some $R_0 \geq 0$ such that $d_{H}^{(R)}(\overline{\Omega}_n,\overline{\Omega}) \rightarrow 0$ for all $R \geq R_0$, 
 \item A sequence $(\Omega_n,z_n) \in \Xb_{d,0}$ converges to $(\Omega,z) \in \Xb_{d,0}$ if $\Omega_n$ converges to $\Omega$ in $\Xb_d$ and $z_n$ converges to $z$ in $\Cb^d$. 
 \end{enumerate}
 
 We will frequently use the following basic properties of this notion of convergence. 
 
 \begin{proposition}\label{prop:haus_conv_compacts} Suppose that $\Omega_n$ converges to $\Omega$ in $\Xb_d$. 
 \begin{enumerate}
\item For any compact set $K \subset \Omega$, there exists some $N \geq 0$ such that $K \subset \Omega_n$ for all $n \geq N$. 
\item If $z_n \in \overline{\Omega}_n$ and $\lim_{n \rightarrow \infty} z_n = z$, then $z \in \overline{\Omega}$. 
\item If $z_n \in \Cb^d \setminus \Omega_n$ and $\lim_{n \rightarrow \infty} z_n = z$, then $z \in \Cb^d \setminus \Omega$. 
\end{enumerate}
\end{proposition}

\begin{proof} A proof Part (1) can be found in~\cite[Lemma 4.4]{Z2016}. Parts (2) and (3) follow immediately from the definition.
\end{proof}
 
 The Kobayashi distance also behaves as one would hope under this notion of convergence. 
 
 \begin{proposition}\label{prop:convergence_of_kob} Suppose that a sequence $\Omega_n$ converges to $\Omega$ in $\Xb_d$. Then 
 \begin{align*}
 \lim_{n \rightarrow \infty} K_{\Omega_n} = K_\Omega
 \end{align*}
 and the convergence is uniform on compact subsets of $\Omega \times \Omega$.
 \end{proposition}
 
\begin{proof} See for instance~\cite[Theorem 4.1]{Z2016}. \end{proof}

We will frequently use the following observation. 

\begin{observation}\label{obs:AA_for_geod}  Suppose $\Omega_n$ converges to $\Omega$ in $\Xb_d$ and $\sigma_n :[0,T_n] \rightarrow \Omega_n$ is a sequence of geodesics where 
\begin{align*}
\lim_{n \rightarrow\infty} \sigma_n(0) = z_0 \in \Omega
\end{align*}
and $T = \lim_{n \rightarrow \infty} T_n \in [0,\infty]$. Then there exists a subsequence $\sigma_{n_j}$ which converges locally uniformly to a geodesic $\sigma : [0,T]\cap \Rb \rightarrow \Omega$. In particular, if $T < \infty$, then 
$$
\lim_{n \rightarrow \infty} \sigma_{n_j}(T_{n_j}) = \sigma(T) \in \Omega. 
$$ 
\end{observation} 

\begin{proof} Fix $ R>0$ and let $B = \{ z \in \Omega : K_\Omega(z,z_0)\leq R+1\}$. Then $B$ is compact and so Proposition~\ref{prop:haus_conv_compacts} implies that $B \subset \Omega_n$ for $n$ sufficiently large. Further, Proposition~\ref{prop:convergence_of_kob} implies that $K_{\Omega_n}(\sigma_n(0), \partial B) > R$ for $n$ sufficiently large. So 
$$
\sigma_n( [0,T_n] \cap [0,R]) \subset B
$$ 
for $n$ sufficiently large. Then Proposition~\ref{prop:convergence_of_kob} and the Arzel\`a-Ascoli theorem imply that $\sigma_n|_{[0,T_n] \cap [0,R]}$ has a convergent subsequence and the limit is a geodesic in $\Omega$. 

Since $R>0$ was arbitrary, there exists a subsequence $\sigma_{n_j}$ which converges locally uniformly to a geodesic $\sigma : [0,T] \cap \Rb \rightarrow \Omega$.
\end{proof}

Next let $\Aff(\Cb^d)$ be the group of complex affine isomorphisms of $\Cb^d$. Then $\Aff(\Cb^d)$ acts on $\Xb_d$ and $\Xb_{d,0}$. Remarkably, the action of $\Aff(\Cb^d)$ on $\Xb_{d,0}$ is co-compact.

\begin{theorem}[Frankel \cite{F1991}]\label{thm:frankel_compactness}
The group $\Aff(\Cb^d)$ acts co-compactly on $\Xb_{d,0}$, that is there exists a compact set $K \subset \Xb_{d,0}$ such that $\Aff(\Cb^d) \cdot K = \Xb_{d,0}$. 
\end{theorem}

Suppose $\Omega \subset \Cb^d$ is a $\Cb$-properly convex domain and $z_n \in \Omega$ is a sequence. Then Theorem~\ref{thm:frankel_compactness} implies that there exists a sequence of affine maps $A_n \in \Aff(\Cb^d)$ such that 
\begin{align*}
\left\{ A_n(\Omega,z_n) : n \in \Nb\right\}
\end{align*}
 is relatively compact in $\Xb_{d,0}$. So there exist $n_j \rightarrow \infty$ such that $A_{n_j}(\Omega,z_{n_j})$ converges to some $(U,u)$ in $\Xb_{d,0}$. The next result shows that the domain $U$ only depends on the choice of $z_{n_j}$. 

\begin{proposition}\label{prop:limit_domain_depends_on_sequence}
Suppose $(\Omega_n,z_n) \in \Xb_{d,0}$, $A_n \in \Aff(\Cb^d)$, and $B_n \in \Aff(\Cb^d)$ are such that 
\begin{align*}
\lim_{n \rightarrow \infty} A_n(\Omega_n, z_n) =\left(U_1, u_1\right) \text{ and } \lim_{n \rightarrow \infty} B_n(\Omega_n, z_n) =\left(U_2, u_2\right)
\end{align*}
in $\Xb_{d,0}$. Then there exist $n_j \rightarrow \infty$ such that 
\begin{align*}
B_{n_j} A_{n_j}^{-1} 
\end{align*}
converges to some  $T \in \Aff(\Cb^d)$ and 
\begin{align*}
T \left(U_1, u_1\right) = \left(U_2, u_2\right).
\end{align*}
\end{proposition}

\begin{proof} The map $T_n = B_n A_n^{-1} : \Cb^d \rightarrow \Cb^d$ induces an isometry 
\begin{align*}
(A_n \Omega_n, K_{A_n \Omega}) \rightarrow (B_n \Omega_n, K_{B_n \Omega})
\end{align*}
with $T_n(A_nz_n) = B_n z_n$. Then by Proposition~\ref{prop:convergence_of_kob} and the Arzel\`a-Ascoli theorem (see the proof of Observation~\ref{obs:AA_for_geod}), we can pass to a subsequence so that $T_n$ converges locally uniformly to an isometry 
\begin{align*}
T : \left(U_1, K_{U_1} \right) \rightarrow \left(U_2, K_{U_2} \right)
\end{align*}
with $T(u_1) = u_2$. Then $T$, being a limit of affine maps of $\Cb^d$, is affine. Since $T$ is an isometry, it is a bijection $U_1 \rightarrow U_2$. Then since $T$ is injective on $U_1$, we have $T \in \Aff(\Cb^d)$ and since $T$ is onto we have $T(U_1,u_1)=(U_2,u_2)$.  \end{proof}

\section{Normalizing maps}\label{sec:normalizing_maps}

The main result of this section is Theorem~\ref{thm:normalizing} where we construct affine maps which ``normalize'' the following data: a $\Cb$-properly convex domain $\Omega$ and some $z_0 \in \Omega$, $\xi \in \partial \Omega$, $q \in [z_0,\xi)$. The results in this section are refinements of various arguments in~\cite{F1989b, F1991}.

\begin{definition} For $r \in (0,1]$ let $\Kb_d(r) \subset \Xb_d$ denote the set of $\Cb$-properly convex domains $\Omega \subset \Cb^d$ where 
\begin{enumerate}
\item $r\Db \cdot e_1 \subset \Omega$ and $\Db \cdot e_j \subset \Omega$ for $j=2,\dots, d$
\item $e_j \in \partial \Omega$ and
\begin{align*}
\left( e_j + \Span_{\Cb}\{e_{j+1}, \dots, e_d\} \right) \cap \Omega =\emptyset
\end{align*}
for $j=1,\dots, d$.
\end{enumerate}
\end{definition}

We first verify that these sets are compact in  $\Xb_d$. 

\begin{proposition}\label{prop:Kdr_compact} For any $r \in (0,1]$, the set $\Kb_d(r)$ is compact in $\Xb_d$. \end{proposition}

\begin{proof}
Suppose $\Omega_n$ is a sequence in $\Kb_d(r)$. For each $R>0$, the set 
 \begin{align*}
 \left\{ K \subset \overline{\Bb_d(0;R)} : K \text{ is compact}\right\}
\end{align*}
is compact in the Hausdorff topology, see for instance~\cite[Proposition 3.6, Theorem 4.2]{M1950}. So we can find nested subsequences 
\begin{align*}
(n_{1,j})_{j=1}^{\infty} \supset (n_{2,j})_{j=1}^{\infty} \supset \dots
 \end{align*}
 such that 
 \begin{align*}
 \lim_{j\rightarrow \infty} \overline{\Omega}_{n_{m,j}} \cap \overline{\Bb_d(0;m)} = C_m
 \end{align*}
 where $C_m$ is a closed convex domain. Since the sequences are nested,
 $$
 C_1 \subset C_2 \subset \dots 
 $$
 So $C : = \cup_{m = 1}^{\infty} C_m$ is convex and $C_m = C \cap \overline{\Bb_d(0;m)}$ for every $m \geq 1$. 
  
 Let $\Omega_\infty$ denote the interior of $C$. Since 
 \begin{align*}
 {\rm ConvHull} \left\{ r\Db \cdot e_1, \Db \cdot e_2, \dots, \Db \cdot e_d \right\} \subset \Omega_n
 \end{align*}
 for every $n \geq 1$, we see that 
  \begin{align*}
 {\rm ConvHull} \left\{ r\Db \cdot e_1, \Db \cdot e_2, \dots, \Db \cdot e_d \right\} \subset C.
 \end{align*}
 So $C$ has non-empty interior. So $\Omega_\infty$ is non-empty and hence $\overline{\Omega}_\infty = C$. Then, by definition, $\Omega_{n_{m,m}}$ converges to $\Omega_\infty$ in the local Hausdorff topology. 
 
 We claim that $\Omega_\infty \in \Kb_d(r)$. Since each $\Omega_n$ is in $\Kb_d(r)$, Proposition~\ref{prop:haus_conv_compacts} Parts (2) and (3) imply that
 \begin{enumerate}
\item $r\Db \cdot e_1 \subset \Omega_\infty$ and $\Db \cdot e_j \subset \Omega_\infty$ for $j=2,\dots, d$
\item $e_j \in \partial \Omega_\infty$ and
\begin{align*}
\left( e_j + \Span_{\Cb}\{e_{j+1}, \dots, e_d\} \right) \cap \Omega_\infty =\emptyset
\end{align*}
for $j=1,\dots, d$.
\end{enumerate} 
So we just have to show that $\Omega_\infty \in \Xb_d$.  Since $0 \in \Omega_\infty$, using Observation~\ref{obs:asymptotic_cone_1} it is enough to show: if $\Cb \cdot v \subset \Omega_\infty$ for some $v \in \Cb^d$, then $v=0$. So suppose that $\Cb \cdot v \subset \Omega_\infty$. Since 
\begin{align*}
\left( e_1+ \Span_{\Cb}\{e_{2}, \dots, e_d\} \right) \cap \Omega_\infty =\emptyset
\end{align*}
we must have $v_1 = 0$. Then since 
\begin{align*}
\left( e_2+ \Span_{\Cb}\{e_{3}, \dots, e_d\} \right) \cap \Omega_\infty=\emptyset
\end{align*}
we must have $v_2 = 0$. Repeating the same argument shows that $v_3 = v_4 = \dots = v_d = 0$. So $v=0$ and hence $\Omega_\infty \in \Xb_d$. 
\end{proof}

\begin{theorem}\label{thm:normalizing} If $\Omega \subset \Cb^d$ is a $\Cb$-properly convex domain, $z_0 \in \Omega$, $\xi \in \partial \Omega$, $H$ is a supporting complex hyperplane of $\Omega$ at $\xi$, $q \in (\xi,z_0]$, and 
\begin{align*}
r:= \frac{\delta_\Omega(z_0)}{\norm{\xi-z_0}},
\end{align*}
then there exists an affine map $A$ with the following properties:
\begin{enumerate}
\item $A\Omega \in \Kb_d(r)$,
\item $A(q)=0$,
\item $A(\xi)=e_1$, 
\item $A(H) = e_1 + \Span_{\Cb}\{ e_2, \dots, e_d\}$, and 
\item if $\delta_{H} = \max\{ \delta_\Omega(q;v) : v\in -\xi + H \text{ non-zero}\}$, then
\begin{align*}
\norm{A(z_1) - A(z_2)} \geq \frac{r}{\sqrt{2} \delta_{H}} \norm{z_1-z_2}
\end{align*}
for all  $z_1, z_2 \in \Cb^d$. 
\end{enumerate}
\end{theorem}

\begin{remark} Notice that $-\xi+H$ is the complex hyperplane through $0$ which is parallel to $H$. \end{remark}

\begin{proof} By translating $\Omega$ we can assume that $q=0$. 

Since $\overline{\Omega}$ contains the convex hull of $\Bb_d(z_0;\delta_\Omega(z_0))$ and $\xi$ we see that: 
\begin{align}
\label{eq:dist_to_bd_dist_to_xi}
\delta_\Omega(0) \geq \frac{\delta_\Omega(z_0)}{\norm{z_0-\xi}} \norm{0-\xi} = r \norm{\xi}. 
\end{align}

We select points $x_1, \dots, x_d \in \partial \Omega$ and complex linear subspaces 
\begin{align*}
P_1 \supset P_2\supset \dots \supset P_d=\{0\}
\end{align*}
with $\dim_{\Cb} P_j = d-j$ using the following procedure. First let $x_1=\xi$ and $P_1 = -\xi+H$. Then assuming $x_1, \dots, x_{j-1}$ and $P_1,\dots, P_{j-1}$ have already been selected, let  $x_{j}$ be a point in $P_{j-1} \cap \partial \Omega$ closest to $q=0$ and let $P_{j}$ be the orthogonal complement of $\Cb \cdot x_j$ in $P_{j-1}$. Then define
\begin{align*}
\tau_j = \norm{x_j}.
\end{align*}

We claim that 
\begin{align}
\label{eq:hyperplane_normalizing_construction}
(x_j+P_j) \cap \Omega = \emptyset
\end{align}
for all $1 \leq j \leq d$. Since $x_1 = \xi$ and $P_1 = -\xi+H$, this clearly holds when $j=1$. Suppose $j > 1$. Then, since $P_{j-1} \cap \Omega$ is convex and $x_j \in \partial (P_{j-1} \cap \Omega)$, there exists a codimension one complex linear subspace $H_j \subset P_{j-1}$ such that $(x_j+H_j) \cap \Omega = \emptyset$. But by our choice of $x_j$ we have 
\begin{align*}
\Bb_d(0;\tau_j) \cap P_{j-1} \subset \Omega
\end{align*}
and $\norm{x_j}=\tau_j$. So $x_j+H_j$ must be tangent to $\partial \Bb_d(0;\tau_j)$ at $x_j$. Hence $H_j =  P_{j}$ and so $(x_j+P_j) \cap \Omega = \emptyset$. 

We next claim that $P_j = \Span_{\Cb}\{ x_{j+1}, \dots, x_d\}$ for all $1 \leq j \leq d$. By construction 
\begin{align*}
P_{j} = \Cb \cdot x_{j+1} + P_{j+1}
\end{align*}
where $P_0 := \Cb^d$. Thus 
\begin{align}
\label{eq:span_normalizing_construction}
P_j &= \Cb \cdot x_{j+1} + P_{j+1} = \Cb \cdot x_{j+1}+\Cb\cdot x_{j+2} + P_{j+2}\notag\\
&=\dots=  \Cb \cdot x_{j+1}+\cdots +\Cb \cdot x_{d} = \Span_{\Cb}\{ x_{j+1}, \dots, x_d\}.
\end{align}

Combining Equations~\eqref{eq:hyperplane_normalizing_construction} and~\eqref{eq:span_normalizing_construction} yields
\begin{align}
\label{eq:Pk_span}
\Omega \cap (x_j + \Span_{\Cb} \{x_{j+1}, \dots, x_d\}) = \emptyset
\end{align}
for all $1 \leq j \leq d$. 

Next let $\Lambda \in \GL_{d}(\Cb)$ be the diagonal matrix 
\begin{align*}
\begin{pmatrix}
\tau_1^{-1} & & \\
& \ddots & \\
& & \tau_{d}^{-1}
\end{pmatrix}.
\end{align*}
Then let $U \in \GL_d(\Cb)$ be the linear map such that 
\begin{align*}
\Lambda U(x_j) = e_j 
\end{align*}
for all $1 \leq j \leq d$. Notice that Equation~\eqref{eq:span_normalizing_construction} with $j=0$ implies that $x_1,\dots,x_d$ is a basis and so $U$ is uniquely defined. Finally, let $A = \Lambda U$. 

By construction we have $A(0)=0$ (that is, $A(q)=0$), $A(\xi)=e_1$, and 
$$
A(H) = A(x_1+P_1) = e_1 + \Span_{\Cb}\{ e_2, \dots, e_d\}.
$$

 We claim that $A\Omega \in \Kb_d(r)$. Since $\tau_1 = \norm{\xi}$, Equation~\eqref{eq:dist_to_bd_dist_to_xi} implies that
\begin{align*}
r \Db \cdot e_1\subset A\Omega.
\end{align*}
Further, for $j\geq 2$ we have 
\begin{align*}
\Db \cdot e_j\subset A\Omega
\end{align*}
since $x_j$ is a closest point to $q=0$ in $P_{j-1} \cap \partial \Omega$. Equation~\eqref{eq:Pk_span} and the definition of $A$ imply that $e_j \in \partial A\Omega$ and
\begin{align}
A\Omega \cap (e_j + \Span_{\Cb} \{e_{j+1}, \dots, e_d\}) = \emptyset.
\end{align}
So $A\Omega \in \Kb_d(r)$. 

Notice that 
\begin{align*}
\norm{A(z_1) - A(z_2)} =\norm{\Lambda U(z_1-z_2)} \geq \frac{1}{\norm{U^{-1}}\max \tau_j} \norm{z_1-z_2}
\end{align*}
for any  $z_1, z_2 \in \Cb^d$. Further, 
\begin{align*}
\tau_1 = \norm{\xi} \leq \frac{1}{r} \delta_\Omega(0) \leq \frac{1}{r} \delta_{H}
\end{align*}
and 
\begin{align*}
\tau_j =\delta_\Omega(0;x_j) \leq \delta_{H}
\end{align*}
for $j \geq 2$. So 
\begin{align*}
\norm{A(z_1) - A(z_2)} \geq \frac{r}{ \delta_{H}\norm{U^{-1}}} \norm{z_1-z_2}
\end{align*}
for any  $z_1, z_2 \in \Cb^d$. Thus we just have to bound $\norm{U^{-1}}$ from above. Now
\begin{align*}
U^{-1}(e_j) = \frac{x_j}{\tau_j}
\end{align*}
for all $1 \leq j \leq d$ and by construction $x_2,\dots, x_d$ are pairwise orthogonal. Hence 
\begin{align*}
\norm{U^{-1} v } \leq \abs{v_1} \norm{\frac{x_1}{\tau_1}}+\norm{\sum_{j=2}^d v_j \frac{x_j}{\tau_j} } =\abs{v_1} +  \sqrt{\sum_{j=2}^d \abs{v_j}^2}  \leq \sqrt{2} \norm{v}. 
\end{align*}
Thus
\begin{equation*}
\norm{A(z_1) - A(z_2)} \geq \frac{r}{\sqrt{2} \delta_{H}} \norm{z_1-z_2}
\end{equation*}
for all  $z_1, z_2 \in \Cb^d$. 
\end{proof}

Using Theorem~\ref{thm:normalizing} we can provide a proof of Theorem~\ref{thm:frankel_compactness}.

\begin{corollary} Define
\begin{align*}
\Kb_{d,0} := \{ (\Omega, 0) : \Omega \in \Kb_d(1) \}.
\end{align*}
Then $\Kb_{d,0}$ is a compact subset of $\Xb_{d,0}$ and $\Aff(\Cb^d) \cdot \Kb_{d,0} = \Xb_{d,0}$. 
\end{corollary}

\begin{proof} Since $\Kb_d(1)$ is a compact subset of $\Xb_d$, we see that  $\Kb_{d,0}$ is a compact subset of $\Xb_{d,0}$. Now fix some $(\Omega, q) \in \Xb_{d,0}$. Then apply Theorem~\ref{thm:normalizing} with $z_0 = q$ and $\xi \in \partial \Omega$ such that $\norm{q-\xi} = \delta_\Omega(q)$. Then 
\begin{align*}
\norm{\xi-z_0} = \delta_\Omega(z_0),
\end{align*}
and so there exists an affine map $A \in \Aff(\Cb^d)$ such that $A\Omega \in \Kb_d(1)$ and $A(q) = 0$. So $A(\Omega, q) \in \Kb_{d,0}$. Then since  $(\Omega, q) \in \Xb_{d,0}$ was arbitrary we see that $\Aff(\Cb^d) \cdot \Kb_{d,0} = \Xb_{d,0}$. 
\end{proof}

The following ``extension'' result will allow us to reduce many arguments to the $d=2$ case. 

\begin{proposition}\label{prop:compact_on_slices} Suppose $\Omega \subset \Cb^d$ is a $\Cb$-properly convex domain. If $1 \leq m \leq d$ and
\begin{align*}
\Omega \cap \Span_{\Cb}\{e_1,\dots, e_m\} \in \Kb_m(r),
\end{align*}
then there exists $A \in \GL_d(\Cb)$ such that $A\Omega \in \Kb_d(r)$ and $A|_{\Span_{\Cb}\{e_1,\dots, e_m\}} = \id$. 
\end{proposition}

\begin{proof} We will select points $x_1, \dots, x_d \in \partial \Omega$ and complex linear subspaces $P_1, \dots, P_d \subset \Cb^d$ with 
\begin{enumerate}
\item $P_1 \supset P_2 \supset \dots \supset P_d=\{0\}$, 
\item $\dim_{\Cb} P_j = d-j$ for $1 \leq j \leq d$, and
\item $\Span_{\Cb}\{e_{j+1},\dots,e_m\} \subset P_j$ for $1 \leq j \leq m-1$.
\end{enumerate}

First for $1 \leq j \leq m$, let $x_j = e_j$. Then we select $P_1,\dots, P_m$ sequentially as follows. Since $\Omega$ is convex and 
\begin{align*}
(e_1 + \Span_{\Cb}\{e_2,\dots,e_m\}) \cap \Omega = \emptyset,
\end{align*}
there exists a complex linear subspace $P_1$ such that $\dim_{\Cb} P_1 = d-1$,
\begin{align*}
 \Span_{\Cb}\{e_2,\dots,e_m\} \subset P_1,
\end{align*}
and 
\begin{align*}
(e_1+P_1) \cap \Omega = \emptyset.
\end{align*}
Then assuming $1 \leq j \leq m-1$ and we have already selected $P_1,\dots, P_j$, we select $P_{j+1}$ as follows. Since $\Omega \cap P_j$ is convex, 
\begin{align*}
\Span_{\Cb}\{e_{j+1},\dots,e_m\} \subset P_j,
\end{align*}
and 
\begin{align*}
(e_{j+1} + \Span_{\Cb}\{e_{j+2},\dots,e_m\}) \cap \Omega = \emptyset,
\end{align*}
there exists a codimension one complex linear subspace $P_{j+1} \subset P_j$ such that 
\begin{align*}
\Span_{\Cb}\{e_{j+2},\dots,e_m\} \subset P_{j+1}
\end{align*}
 and $(e_{j+1} + P_{j+1}) \cap \Omega = \emptyset$. 

Next we select $x_{m+1},\dots, x_d$ and  $P_{m+1}, \dots, P_d$. Supposing $j \geq m$ and that $x_1, \dots, x_{j}$ and $P_1, \dots, P_{j}$ have already been selected, we pick $x_{j+1}$ and $P_{j+1}$ as follows: let $x_{j+1}$ be a point in $P_{j} \cap \partial \Omega$ closest to $0$ and let $P_{j+1}$ be a $(d-j-1)$-dimensional complex subspace such that $P_{j+1} \subset P_{j}$ and $(x_{j+1} + P_{j+1}) \cap \Omega = \emptyset$. Since $P_{j} \cap \Omega$ is convex and $x_{j+1} \in \partial (P_{j} \cap \Omega)$, such a subspace exists.

Now let $A \in \GL_d(\Cb)$ be the complex linear map with $A(x_j)=e_j$ for $1 \leq j \leq d$. Since $x_1, \dots, x_d$ is a basis of $\Cb^d$, the linear map $A$ is well defined. Since $x_j=e_j$ when $1 \leq j \leq m$ we see that $A|_{\Span_{\Cb}\{e_1, \dots, e_m\}} = \id$. Arguing as in the proof of Theorem~\ref{thm:normalizing} shows that $A \Omega \in \Kb_d(r)$. 
\end{proof}

\part{Necessary and sufficient conditions for Gromov hyperbolicity} 

\section{Prior work and outline of the proof of Theorem~\ref{thm:main_equivalence}}\label{sec:prior_work_outline_Gromov}

In this section we recall some prior results concerning the Gromov hyperbolicity of the Kobayashi metric. Then we give an outline of the proof of Theorem~\ref{thm:main_equivalence}. 

In~\cite{Z2016}, we established the following necessary conditions. 

\begin{theorem}\label{thm:prior_nec}\cite{Z2016}
Suppose $\Omega$ is a $\Cb$-properly convex domain and $(\Omega, K_\Omega)$ is Gromov hyperbolic, then:
\begin{enumerate}
\item $\Omega$ has simple boundary, 
\item if $D \in \overline{\Aff(\Cb^d) \cdot \Omega} \cap \Xb_d$, then $(D,K_D)$ is Gromov hyperbolic, and
\item every domain in $\overline{\Aff(\Cb^d) \cdot \Omega} \cap \Xb_d$ has simple boundary. 
 \end{enumerate}
\end{theorem}

\begin{proof} Part (1) is~\cite[Theorem 1.6]{Z2016} and Part (2) is~\cite[Theorem 1.8]{Z2016}. Part (3) is an immediate consequence of Parts (1) and (2).\end{proof}

In~\cite{Z2016} we also established a sufficient condition for the Kobayashi metric to be Gromov hyperbolic, however the result requires several definitions to state. 

\begin{definition}
Given a curve $\sigma: \Rb \rightarrow \Cb^d$ the \emph{forward accumulation set of $\sigma$} is
\begin{align*}
\sigma(\infty) := \left\{ z \in \Cb^d \cup \{\infty\} : \text{ there exist } t_n \rightarrow \infty \text{ with } \sigma(t_n) \rightarrow z\right\}
\end{align*}
and the \emph{backward accumulation set of $\sigma$} is
\begin{align*}
\sigma(-\infty) := \left\{ z \in \Cb^d \cup \{\infty\} : \text{ there exist } t_n \rightarrow -\infty \text{ with } \sigma(t_n) \rightarrow z\right\}.
\end{align*}
\end{definition}

\begin{definition}\label{defn:well_behaved}
Suppose $\Omega \subset \Cb^d$ is a domain. We say \emph{geodesics in $(\Omega, K_\Omega)$ are well-behaved} if 
\begin{align*}
\sigma(\infty) \cap \sigma(-\infty) = \emptyset
\end{align*}
for every geodesic line $\sigma: \Rb \rightarrow \Omega$.
\end{definition}

\begin{definition}\label{defn:visibility_seqn}
Suppose $\Omega_n$ converges to $\Omega$ in $\Xb_d$. We say $\Omega_n$ is a \emph{visibility sequence} if for every sequence $\sigma_n :[a_n,b_n] \rightarrow \Omega_n$ of geodesics with 
\begin{align*}
\lim_{n \rightarrow\infty} \sigma_n(a_n) = \xi \in \Cb^d \cup\{\infty\},
\end{align*}
\begin{align*}
\lim_{n \rightarrow\infty} \sigma_n(b_n) = \eta \in \Cb^d \cup\{\infty\},
\end{align*}
and $\xi \neq \eta$, then there exist sequences $n_j \rightarrow \infty$ and $T_j \in [a_{n_j},b_{n_j}]$ such that $\sigma_{n_j}(T_j)$ converges to a point in $\Omega$. \end{definition}

\begin{remark} Informally the visibility condition says that geodesic segments between distinct points ``bend'' into the domain. 
\end{remark}

\begin{theorem}\label{thm:prior_suff}\cite[Theorem 8.3]{Z2016}
Suppose $\Omega$ is a $\Cb$-properly convex domain. Assume for any sequence $u_n \in \Omega$ there exist $n_j \rightarrow \infty$ and $A_j \in \Aff(\Cb^d)$ so that 
\begin{enumerate}
\item $A_j(\Omega, u_{n_j})$ converges to some $(\Omega_\infty, u_\infty)$ in $\Xb_{d,0}$,
\item geodesics in $(\Omega_\infty, K_{\Omega_\infty})$ are well behaved, and
\item $A_j\Omega$ is a visibility sequence. 
\end{enumerate}
Then $(\Omega, K_\Omega)$ is Gromov hyperbolic. 
\end{theorem}

Theorem 8.3 in~\cite{Z2016} is formulated in a different way,  so we will provide the argument. But first a lemma.  

\begin{lemma}\label{lem:limits_of_geod}
Assume that $\Omega_n$ is a visibility sequence converging to some $\Omega_\infty$ in $\Xb_{d}$ and $\sigma_n :[0,T_n] \rightarrow \Omega_n$ is a sequence of geodesics which converges locally uniformly to a geodesic $\sigma: [0,\infty) \rightarrow \Omega_\infty$. Then 
\begin{align}
\label{eq:limits_equal}
\lim_{t \rightarrow \infty} \sigma(t) = \lim_{n \rightarrow \infty} \sigma_n(T_n) \in \Cb^d \cup \{\infty\}
\end{align}
(in particular, the two limits exist). 
\end{lemma}

\begin{proof} Suppose for a contradiction that Equation~\eqref{eq:limits_equal} is false. Then there exist $s_m \rightarrow \infty$, $n_m \rightarrow \infty$, and $\eta, \xi \in \Cb^d \cup \{\infty\}$ such that $\sigma(s_m) \rightarrow \eta$,  $\sigma_{n_m}(T_{n_m}) \rightarrow \xi$, and $\eta \neq \xi$. Since $s_m, T_{n_m} \rightarrow \infty$, Proposition~\ref{prop:convergence_of_kob} implies that $\eta, \xi \in \partial \Omega_\infty \cup \{\infty\}$. 

Since $\sigma_{n}$ converges locally uniformly to $\sigma$ we can pick $s_n^\prime$ so that $\sigma_{n}(s_n^\prime) \rightarrow \eta$. Since $\eta \in \partial \Omega_\infty \cup \{\infty\}$, Observation~\ref{obs:AA_for_geod} implies that $s_n^\prime \rightarrow \infty$. 

Let $a_m  = \min\{ s_{n_m}^\prime, T_{n_m}\}$ and $b_m = \max\{ s_{n_m}^\prime, T_{n_m}\}$. 
Since $\Omega_{n}$ is a visibility sequence we can pass to another subsequence and find $S_m \in [a_m, b_m]$ so that $\sigma_{n_m}(S_m)$ converges to a point $z_\infty \in \Omega_\infty$. Notice that $S_m \rightarrow \infty$ since  $a_{m} \rightarrow \infty$. But then by Proposition~\ref{prop:convergence_of_kob}
\begin{align*}
\infty & > K_{\Omega_\infty}(z_\infty, \sigma(0)) = \lim_{m \rightarrow \infty} K_{\Omega_{n_m}}(\sigma_{n_m}(S_m), \sigma_{n_m}(0))   = \lim_{m \rightarrow \infty} S_m = \infty.
\end{align*}
So we have a contradiction.

\end{proof}

\begin{proof}[Proof of Theorem~\ref{thm:prior_suff}]
Suppose for a contradiction that $(\Omega, K_\Omega)$ is not Gromov hyperbolic. Then by Theorem~\ref{thm:thin_triangle}, for every $n \in \Nb$ there exists a geodesic triangle with vertices $x_n, y_n, z_n \in \Omega$ and edges $\sigma_{x_ny_n}, \sigma_{y_nz_n}, \sigma_{z_nx_n}$ such that 
\begin{align*}
K_\Omega(u_n, \sigma_{y_nz_n} \cup \sigma_{z_nx_n}) > n
\end{align*}
for some $u_n$ in the geodesic $\sigma_{x_ny_n}$. Notice that 
\begin{align}
\label{eq:un_dist_to_xn_yn_zn}
K_\Omega(u_n, \{x_n,y_n,z_n\}) > n.
\end{align}

After possibly passing to a subsequence, there exist affine maps $A_n \in \Aff(\Cb^d)$ such that 
\begin{enumerate}
\item $A_n(\Omega, u_{n})$ converges to some $(\Omega_\infty,u_\infty)$ in $\Xb_{d,0}$,
\item geodesics in $(\Omega_\infty, K_{\Omega_\infty})$ are well behaved, and
\item $A_n\Omega$ is a visibility sequence. 
\end{enumerate}
By passing to another subsequence we can suppose that $A_nx_n$, $A_ny_n$, $A_nz_n$ converge to $x_\infty$, $y_\infty$, $z_\infty$ in $\Cb^d \cup \{\infty\}$. 

We can parameterize $\sigma_{x_ny_n}:[a_n,b_n] \rightarrow \Omega$ so that $\sigma_{x_ny_n}(0)=u_n$. Notice that Equation~\eqref{eq:un_dist_to_xn_yn_zn} implies that 
\begin{align*}
\lim_{n \rightarrow \infty} a_n = \lim_{n \rightarrow \infty} - K_\Omega(x_n,u_n)=- \infty 
\end{align*} 
and 
\begin{align*}
\lim_{n \rightarrow \infty} b_n = \lim_{n \rightarrow \infty} K_\Omega(u_n,y_n)=\infty.
\end{align*} 
Observation~\ref{obs:AA_for_geod} implies that we can pass to a subsequence so that $A_n\sigma_{x_ny_n}$ converges to a geodesic $\sigma_{xy} : \Rb \rightarrow \Omega_\infty$ with $\sigma_{xy}(0) = u_\infty$. 

By Lemma~\ref{lem:limits_of_geod}
\begin{align*}
\lim_{t \rightarrow -\infty} \sigma_{xy}(t) = \lim_{n \rightarrow \infty} A_n x_n = x_\infty
\end{align*}
and
\begin{align*}
\lim_{t \rightarrow \infty} \sigma_{xy}(t) = \lim_{n \rightarrow \infty} A_n y_n = y_\infty. 
\end{align*}
Since geodesics in $(\Omega_\infty, K_{\Omega_\infty})$ are well behaved, we have $x_\infty \neq y_\infty$. So by possibly relabelling $x_n$ and $y_n$, we may assume that $z_\infty \neq x_\infty$. Then since $A_n\Omega$ is a visibility sequence, we can pass to a subsequence and reparametrize $\sigma_{z_nx_n}$ to assume that $A_n\sigma_{z_nx_n}(0)$ converges to a point $w_\infty \in \Omega_\infty$. Then by Proposition~\ref{prop:convergence_of_kob}
\begin{align*}
K_{\Omega_\infty}(u_\infty,w_\infty) 
&= \lim_{n \rightarrow \infty} K_{A_n\Omega}(A_nu_n, A_n\sigma_{z_nx_n}(0)) =
\lim_{n \rightarrow \infty} K_{\Omega}(u_n, \sigma_{z_nx_n}(0)) \\
& \geq \lim_{n \rightarrow \infty} K_{\Omega}(u_n, \sigma_{z_nx_n})\geq \lim_{n \rightarrow \infty} n = \infty.
\end{align*}
So we have a contradiction. 

\end{proof}

\subsection{A sufficient condition for visibility}\label{sec:suff_for_visible}

Motivated by work of Mercer, in~\cite{Z2016} we established a sufficient condition for a sequence of convex domains to be a visibility sequence. 

\begin{definition}[{Mercer \cite[Definition 2.7]{M1993}}] For $m \geq 1$, a bounded convex domain $\Omega$ is called \emph{$m$-convex} if there exists $C >0$ such that 
\begin{align}
\label{eqn:defn_m_convex}
\delta_{\Omega}(z; v) \leq C \delta_{\Omega}(z)^{1/m}
\end{align}
for all $z \in \Omega$ and non-zero $v \in \Cb^d$.
\end{definition}

\begin{remark}\label{rmk:m_convex_values} When $d=1$, any convex domain is 1-convex since $\delta_\Omega(z;v) = \delta_\Omega(z)$ for every $z \in \Omega$ and non-zero $v \in \Cb$. When $d \geq 2$, any $m$ that satisfies Equation~\eqref{eqn:defn_m_convex} has to be at least two: by Alexandrov's theorem $\partial \Omega$ contains a $C^2$ point $\xi$ and 
\begin{align*}
\delta_\Omega(\xi + t n_\xi; v) \approx t^{1/2} \approx \delta_\Omega(\xi + t n_\xi)^{1/2}
\end{align*}
if $t > 0$ is sufficiently small, $n_\xi$ is the inward pointing unit normal vector of $\partial \Omega$ at $\xi$, and $\xi+\Cb \cdot v$ is tangent to $\partial \Omega$ at $\xi$.
\end{remark}

When $\Omega$ is a smoothly bounded convex domain, it is easy to show that $\Omega$ is $m$-convex for some $m$ if and only if $\partial \Omega$ has finite type in the sense of D'Angelo, see for instance~\cite[Section 9]{Z2016}. Thus, for convex domains $m$-convexity can be viewed as a low regularity analogue of finite type. 

For $m$-convex domains, Mercer proved a type of visibility result for complex geodesics, see~\cite[Lemma 3.3]{M1993}. Motivated by this result we established the following visibility result for sequences of domains. 

\begin{proposition}\label{prop:visible_sequence}\cite[Proposition 7.8]{Z2016}
Suppose $\Omega_n$ converges to $\Omega$ in $\Xb_d$. Assume for any $R>0$ there exist $C, N>0$ and $m \geq 1$ such that 
\begin{align*}
\delta_{\Omega_n}(z; v) \leq C \delta_{\Omega_n}(z)^{1/m}
\end{align*}
for all $n\geq N$, $z \in \Bb_d(0;R) \cap \Omega_n$, and $v \in \Cb^d$ non-zero. Then $\Omega_n$ is a visibility sequence. 
\end{proposition}

The proof in~\cite[Proposition 7.8]{Z2016} is somewhat indirect: first a visibility result for complex geodesics is established and then this is used to establish a visibility result for geodesics. A more direct argument can be found in~\cite[Proposition 4.5.10]{BC2017}.

\subsection{Outline of the proof of Theorem~\ref{thm:main_equivalence}} Theorem~\ref{thm:prior_nec} provides one direction of the desired equivalence, so we only have to consider the case when $\Omega \subset \Cb^d$ is a bounded convex domain and every domain in 
\begin{align*}
\overline{\Aff(\Cb^d) \cdot \Omega} \cap \Xb_d
\end{align*}
has simple boundary. 

We will use Theorem~\ref{thm:prior_suff} to show that $(\Omega, K_\Omega)$ is Gromov hyperbolic. Here is the sketch of the argument: fix a sequence $u_n \in \Omega$. Then by Theorem~\ref{thm:frankel_compactness} we can find a sequence of affine maps $A_n$ such that $\{A_n(\Omega, u_n) : n \in \Nb\}$ is relatively compact in $\Xb_{d,0}$. Then by passing to a subsequence we can suppose that $A_n(\Omega,u_n)$ converges to some $(\Omega_\infty,u_\infty) \in \Xb_{d,0}$. To apply Theorem~\ref{thm:prior_suff}, we need to show that $A_n \Omega$ is a visibility sequence and geodesics in $\Omega_\infty$ are well behaved. This will be accomplished as follows:
\begin{enumerate}
\item In Section~\ref{sec:local_m_convexity}, we prove general results which imply that $A_n \Omega$ satisfies the hypothesis of Proposition~\ref{prop:visible_sequence} and hence is a visibility sequence. 
\item In Section~\ref{sec:m_convex_vs_gromov}, we discuss the general relationship between $m$-convexity and Gromov hyperbolicity. This is not necessary for the proof of Theorem~\ref{thm:main_equivalence}, but clarifies the relationship between the two definitions. 
\item In Section~\ref{sec:behavior_of_geodesics}, we prove general results which will imply that geodesics in $\Omega_\infty$ are well behaved. 
\item In Section~\ref{sec:pf_of_thm_main_equivalence}, we prove a generalization of Theorem~\ref{thm:main_equivalence}.
\end{enumerate}

\section{Local $m$-convexity}\label{sec:local_m_convexity}

In this section we establish the following sufficient condition for a local $m$-convexity condition to hold. 

\begin{theorem}\label{thm:loc_m_convex} Suppose $\Kc \subset \Xb_d$ is a compact set and every domain in 
\begin{align*}
\overline{\Aff(\Cb^d) \cdot \Kc} \cap \Xb_d
\end{align*} 
has simple boundary. Then for any $R > 0$ there exist $C>0$ and $m \geq 1$ such that 
\begin{align*}
\delta_{\Omega}(z; v) \leq C \delta_{\Omega}(z)^{1/m}
\end{align*}
for all $\Omega \in \Kc$, $z \in \Bb_d(0;R) \cap \Omega$, and $v \in \Cb^d$ non-zero. 
\end{theorem}

Before proving the theorem, we state and prove two corollaries.

\begin{corollary}\label{cor:m_convex}
Suppose  $\Omega$ is a $\Cb$-properly convex domain and every domain in $\overline{\Aff(\Cb^d) \cdot \Omega} \cap \Xb_d$ has simple boundary. Then for any $R > 0$ there exist $C>0$ and $m\geq 1$ such that 
\begin{align*}
\delta_{\Omega}(z; v) \leq C \delta_{\Omega}(z)^{1/m}
\end{align*}
for all $z \in \Bb_d(0;R) \cap \Omega$ and $v \in \Cb^d$ non-zero. 
\end{corollary}

\begin{proof}[Proof of Corollary~\ref{cor:m_convex}]
Simply apply Theorem~\ref{thm:loc_m_convex} to $\Kc := \{ \Omega\}$. 
\end{proof}

\begin{corollary}\label{cor:simple_bd_implies_visibility}
Suppose $\Omega$ is a $\Cb$-properly convex domain and every domain in $\overline{\Aff(\Cb^d) \cdot \Omega} \cap \Xb_d$ has simple boundary. If $A_n \in \Aff(\Cb^d)$ is a sequence of affine maps such that $A_n \Omega$ converges to some $\Omega_\infty$ in $\Xb_{d}$, then the sequence $A_n \Omega$ is a visibility sequence. \end{corollary}

\begin{proof}
Since $A_n \Omega$ converges to $\Omega_\infty$, the set $\Kc = \{ A_n \Omega : n \geq 1\} \cup \{\Omega_\infty\}$ is compact in $\Xb_{d}$. Further, 
\begin{align*}
\overline{\Aff(\Cb^d) \cdot \Omega_\infty} \cap \Xb_d \subset \overline{\Aff(\Cb^d) \cdot \Omega} \cap \Xb_d
\end{align*} 
and so
\begin{align*}
\overline{\Aff(\Cb^d) \cdot \Kc} \cap \Xb_d = \overline{\Aff(\Cb^d) \cdot \Omega} \cap \Xb_d.
\end{align*} 
So Theorem~\ref{thm:loc_m_convex} implies that for any $R > 0$ there exist $C>0$ and $m\geq 1$ such that 
\begin{align*}
\delta_{A_n\Omega}(z; v) \leq C \delta_{A_n\Omega}(z)^{1/m}
\end{align*}
for all $n \geq 1$, $z \in \Bb_d(0;R) \cap A_n\Omega$, and $v \in \Cb^d$ non-zero. Then $A_n\Omega$ is a visibility sequence by Proposition~\ref{prop:visible_sequence}.

\end{proof}

The rest of the section is devoted to the proof of Theorem~\ref{thm:loc_m_convex}. So fix a compact set $\Kc \subset \Xb_d$ where every domain in $\overline{\Aff(\Cb^d) \cdot \Kc} \cap \Xb_d$ has simple boundary.  

\begin{lemma}\label{lem:WOLOG_contains_zero} Without loss of generality we can assume that $0 \in \Omega$ for every $\Omega \in \Kc$. \end{lemma}

\begin{proof} We first claim that there exists $r > 0$ such that: for every $\Omega \in \Kc$ there exists $z \in \Omega$ with $\norm{z} \leq r$ and $\delta_\Omega(z) \geq 1/r$. Suppose not, then for every $n \in\Nb$ there exists $\Omega_n \in \Kc$ with 
\begin{align*}
\left\{ z \in \Omega_n : \norm{z} \leq n \text{ and } \delta_{\Omega_n}(z) \geq 1/n \right\}  = \emptyset.
\end{align*}
Since $\Kc$ is compact, we can pass to a subsequence and suppose that $\Omega_n$ converges to some $\Omega_\infty$ in $\Xb_d$. 

Fix some $u \in \Omega_\infty$. Then let $r_0 = \max\{ \norm{u}, \frac{2}{\delta_\Omega(u)}\}$. Then  $\Bb_d(u;2/r_0) \subset \Omega_\infty$ and so by Proposition~\ref{prop:haus_conv_compacts} part (1) there exists $N > 0$ such that $\overline{\Bb_d(u;1/r_0)} \subset \Omega_n$ for every $n \geq N$. Thus 
\begin{align*}
u \in \left\{ z \in \Omega_n :  \norm{z} \leq n \text{ and } \delta_{\Omega_n}(z) \geq 1/n  \right\} 
\end{align*}
when $n \geq \max\{r_0,N\}$ and so we have a contradiction. Hence there exists some $r> 0$ with the desired property. 

Next let $\Kc_0$ denote the set of domains of the form $-z+\Omega$ where $\Omega \in \Kc$, $z \in \Omega$, $\norm{z} \leq r$, and $\delta_\Omega(z) \geq 1/r$. Then $\Kc_0$ is compact in $\Xb_{d}$ and $0 \in \Omega$ for every $\Omega \in \Kc_0$. Further $\Kc_0 \subset \Aff(\Cb^d) \cdot \Kc$ and so
\begin{align*}
\overline{\Aff(\Cb^d) \cdot \Kc_0} \cap \Xb_d=\overline{\Aff(\Cb^d) \cdot \Kc} \cap \Xb_d.
\end{align*}
Hence $\Kc_0$ satisfies the hypothesis of Theorem~\ref{thm:loc_m_convex}. Finally, since every domain in $\Kc$ is a bounded translate of a domain in $\Kc_0$, if Theorem~\ref{thm:loc_m_convex}  is true for $\Kc_0$ it is also true for $\Kc$. 
\end{proof}

Using Lemma~\ref{lem:WOLOG_contains_zero}, we may assume that $0 \in \Omega$ for every $\Omega \in \Kc$. Then, since $\Kc$ is compact, there exists $\delta_0 > 0$ such that 
\begin{align*}
\overline{\Bb_d(0;\delta_0)} \subset \Omega
\end{align*}
for every $\Omega \in \Kc$.

Next for $\Omega \in \Kc$ and $z\in \Omega \setminus \{0\}$, define $\pi_\Omega(z) \in \partial \Omega \cup \{\infty\}$ as follows: if 
\begin{align*}
\Omega \cap \Rb_{\geq 0} \cdot z =  \Rb_{\geq 0} \cdot z, 
\end{align*}
then let $\pi_\Omega(z) = \infty$. Otherwise, let 
\begin{align*}
\{ \pi_\Omega(z) \} = \partial\Omega \cap \Rb_{\geq 0} \cdot z .
\end{align*}

For the rest of section fix $R > 0$. Then for $\Omega \in \Kc$ let
\begin{align*}
\Omega^{(R)} := \left\{ z \in \Omega  \setminus \{0\}: \norm{z} \leq R \text{ and } \norm{\pi_\Omega(z)} \leq R+1\right\}.
\end{align*}
Also for $z \in \Omega^{(R)}$, let $T_\Omega(z) \subset \Cb^d$ denote the set of unit vectors $v \in \Cb^d$ where
\begin{align*}
(\pi_\Omega(z) + \Cb \cdot v) \cap \Omega = \emptyset.
\end{align*}
Notice that, since $ \Omega$ is convex, the set $T_\Omega(z)$ consists of a union of complex hyperplanes intersected with the unit sphere. 

Define
\begin{align*}
r_0:=\frac{\delta_0}{R+1}.
\end{align*}

\begin{lemma}\label{lem:dist_est} If $\Omega \in \Kc$ and  $z \in \Omega^{(R)} $, then
\begin{align*}
\delta_{\Omega}(z) \geq  r_0 \norm{ \pi_\Omega(z) - z }.
\end{align*}
\end{lemma}

\begin{proof} Notice that $\overline{\Omega}$ contains the convex hull of $\Bb_d(0;\delta_0)$ and $\pi_\Omega(z)$. \end{proof}

We first establish the theorem for certain base points and directions. 

\begin{lemma} There exist $C_0>0$ and $m \geq 1$ such that: if $\Omega \in \Kc$, then
\begin{align*}
\delta_{\Omega}(z;v) \leq C_0 \delta_{\Omega}(z)^{1/m}
\end{align*}
for all $z \in \Omega^{(R)}$ and $v \in T_\Omega(z)$.
\end{lemma}

\begin{proof} For $\Omega \in \Kc$ and $z \in \Omega^{(R)}$ define $r_\Omega(z): = \norm{\pi_\Omega(z)-z}$. By the estimate in Lemma~\ref{lem:dist_est} it is enough to prove that there exist $C>0$ and $m \geq1 $  such that 
\begin{align*}
\delta_{\Omega}(z;v) \leq C r_\Omega(z)^{1/m}
\end{align*}
for all $\Omega \in \Kc$, $z \in \Omega^{(R)}$, and $v \in T_\Omega(z)$.

Suppose for a contradiction that such $C>0$, $m \geq 1$ do not exist. Then for each $m \in \Nb$ we can find $\Omega_m \in \Kc$, $z_m \in \Omega^{(R)}_m$, and $v_m \in T_{\Omega_m}(z_m)$ such that 
\begin{align*}
\delta_{\Omega_m}(z_m;v_m) = C_m r_{\Omega_m}(z_m)^{1/m}
\end{align*}
and $C_m \geq m$. Since $\Kc$ is compact in $\Xb_d$ we have
\begin{align*}
M:=\sup \left\{ \delta_{ \Omega}(z;v) : \Omega \in \Kc, z \in \Omega \cap \overline{\Bb_d(0;R)}, v \in \Cb^d \setminus\{0\} \right\} < \infty.
\end{align*}
Then, since $C_m \geq m$, we must have 
\begin{align}
\label{eq:r_goes_to_zero}
\lim_{m \rightarrow \infty} r_{\Omega_m}(z_m) \leq \lim_{m \rightarrow \infty} \left(\frac{M}{C_m}\right)^m= 0.
\end{align} 

Since $\Omega_{m}$ is convex, the function $f_m:[0,1]\rightarrow \Rb$ defined by
\begin{align*}
f_m(t) = \frac{\norm{\pi_{\Omega_m}(z_m)-tz_m}^{1/m}}{\delta_{\Omega_m}(tz_m;v_m)}
\end{align*}
is continuous. Let $t_m \in [0,1]$ be a minimum point of $f$. Notice that $f_m(1) = \frac{1}{C_m} \leq \frac{1}{m}$ and
$$
f_m(0) = \frac{ \norm{\pi_{\Omega_m}(z_m)}^{1/m}}{ \delta_{\Omega_m}(0;v_m)} \geq \frac{ \delta_0^{1/m}}{M}.
$$
So for $m$ sufficiently large, $f_m(1) < f_m(0)$ and hence $t_m \in (0,1]$. So after possibly passing to a tail of the sequence, replacing $z_m$ with $t_mz_m$, and increasing $C_m$, we can assume that each $z_m$ has the following extremal property: 
\begin{align}
\label{eq:maximal_choice}
\delta_{ \Omega_m}(tz_m;v_m) \leq C_m r_{\Omega_m}(tz_m)^{1/m}
\end{align}
for all $t \in (0,1]$. Finally, by replacing $v_m$ by some $e^{i\theta_m}v_m$ where $\theta_m \in \Rb$, we can assume that 
\begin{align*}
z_m + C_m r_{\Omega_m}(z_m)^{1/m} v_m \in \partial \Omega_m.
\end{align*}
Notice that $v_m$ is still contained in $T_{\Omega_m}(z_m)$. 

Let 
\begin{align*}
a_m := \pi_{\Omega_m}(z_m) \in \partial \Omega_{m}
\end{align*}
and 
\begin{align*}
b_m:=z_m + C_m r_{\Omega_m}(z_m)^{1/m} v_m \in \partial \Omega_m.
\end{align*}
Then let $B_m \in \Aff(\Cb^d)$ be an affine map such that $B_m(z_m)=0$, $B_m(a_m)=e_1$, and $B_m(b_m) = e_2$. By Lemma~\ref{lem:dist_est}, we see that 
\begin{align*}
r_0\Db \cdot e_1 \subset B_m\Omega_m
\end{align*}
and since $v_m \in T_{\Omega_m}(z_m)$ we see that 
\begin{align*}
B_m \Omega_m \cap (e_1 + \Cb \cdot e_2) = \emptyset.
\end{align*}
By construction $e_2=B_m(b_m) \in \partial B_m \Omega_m$ and since $\delta_{\Omega_m}(z_m;v_m) = \norm{b_m-z_m}$ we see that 
\begin{align*}
\Db \cdot e_2 \subset B_m\Omega_m.
\end{align*}
Thus 
\begin{align*}
B_m \Omega_m \cap \Span_{\Cb}\{e_1, e_2\} \in \Kb_2(r_0).
\end{align*}
So by Proposition~\ref{prop:compact_on_slices},  we can assume that $B_m \Omega_m \in \Kb_d(r_0)$. Then, since $\Kb_d(r_0)$ is compact, we can pass to a subsequence so that $B_m \Omega_m$ converges to some $\mathcal{D}_1$ in $\Xb_d$. 

Next define
\begin{align*}
\mathcal{C}:=\bigcup_{-\infty < t < 1} \Bb_1\Big(t; r_0\abs{t-1}\Big) \subset \Cb.
\end{align*}
Then $\mathcal{C}$ is a convex open cone in $\Cb$ based at $1$. 

\medskip

\noindent \textbf{Claim 1:} $\mathcal{C} \times \{(0,\dots, 0)\}\subset \mathcal{D}_1$. 

\medskip

\noindent \emph{Proof of Claim 1:} Since 
$$
B_m(a_m+\lambda z_m)  = \left(1 + \frac{\lambda}{r_{\Omega_m}(z_m)} \right) e_1 \quad \text{for all } \lambda \in \Cb,
$$
Lemma~\ref{lem:dist_est} implies 
\begin{align*}
\mathcal{C} \times \{(0,\dots, 0)\} \cap \Bb_d\left(0; \frac{\norm{\pi_{\Omega_m}(z_m)}}{r_{\Omega_m}(z_m)}\right) \subset B_m \Omega_m.
\end{align*}
So it suffices to show that 
\begin{align*}
\lim_{m \rightarrow \infty} \frac{\norm{\pi_{\Omega_m}(z_m)}}{r_{\Omega_m}(z_m)} = \infty.
\end{align*}
Using the fact that $\delta_{\Omega_m}(0) \geq \delta_0$, we have 
\begin{align}
\label{eq:z_m_lower_bound}
\liminf_{m \rightarrow \infty}\norm{\pi_\Omega(z_m)} \geq \liminf_{m \rightarrow \infty}\delta_{\Omega_m}(0) \geq \delta_0. 
\end{align}
Then combining Equations~\eqref{eq:r_goes_to_zero} and~\eqref{eq:z_m_lower_bound} yields 
\begin{align*}
\lim_{m \rightarrow \infty} \frac{\norm{\pi_{\Omega_m}(z_m)}}{r_{\Omega_m}(z_m)} = \infty.
\end{align*}
This proves Claim 1. \hspace*{\fill}$\blacktriangleleft$

\medskip

\noindent \textbf{Claim 2:} $\Rb_{\leq 0} \cdot e_1 + \Db \cdot e_2 \subset \mathcal{D}_1$.

\medskip

\noindent \emph{Proof of Claim 2:}
By Claim 1 we have
\begin{align*}
\Rb_{\leq 0}\cdot e_1 \subset  \mathcal{D}_1.
\end{align*}
Since $\Dc_1 \in \Kb_d(r_0)$ we have $\Db \cdot e_2 \subset  \mathcal{D}_1$. So by Observation~\ref{obs:asymptotic_cone_1}
\begin{align*}
\Rb_{\leq 0} \cdot e_1 + \Db \cdot e_2 \subset  \mathcal{D}_1.
\end{align*}
This proves Claim 2. \hspace*{\fill}$\blacktriangleleft$

\medskip

\noindent \textbf{Claim 3:}  For each $t \geq 0$ there exists some $\lambda_t \in \partial \Db$ such that
\begin{align*}
-te_1 + \lambda_t e_2 \in \partial  \mathcal{D}_1. 
\end{align*}

\medskip

\noindent \emph{Proof of Claim 3:} 
Since $ \mathcal{D}_1 \in \Kb_d(r_0)$, we have $e_2 \in \partial  \mathcal{D}_1$ and so the claim is true when $t =0$. Next fix $t > 0$. Then for $m$ sufficiently large
\begin{align*}
B_m^{-1}(-te_1) \in (0,z_m)
\end{align*}
and
\begin{align*}
r_{\Omega_m}( B_m^{-1}(-te_1)) = (1+t)r_{\Omega_m}(z_m).
\end{align*}
Then by Equation~\eqref{eq:maximal_choice}
\begin{align*}
\delta_{ \Omega_m}(B_m^{-1}(-te_1);v_m) \leq C_m r_{\Omega_m}( B_m^{-1}(-te_1))^{1/m} =  C_m(1+t)^{1/m} r_{\Omega_m}(z_m)^{1/m}.
\end{align*}
Then
\begin{align*}
\delta_{ B_m\Omega_m}(-te_1;e_2) = \frac{1}{C_mr_{\Omega}(z_m)^{1/m}} \delta_{\Omega_m}(B_m^{-1}(-te_1);v_m) \leq (1+t)^{1/m}.
\end{align*}
So
\begin{align*}
\delta_{ \mathcal{D}_1}(-te_1; e_2) =\lim_{m \rightarrow \infty} \delta_{ B_m\Omega}(-te_1;e_2)  \leq \lim_{m \rightarrow \infty} (1+t)^{1/m}=1.
\end{align*}

By Claim 2, we have $\delta_{ \mathcal{D}_1}(-te_1;e_2) \geq 1$ and so we must have
\begin{align*}
\delta_{ \mathcal{D}_1}(-te_1;e_2) = 1.
\end{align*}
This proves Claim 3.  \hspace*{\fill}$\blacktriangleleft$

\medskip

Now for each $k \in \Nb$, let $A_k \in \Aff(\Cb^2)$ be the affine map 
\begin{align*}
A_k(z) =e_1+ \begin{pmatrix} \frac{1}{k+1} & 0 \\ 0 & \lambda_k^{-1} \end{pmatrix} (z - e_1).
\end{align*}

\noindent \textbf{Claim 4:} For all $k \geq 0$, 
\begin{align*}
A_k( \mathcal{D}_1 \cap \Span_{\Cb}\{e_1,e_2\}) \in \Kb_2(r_0).
\end{align*}

\medskip

\noindent \emph{Proof of Claim 4:} Let $\mathcal{U}_k := A_k( \mathcal{D}_1 \cap \Span_{\Cb}\{e_1,e_2\})$. 

Since $A_k(e_1 + \Cb \cdot e_2)=e_1+\Cb\cdot e_2$ and $ \mathcal{D}_1 \in \Kb_d(r_0)$, we see that $e_1 \in \partial \mathcal{U}_k$ and $(e_1+\Cb \cdot e_2) \cap \mathcal{U}_k = \emptyset$. By Claim 3, $-ke_1 + \lambda_k e_2 \in \partial  \mathcal{D}_1$ and so
\begin{align*}
e_2 = A_k(-ke_1 + \lambda_k e_2) \in \partial \mathcal{U}_k.
\end{align*}
By Claim 2 and 3, $\delta_{ \mathcal{D}_1}(-ke_1;e_2) = 1$ and so 
\begin{align*}
\Db \cdot e_2 = A_k(-ke_1+\Db\cdot e_2) \subset \mathcal{U}_k.
\end{align*}
 Finally, by Claim 1 
\begin{align*}
\mathcal{C} \times \{0\} = A_k \left( \mathcal{C} \times \{0\} \right)  \subset \mathcal{U}_k
\end{align*}
and so $r_0\Db \cdot e_1 \subset \mathcal{U}_k$. Thus $\mathcal{U}_k \in \Kb_2(r_0)$. \hspace*{\fill}$\blacktriangleleft$

\medskip

Now using Proposition~\ref{prop:compact_on_slices} we can extend $A_k$ to an affine automorphism of $\Cb^d$ such that $A_k \mathcal{D}_1 \in \Kb_d(r_0)$. Then by passing to a subsequence we can suppose that $A_k \mathcal{D}_1$ converges to some $\mathcal{D}_2$ in $\Xb_d$. Now since each $A_k\mathcal{D}_1$ is in $\Kb_d(r_0)$ we see that 
\begin{align}
\label{eq:supporting_hyperplane_D}
(e_1 + \Span_{\Cb}\{e_2,\dots, e_d\})\cap \mathcal{D}_2 = \emptyset.
\end{align}
Further,
\begin{align*}
\left(1-\frac{1}{k+1}\right)e_1 + \Db \cdot e_2 = A_k( \Db \cdot e_2) \subset A_k \mathcal{D}_1
\end{align*}
and so $e_1 + \Db \cdot e_2 \subset \overline{\mathcal{D}}_2$. Then Equation~\eqref{eq:supporting_hyperplane_D} implies that $e_1+\Db \cdot e_2 \subset \partial \mathcal{D}_2$. But 
\begin{align*}
\mathcal{D}_2 \subset \overline{\Aff(\Cb^d) \cdot \mathcal{D}_1} \cap \Xb_d \subset \overline{\Aff(\Cb^d) \cdot \Kc} \cap \Xb_d
\end{align*}
which contradicts the assumption that every domain in $\overline{\Aff(\Cb^d) \cdot \Kc} \cap \Xb_d$ has simple boundary. 
\end{proof}

\begin{lemma}\label{lem:Omega_R_any_direction} There exists $C_1 > 0$ such that: if $\Omega \in \Kc$, then
\begin{align*}
\delta_{\Omega}(z;v) \leq C_1 \delta_{\Omega}(z)^{1/m}
\end{align*}
for all $z \in \Omega^{(R)}$ and $v \in \Cb^d$ non-zero.
\end{lemma}

\begin{proof}
Define
\begin{align*}
M_1 :=\sup \left\{ \delta_{ \Omega}(0;v) : \Omega \in \Kb_d(r_0), v \in \Cb^d \setminus\{0\} \right\} < \infty
\end{align*}
(recall that $r_0 = \delta_0/(R+1)$). We claim that 
\begin{align*}
C_1 := \frac{\sqrt{2}M_1C_0}{r_0}
\end{align*}
suffices. 

Fix $\Omega \in \Kc$, $z \in \Omega^{(R)}$, and $v \in \Cb^d$ non-zero. Let $\xi = \pi_\Omega(z)$ and $H$ be a supporting hyperplane of $\Omega$ at $\xi$. Notice that 
\begin{align*}
r_0 \norm{\xi-0} \leq r_0(R+1) = \delta_0 \leq \delta_\Omega(0).
\end{align*}
So by Theorem~\ref{thm:normalizing}, there exists an affine map $A$ such that $A\Omega \in \Kb_d(r_0)$, $A(z)=0$, and if $\delta_{H} = \max\{ \delta_\Omega(z;v) : v\in -\xi + H \text{ non-zero}\}$, then
\begin{align}
\label{eqn:lower_bd_on_A}
\norm{A(z_1) - A(z_2)} \geq \frac{r_0}{\sqrt{2} \delta_{H}} \norm{z_1-z_2}
\end{align}
for all  $z_1, z_2 \in \Cb^d$.

By the previous Lemma 
\begin{align*}
\delta_{H} \leq C_0 \delta_\Omega(z)^{1/m}.
\end{align*}
Suppose $A(\cdot) = b + g(\cdot)$ where $g \in \GL_d(\Cb)$ and $b \in \Cb^d$. Then Equation~\eqref{eqn:lower_bd_on_A} implies
\begin{equation*}
\delta_\Omega(z;v) \leq \frac{\sqrt{2}\delta_{H}}{r_0}\delta_{A\Omega}(0;g(v)) \leq \frac{\sqrt{2}}{r_0}C_0 \delta_\Omega(z)^{1/m} M_1= C_1\delta_\Omega(z)^{1/m}. \qedhere
\end{equation*}

\end{proof}

\begin{lemma} There exists $C_2 > 0$ such that: if $\Omega \in \Kc$, then
\begin{align*}
\delta_{\Omega}(z;v) \leq C_2 \delta_{\Omega}(z)^{1/m}
\end{align*}
for all $z \in \Omega \cap \Bb_d(0;R)$ and $v \in \Cb^d$ non-zero.
\end{lemma}

\begin{proof} Let 
\begin{align*}
M_2 :=\sup \left\{ \delta_{ \Omega}(z;v) : \Omega \in \Kc, z \in \Omega \cap \Bb_d(0;R), v \in \Cb^d \setminus\{0\} \right\} < \infty.
\end{align*}
We claim that
\begin{align*}
C_2 = \max\left\{ C_1, \frac{M_2}{r_0^{1/m}}\right\}
\end{align*}
suffices. 

Fix $\Omega \in \Kc$, $z \in \Omega \cap \Bb_d(0;R)$, and $v \in \Cb^d$ non-zero. By the last lemma we only have to consider the case when $z \notin \Omega^{(R)}$. We consider two cases. 

\medskip
\noindent \textbf{Case 1:} $z=0$. Then 
\begin{align*}
\delta_{ \Omega}(0;v) \leq M_2 \leq \frac{M_2}{\delta_0^{1/m}}\delta_\Omega(0)^{1/m} \leq C_2 \delta_\Omega(0)^{1/m}
\end{align*}
since $r_0 =\delta_0/(R+1)< \delta_0$. 

\medskip
\noindent \textbf{Case 2:}  $\norm{\pi_\Omega(z)} > R+1$. Since $\overline{\Omega}$ contains the convex hull of $\Bb_d(0;\delta_0)$ and $\pi_\Omega(z)$, 
\begin{align*}
\delta_{ \Omega}(z) \geq \frac{\delta_0}{R+1}=r_0.
\end{align*} 
Then
\begin{equation*}
\delta_{\Omega}(z;v) \leq M_2 \leq \frac{M_2}{r_0^{1/m}} \delta_{\Omega}(z)^{1/m} \leq C_2 \delta_{\Omega}(z)^{1/m}. \qedhere
\end{equation*}
\end{proof}

This completes the proof of Theorem~\ref{thm:loc_m_convex}.

\section{$m$-convexity versus Gromov hyperbolicity}\label{sec:m_convex_vs_gromov}

As mentioned in Section~\ref{sec:suff_for_visible}, for smoothly bounded convex domains it is easy to show that $\Omega$ is $m$-convex for some $m$ if and only if $\partial \Omega$ has finite type. In particular, we have the following equivalences.

\begin{theorem}\cite[Theorem 1.1]{Z2016} Suppose $\Omega$ is a bounded convex domain with $C^\infty$ boundary. Then the following are equivalent: 
\begin{enumerate}
\item $\partial \Omega$ has finite type in the sense of D'Angelo,
\item $(\Omega, K_\Omega)$ is Gromov hyperbolic, 
\item $\Omega$ is $m$-convex for some $m\geq1$. 
\end{enumerate}
\end{theorem}

In the non-smooth case, Gromov hyperbolicity implies ``local'' $m$-convexity. 

\begin{corollary}\label{cor:m_convex_Gromov_hyp}
Suppose $\Omega$ is a $\Cb$-properly convex domain and $(\Omega, K_\Omega)$ is Gromov hyperbolic. Then for any $R > 0$ there exist $C>0$ and $m \geq 1$ such that 
\begin{align*}
\delta_{\Omega}(z; v) \leq C \delta_{\Omega}(z)^{1/m}
\end{align*}
for all $z \in \Bb_d(0;R) \cap \Omega$ and $v \in \Cb^d$ non-zero. 
\end{corollary}

\begin{proof}[Proof of Corollary~\ref{cor:m_convex_Gromov_hyp}]
This is a consequence of Theorem~\ref{thm:prior_nec} and Corollary~\ref{cor:m_convex}.
\end{proof}

However, as the next example shows, $m$-convexity does not, in general, imply Gromov hyperbolicity.

\begin{example}\label{ex:intersections} Let $\Omega_1, \dots, \Omega_d$ be bounded strongly convex domains with $C^\infty$ boundaries such that: $0 \in \partial \Omega_j$, the real hyperplane
\begin{align*}
\{ (z_1,\dots, z_d) \in \Cb^d : { \rm Re}(z_j) = 0\}
\end{align*}
 is  tangent to $\Omega_j$ at $0$, and 
 \begin{align*}
\Omega_j \subset \{ (z_1,\dots, z_d) \in \Cb^d : { \rm Re}(z_j) > 0\}.
\end{align*}
Define $\Omega= \cap_{j=1}^d \Omega_j$. Since each $\Omega_j$ has smooth boundary, we see that 
\begin{align*}
(\epsilon, \dots, \epsilon) \in \Omega
\end{align*}
for $\epsilon > 0$ sufficiently small. So $\Omega$ is non-empty. Further, since each $\Omega_j$ is strongly convex, there exists $C > 0$ such that 
\begin{align*}
\delta_{\Omega_j}(z;v) \leq C \delta_{\Omega_j}(z)^{1/2}
\end{align*}
for all $1 \leq j \leq d$, $z \in \Omega_j$, and $v \in \Cb^d$ non-zero. Then for $z \in \Omega$ and $v \in \Cb^d$ non-zero
\begin{align*}
\delta_{\Omega}(z;v)=\min_{1 \leq j\leq d} \delta_{\Omega_j}(z;v) \leq \min_{1 \leq j\leq d} C \delta_{\Omega_j}(z)^{1/2} = C \delta_{\Omega}(z)^{1/2}.
\end{align*}
So $\Omega$ is $2$-convex. However $n \cdot \Omega$ converges in the local Hausdorff topology to 
\begin{align*}
D =  \{ (z_1,\dots, z_d) \in \Cb^d : { \rm Re}(z_1) > 0, \dots,  { \rm Re}(z_d) > 0\}.
\end{align*}
Since $D$ does not have simple boundary, Theorem~\ref{thm:prior_nec} implies that $(\Omega,K_\Omega)$ is not Gromov hyperbolic.  
\end{example}

\section{The behavior of geodesics in a fixed domain}\label{sec:behavior_of_geodesics}

In this section we study the asymptotic behavior of geodesics in a fixed convex domain. Recall, from Definition~\ref{defn:end_comp}, that $\overline{\Omega}^{\End}$ denotes the end compactification of $\overline{\Omega}$. 

We first establish the following visibility result. 

\begin{proposition}\label{prop:geod_in_domain_1}
Suppose $\Omega$ is a $\Cb$-properly convex domain and every domain in $\overline{\Aff(\Cb^d) \cdot \Omega} \cap \Xb_d$ has simple boundary. Assume $\sigma_n : [a_n,b_n] \rightarrow \Omega$ is a sequence of geodesics such that 
\begin{align*}
\lim_{n \rightarrow \infty} \sigma_n(a_n) = \xi\in  \overline{\Omega}^{\End}
\end{align*}
and 
\begin{align*}
 \lim_{n \rightarrow \infty} \sigma_n(b_n) = \eta \in \overline{\Omega}^{\End}.
 \end{align*}
 If $\xi \neq \eta$, then exist sequences $n_j \rightarrow \infty$ and $T_j \in [a_{n_j},b_{n_j}]$ such that $\sigma_{n_j}(T_j)$ converges to a point in $\Omega$.
\end{proposition}

\begin{remark}\ \begin{enumerate}
\item Informally this proposition says that geodesics joining two distinct points in $\overline{\Omega}^{\End}$ ``bend'' into the domain. 
\item Notice that in Definition~\ref{defn:visibility_seqn} we consider the one point compactification of $\Cb^d$ while in Proposition~\ref{prop:geod_in_domain_1} we consider the end compactification of $\overline{\Omega}$. 
\end{enumerate}
 \end{remark}

\begin{proof} By Corollary~\ref{cor:simple_bd_implies_visibility} the constant sequence $\Omega, \Omega, \dots$ is a visibility sequence. Up to relabeling $\xi$ and $\eta$ it is enough to consider two cases:

\medskip
\noindent \textbf{Case 1:} $\xi \in \Cb^d$. In this case, the Proposition follows immediately from applying the visibility property to the geodesics $\sigma_n$. 

\medskip
\noindent \textbf{Case 2:} $\xi, \eta \notin \Cb^d$. Then there exists $R > 0$ such that $\sigma_n(a_n)$ and $\sigma_n(b_n)$ are in different connected components of $\overline{\Omega} \setminus \overline{\Bb_d(0;R)}$ for $n$ sufficiently large. So there exist $a_n^\prime \in [a_n,b_n]$ such that $\norm{\sigma_n(a_n^\prime)}\leq R$ when $n$ is sufficiently large.  Passing to a subsequence, we can assume that $\sigma_n(a_n^\prime) \rightarrow \xi^\prime \in \Cb^d$. Then we can apply the visibility property to the sequence of geodesics $\sigma_n|_{[a_n^\prime, b_n]}$. 
\end{proof} 

\begin{proposition}\label{prop:geod_in_domain_2} Suppose $\Omega$ is a $\Cb$-properly convex domain and every domain in $\overline{\Aff(\Cb^d) \cdot \Omega} \cap \Xb_d$ has simple boundary. If $\sigma: [0,\infty) \rightarrow \Omega$ is a geodesic ray, then 
\begin{align*}
\lim_{t \rightarrow \infty} \sigma(t)
\end{align*}
exists in $\partial_{\End} \Omega$.
\end{proposition}

\begin{proof} Suppose not, then there exist sequences $a_n \rightarrow \infty$ and $b_n \rightarrow \infty$ such that 
\begin{align*}
\lim_{n \rightarrow \infty} \sigma_n(a_n) = \xi\in \partial_{\End} \Omega
\end{align*}
and 
\begin{align*}
 \lim_{n \rightarrow \infty} \sigma_n(b_n) = \eta \in \partial_{\End} \Omega,
 \end{align*}
but $\xi \neq \eta$. By passing to subsequences we can suppose that $a_n \leq b_n$ for all $n$. Then by Proposition~\ref{prop:geod_in_domain_1} and passing to a subsequence there exist $T_n \in [a_n, b_n]$ such that $\sigma(T_n)$ converges to some $z_\infty \in \Omega$. Then 
\begin{align*}
\infty > K_\Omega(\sigma(0),z_\infty) = \lim_{n \rightarrow \infty} K_\Omega(\sigma(0), \sigma(T_n)) \geq \lim_{n \rightarrow \infty} a_n =\infty
\end{align*}
and we have a contradiction. 
\end{proof}

The final result of this section requires a definition. First recall, from Definition~\ref{defn:asym_cone}, that ${\rm AC}(\Omega)$ is the asymptotic cone of $\Omega$. 

\begin{definition}\label{defn:totally_real} \ \begin{enumerate}
\item A real linear subspace $V \subset \Cb^d$ is \emph{totally real} if $V \cap i V = \{0\}$. 
\item When $\Omega$ is a $\Cb$-properly convex domain, ${ \rm AC}(\Omega)$ is \emph{totally real} if 
\begin{align*}
\Span_{\Rb} { \rm AC}(\Omega)
\end{align*} is totally real. 
\end{enumerate}
\end{definition}

\begin{proposition}\label{prop:showing_well_behaved} Suppose $\Omega$ is a $\Cb$-properly convex domain and every domain in $\overline{\Aff(\Cb^d) \cdot \Omega} \cap \Xb_d$ has simple boundary. Further assume that
\begin{enumerate}
\item $\Omega$ is bounded or 
\item $\Omega$ is unbounded and $\rm{AC}(\Omega)$ is not totally real.
\end{enumerate}
If $\sigma: \Rb \rightarrow \Omega$ is a geodesic, then 
\begin{align*}
\lim_{t \rightarrow \infty} \sigma(t) \neq \lim_{t \rightarrow -\infty} \sigma(t)
\end{align*}
 in $\partial_{\End} \Omega$.
\end{proposition}

\begin{remark} \ \begin{enumerate}
\item If $\Bc = \{ x \in \Rb^d : \norm{x} < 1\}$ and $\Omega = \Bc + i\Rb^d$, then one can show that every domain in $\overline{\Aff(\Cb^d) \cdot \Omega} \cap \Xb_d$ has simple boundary, but there exists a geodesic $\sigma: \Rb \rightarrow \Omega$ with 
\begin{align*}
\lim_{t \rightarrow \infty} \sigma(t) = \lim_{t \rightarrow -\infty} \sigma(t) \in \partial_{\End} \Omega.
\end{align*}
Thus some extra assumption is necessary when $\Omega$ is unbounded. 
\item When $\Omega$ is unbounded and $\rm{AC}(\Omega)$ is not totally real, then $\overline{\Omega}^{\rm End}$ is simply the one-point compactification of $\overline{\Omega}$ (see Observation~\ref{obs:asymptotic_cone_3}).
\end{enumerate}
\end{remark}

\begin{proof} By Proposition~\ref{prop:geod_in_domain_2} both limits exist. Suppose for a contradiction that 
\begin{align*}
\xi := \lim_{t \rightarrow \infty} \sigma(t) = \lim_{t \rightarrow -\infty} \sigma(t) \in \partial_{\End}\Omega.
\end{align*}

\medskip

\noindent \textbf{Case 1:} $\xi \in \Cb^d$. Fix some $z_0 \in \Omega$ and let $z_n \in [z_0, \xi)$ be a sequence converging to $\xi$. By Theorem~\ref{thm:normalizing}, there exist $r > 0$ and affine maps $A_n \in \Aff(\Cb^d)$ such that $A_n\Omega \in \Kb_d(r)$, $A_n(z_n)=0$, and $A_n(\xi)=e_1$. Since $\Kb_d(r)$ is compact, we can pass to a subsequence and assume that $A_n \Omega$ converges to some $\Omega_\infty$ in $\Xb_d$. By Corollary~\ref{cor:simple_bd_implies_visibility}, the sequence $A_n \Omega$ is a visibility sequence. 

Consider the geodesics $\gamma_{1,n} : [0,\infty) \rightarrow A_n \Omega$ and $\gamma_{2,n} : [0,\infty) \rightarrow A_n \Omega$ given by $\gamma_{1,n}(t) = A_n \sigma(t)$ and $\gamma_{2,n}(t) = A_n\sigma(-t)$. Since $\Omega$ has simple boundary and $z_n \rightarrow \xi \in \partial \Omega$, we see that 
\begin{align*}
\lim_{n \rightarrow \infty} \max\left\{ \delta_\Omega(z_n;v) : v \in \Cb^d, \norm{v}=1\right\}=0.
\end{align*}
So by Theorem~\ref{thm:normalizing} part (5), 
\begin{align*}
\lim_{n \rightarrow \infty} \norm{A_n\sigma(0)} 
&= \lim_{n \rightarrow \infty} \norm{A_n\sigma(0)-0}= \lim_{n \rightarrow \infty} \norm{A_n\sigma(0)-A_n z_n} \\
& \geq \frac{r}{\sqrt{2}} \norm{\sigma(0)-\xi} \lim_{n \rightarrow \infty} \frac{1}{\max\{ \delta_\Omega(z_n;v) : \norm{v}=1\}} = \infty. 
\end{align*}
So
\begin{align*}
\lim_{n \rightarrow \infty} \norm{\gamma_{j,n}(0)} =\lim_{n \rightarrow \infty} \norm{A_n\sigma(0)} = \infty.
\end{align*}
Further, for any $n$ we have 
\begin{align*}
\lim_{t \rightarrow \infty} \gamma_{j,n}(t) = A_n(\xi) = e_1.
\end{align*}
So we can find $b_{1,n}, b_{2,n} \rightarrow \infty$ such that 
\begin{align*}
\lim_{n \rightarrow \infty} \gamma_{j,n}(b_{j,n}) = e_1.
\end{align*}
Since $A_n \Omega$ is a visibility sequence, after passing to a subsequence there exist $T_{j,n} \in [0,b_{j,n}]$ so that  $\lim_{n \rightarrow \infty} \gamma_{j,n}(T_{j,n}) = z_j \in \Omega_\infty$. Notice that since $\lim_{n \rightarrow \infty} \norm{\gamma_{j,n}(0)} = \infty$, the ``in particular'' part of Observation~\ref{obs:AA_for_geod} implies that  
\begin{align*}
\lim_{n \rightarrow \infty} T_{1,n} = \lim_{n \rightarrow \infty} T_{2,n} =\infty.
\end{align*}
But then Proposition~\ref{prop:convergence_of_kob} implies
\begin{align*}
\infty & > K_{\Omega_\infty}(z_1,z_2) = \lim_{n \rightarrow \infty} K_{A_n\Omega}(\gamma_{1,n}(T_{1,n}), \gamma_{2,n}(T_{2,n})) \\
& = \lim_{n \rightarrow \infty} K_{\Omega}(\sigma(T_{1,n}), \sigma(-T_{2,n})) = \lim_{n \rightarrow \infty} T_{1,n} + T_{2,n} = \infty
\end{align*}
and we have a contradiction. 

\medskip

\noindent \textbf{Case 2:} $\xi \notin \Cb^d$. Then $\Omega$ is unbounded and so $\rm{AC}(\Omega)$ is not totally real. This implies that there exists a complex line $L$ such that $L \cap \rm{AC}(\Omega)$ has non-empty interior in $L$. By applying an affine transformation to $\Omega$ we can assume that $L = \Cb \cdot e_1$, $\sigma(0) =0$, $e_1 \in \partial \Omega$, and
\begin{align*}
 \{ (x+iy, 0,\dots, 0) : x < 1 - \alpha \abs{y} \} \subset \Cb \cdot e_1 \cap \Omega
\end{align*}
for some $\alpha > 0$. 

Let $A_n \in \Aff(\Cb^d)$ be an affine map such that 
$$A_n(ze_1) =\left(1+\frac{1}{n}(z-1)\right)e_1.$$ 
Then $e_1=A_n(e_1)\in \partial A_n \Omega$ and 
\begin{align*}
 \{ (x+iy, 0,\dots, 0) : x < 1 - \alpha \abs{y} \} \subset \Cb \cdot e_1 \cap A_n\Omega.
 \end{align*}
 So there exists some $r > 0$ such that $A_n \Omega \cap \Cb \cdot e_1 \in \Kb_1(r)$ for all $n$. Then using Proposition~\ref{prop:compact_on_slices} we can assume that $A_n \Omega \in \Kb_d(r)$ for all $n$. Since $\Kb_d(r)$ is compact, we can pass to a subsequence and assume that $A_n \Omega$ converges to some $\Omega_\infty$ in $\Xb_{d}$. By Corollary~\ref{cor:simple_bd_implies_visibility}, the sequence $A_n \Omega$ is a visibility sequence.
 
 Consider the geodesics $\gamma_{1,n} : [0,\infty) \rightarrow A_n \Omega$ and $\gamma_{2,n} : [0,\infty) \rightarrow A_n \Omega$ given by $\gamma_{1,n}(t) = A_n \sigma(t)$ and $\gamma_{2,n}(t) = A_n\sigma(-t)$. By construction 
 \begin{align*}
 \lim_{n \rightarrow \infty} \gamma_{j,n}(0) = \lim_{n \rightarrow \infty} A_n(0) = \lim_{n \rightarrow \infty} \left(1-\frac{1}{n} \right) e_1 = e_1
 \end{align*}
 and 
  \begin{align*}
 \lim_{t \rightarrow \infty} \norm{\gamma_{j,n}(t) }=\infty
 \end{align*}
 for every $n$. Since $A_n \Omega$ is a visibility sequence, after passing to a subsequence there exist $T_{1,n}, T_{2,n} \in [0,\infty)$ so that $\lim_{n \rightarrow \infty} \gamma_{j,n}(T_{j,n})=z_j \in \Omega_\infty$. Notice that since $\lim_{n \rightarrow \infty} \gamma_{j,n}(0) = e_1 \in \partial \Omega_\infty$, the ``in particular'' part of Observation~\ref{obs:AA_for_geod} implies that  
\begin{align*}
\lim_{n \rightarrow \infty} T_{1,n} = \lim_{n \rightarrow \infty} T_{2,n} =\infty.
\end{align*}
 But then Proposition~\ref{prop:convergence_of_kob} implies
\begin{align*}
\infty & > K_{\Omega_\infty}(z_1,z_2) = \lim_{n \rightarrow \infty} K_{A_n\Omega}(\gamma_{1,n}(T_{1,n}), \gamma_{2,n}(-T_{2,n})) \\
& = \lim_{n \rightarrow \infty} K_{\Omega}(\sigma(T_{1,n}), \sigma(-T_{2,n})) = \lim_{n \rightarrow \infty} T_{1,n} + T_{2,n} = \infty
\end{align*}
and we have a contradiction. 
\end{proof}

\section{Proof of Theorem~\ref{thm:main_equivalence}}\label{sec:pf_of_thm_main_equivalence}

In this section we establish Theorem~\ref{thm:main_equivalence} by proving the following stronger result. 

\begin{theorem}\label{thm:main_equivalence_general_version} Suppose  $\Omega$ is $\Cb$-properly convex and  either
\begin{enumerate}
\item $\Omega$ is bounded or 
\item $\Omega$ is unbounded and $\rm{AC}(\Omega)$ is not totally real (see Definition~\ref{defn:totally_real}). 
\end{enumerate}
Then $(\Omega, K_\Omega)$ is Gromov hyperbolic if and only if every domain in 
\begin{align*}
\overline{\Aff(\Cb^d) \cdot \Omega} \cap \Xb_d
\end{align*}
has simple boundary. 
\end{theorem}

\begin{remark} If $\Bc = \{ x \in \Rb^d : \norm{x} < 1\}$ and $\Omega = \Bc + i\Rb^d$, then one can show that every domain in $\overline{\Aff(\Cb^d) \cdot \Omega} \cap \Xb_d$ has simple boundary. However, $\Bc$ is bounded and so $(\Omega, K_\Omega)$ is not Gromov hyperbolic by Corollary~\ref{cor:tube_domains}. Thus some extra assumption is necessary when $\Omega$ is unbounded. 
 \end{remark}

We need one lemma.

\begin{lemma}\label{lem:keep_asymptotic_cone} 
Suppose  $\Omega$ is $\Cb$-properly convex and either
\begin{enumerate}
\item $\Omega$ is bounded or 
\item $\Omega$ is unbounded and $\rm{AC}(\Omega)$ is not totally real.
\end{enumerate}
If $\mathcal{D} \in \overline{\Aff(\Cb^d) \cdot \Omega} \cap \Xb_d$, then either 
\begin{enumerate}
\item $\mathcal{D}$ is bounded or 
\item $\mathcal{D}$ is unbounded and ${\rm AC}(\mathcal{D})$ is not totally real.
\end{enumerate}
\end{lemma}

\begin{proof} Suppose that $\mathcal{D} \in \overline{\Aff(\Cb^d) \cdot \Omega} \cap \Xb_d$. Then there is a sequence $A_n \in \Aff(\Cb^d)$ such that $A_n \Omega \rightarrow \mathcal{D}$. We break the proof into two cases. 

\medskip

\noindent \textbf{Case 1:}  $\Omega$ is unbounded. Then $\rm{AC}(\Omega)$ is not totally real. Then, since ${\rm AC}(\Omega)$ is convex, there exists a complex line $L$ through $0$ such that $\mathcal{C}:=L \cap \rm{AC}(\Omega)$ is a convex cone with non-empty interior in $L$.

Suppose that $A_n(\cdot) = b_n + g_n(\cdot)$ for some $b_n \in \Cb^d$ and $g_n \in \GL_d(\Cb)$. Then ${ \rm AC}(A_n\Omega) = g_n {\rm AC}(\Omega)$. Since $g_n \in \GL_d(\Cb)$ and $\mathcal{C}$ is a one-dimensional cone, there exists a unitary matrix $u_n \in { \rm U}(d)$ such that $g_n \mathcal{C} = u_n \mathcal{C}$. By passing to a subsequence we can suppose that $u_n \rightarrow u \in { \rm U}(d)$. Then $u \mathcal{C} \subset { \rm AC}(\mathcal{D})$. So $\mathcal{D}$ is unbounded and ${ \rm AC}(\mathcal{D})$ is not totally real. 

\medskip

\noindent \textbf{Case 2:} $\Omega$ is bounded. Now fix some $z \in \mathcal{D}$. Then by passing to a tail of $(A_n)_{n \in \Nb}$, we can assume that $z \in A_n\Omega$ for all $n$. So if $z_n : = A_n^{-1}z$, then $A_n(\Omega, z_n)$ converges to $(\mathcal{D},z)$ in $\Xb_{d,0}$. By passing to a subsequence we can suppose that $z_n \rightarrow z^\prime \in \overline{\Omega}$. Now we consider two cases based on the location of $z^\prime$. 

\medskip 

\noindent \textbf{Case 2(a):} $z^\prime \in \Omega$. Then $(\Omega, z_n)$ converges to $(\Omega, z^\prime)$ in $\Xb_{d,0}$ and so by Proposition~\ref{prop:limit_domain_depends_on_sequence}
\begin{align*}
(\mathcal{D},z) = T(\Omega, z^\prime)
\end{align*}
for some $T \in \Aff(\Cb^d)$. Then $\mathcal{D} = T \Omega$ and so $\mathcal{D}$ is bounded.

\medskip

\noindent \textbf{Case 2(b):} $z^\prime \in \partial \Omega$. Fix some $z_0 \in \Omega$. For each $n$, let $L_n$ denote the complex line containing $z_0$ and $z_n$. Let $\xi_n \in \partial \Omega$ be the point of intersection with the ray $z_0 + \Rb_{>0}(z_n-z_0)$. Since $\overline{\Omega}$ contains the convex hull of $\Bb_d(z_0; \delta_\Omega(z_0))$ and $\xi_n$, there exist some $r > 0$ and $\theta \in (0,\pi/2]$, which are independent of $n$, such that 
\begin{align*}
\mathcal{C}_n := \{ z \in L_n : \norm{z-\xi_n} < r, \angle(z-\xi_n, z_n-\xi_n) < \theta\} \subset \Omega.
\end{align*}

Next let $B_n \in \Aff(\Cb^d)$ be an affine map such that $B_n(\xi_n)=e_1$ and $B_n(z_n) = 0$. Then, since $\mathcal{C}_n \subset \Omega$, we see that 
\begin{align*}
\left\{ ze_1 : \abs{z-1} < \frac{r}{r_n}, \, \angle(z-1, -1) < \theta\right\} \subset B_n\Omega
\end{align*}
where $r_n = \norm{z_n-\xi_n}$. In particular, there exists some $\epsilon > 0$, which is independent of $n$, such that 
\begin{align*}
B_n \Omega \cap \Cb \cdot e_1 \in \Kb_1(\epsilon). 
\end{align*}
But then, using Proposition~\ref{prop:compact_on_slices}, we can assume that $B_n \Omega \in \Kb_d(\epsilon)$. Then by passing to a subsequence we can suppose that $B_n(\Omega, z_n)=(B_n\Omega, 0)$ converges to some $(\mathcal{D}^\prime, 0)$ in  $\Xb_{d,0}$. Then by Proposition~\ref{prop:limit_domain_depends_on_sequence} there exists some $T \in \Aff(\Cb^d)$ such that $\mathcal{D} = T\mathcal{D}^\prime$. Finally since $r_n \rightarrow 0$ we see that
\begin{align*}
\left\{ ze_1 :  \angle(z-1, -1) < \theta\right\} \subset \mathcal{D}^\prime.
\end{align*}
So $\rm{AC}(\mathcal{D}^\prime)$, and hence $\rm{AC}(\mathcal{D})$, is not totally real.
 
\medskip

\end{proof}

\begin{proof}[Proof of Theorem~\ref{thm:main_equivalence_general_version}] If $(\Omega, K_\Omega)$ is Gromov hyperbolic, then Theorem~\ref{thm:prior_nec} implies that every domain in 
\begin{align*}
\overline{\Aff(\Cb^d) \cdot \Omega} \cap \Xb_d
\end{align*}
has simple boundary. 

Next suppose that every domain in 
\begin{align*}
\overline{\Aff(\Cb^d) \cdot \Omega} \cap \Xb_d
\end{align*}
has simple boundary. We will use Theorem~\ref{thm:prior_suff} to deduce that $(\Omega, K_\Omega)$ is Gromov hyperbolic. Fix a sequence $u_n \in \Omega$. By Theorem~\ref{thm:frankel_compactness} there exist sequences $n_k \rightarrow \infty$ and $A_k \in \Aff(\Cb^d)$ such that  $A_k(\Omega, u_{n_k})$ converges to some $(\Omega_\infty, u_\infty)$ in $\Xb_{d,0}$. By Lemma~\ref{lem:keep_asymptotic_cone} either 
\begin{enumerate}
\item $\Omega_\infty$ is bounded or 
\item $\Omega_\infty$ is unbounded and ${ \rm AC}(\Omega_\infty)$ is not totally real.
\end{enumerate}
Then Observation~\ref{obs:asymptotic_cone_3} implies that $\overline{\Omega}_\infty^{\End}$ coincides with either $\overline{\Omega}_\infty$ or the one point compactification of $\overline{\Omega}_\infty$. In either case we have an embedding $\overline{\Omega}_\infty^{\End} \hookrightarrow \Cb^d \cup\{\infty\}$. Then, since 
\begin{align*}
\overline{\Aff(\Cb^d) \cdot \Omega_\infty} \cap \Xb_d \subset \overline{\Aff(\Cb^d) \cdot \Omega} \cap \Xb_d,
\end{align*}
Proposition~\ref{prop:showing_well_behaved} implies that geodesics in $(\Omega_\infty, K_{\Omega_\infty})$ are well behaved. Further, Corollary~\ref{cor:simple_bd_implies_visibility} implies that $A_k\Omega$ is a visibility sequence. 

Then since $u_n \in \Omega$ was an arbitrary sequence, Theorem~\ref{thm:prior_suff} implies that $(\Omega, K_\Omega)$ is Gromov hyperbolic.

\end{proof}

\part{Subelliptic estimates}

\section{Prior work and the outline of the proof of Theorem~\ref{thm:intersection}}\label{sec:prior_outline_subelliptic}

We will use the following result of Straube in the proof of Theorem~\ref{thm:intersection}. 

\begin{theorem}[Straube \cite{S1997}]\label{thm:straube}
Suppose $\Omega$ is a bounded pseudoconvex domain in $\Cb^d$ and $\partial \Omega$ is the graph of a Lipschitz function near some $\xi \in \partial \Omega$. Assume that there exist $C_0>0$, $m > 2$, a neighborhood $U$ of $\xi$ in $\Cb^d$, and a bounded plurisubharmonic function $G: U \cap \Omega \rightarrow \Rb$ such that 
\begin{align*}
i\partial \bar{\partial} G(z) \geq \frac{C_0}{\delta_{\Omega}(z)^{2/m}} i\partial \bar{\partial} \norm{z}^2 \text{ on } U \cap \Omega
\end{align*}
as currents. Then there exists a neighborhood $V$ of $\xi$ and there exists a constant $C_1 >0$ such that 
\begin{align*}
\norm{u}_{\frac{1}{m}, V \cap \Omega} \leq C_1 ( \|\bar{\partial} u\|_0 + \|\bar{\partial}^* u\|_0)
\end{align*}
for all $u \in L^2_{(0,q)}(\Omega) \cap{ \rm dom}(\bar{\partial}) \cap { \rm dom}(\bar{\partial}^*)$ and $1 \leq q \leq d$. 
\end{theorem}

\begin{remark} For smoothly bounded pseudoconvex domains, Theorem~\ref{thm:straube} is due to Catlin~\cite[Theorem 2.2]{C1987}. 
\end{remark}

In the case of smoothly bounded convex domains with finite type in the sense of D'Angelo, McNeal~\cite{M1994} constructed functions satisfying the hypotheses of Theorem~\ref{thm:straube} (see ~\cite{M2002,NPT2013} for some corrections). We will construct such functions using a similar approach, however McNeal's work relies heavily on the smoothness of the boundary and in particular on properties of families of convex polynomials with bounded degree. In our proof, we replace McNeal's algebraic and analytic arguments with metric space arguments using the Gromov hyperbolicity assumption. Throughout the argument we also use the geometric estimates established in Section~\ref{sec:local_m_convexity}. 

The proof of Theorem~\ref{thm:intersection} has the following outline: 
\begin{enumerate}
\item In Section~\ref{sec:vm}, we recall the construction of ``visual metrics'' on the Gromov boundary of a Gromov hyperbolic metric space. 
\item In Section~\ref{sec:vm_and_normalizing}, we study how visual metrics behave under the normalizing maps defined in Section~\ref{sec:normalizing_maps}.
\item In Section~\ref{sec:psh_on_normal_domains}, we construct well behaved plurisubharmonic functions on normalized domains.
\item In Section~\ref{sec:psh_on_general_domains}, we use the results from the previous two sections to construct functions satisfying the hypothesis of Theorem~\ref{thm:straube}.
\item In Section~\ref{sec:pf_of_thm_intersection}, we prove Theorem~\ref{thm:intersection}. 
\item In Section~\ref{sec:optimal_estimate}, we explain the order of subelliptic estimate obtained by our argument. 
\end{enumerate}

The visual metric is analogous to the metric considered by McNeal in~\cite[Section 5]{M1994}. The normalizing maps are analogous to the ``polydisk coordinates'' considered by McNeal in~\cite[Section 3]{M1994}. The constructions in Sections~\ref{sec:psh_on_normal_domains} and~\ref{sec:psh_on_general_domains} are analogous to McNeal's constructions in~\cite[Propositions 3.1, 3.2]{M1994}. 

\section{Visual metrics}\label{sec:vm}

Suppose $(X,d)$ is a proper geodesic Gromov hyperbolic metric space. As in Section~\ref{sec:GH_basics}, let $\partial_GX$ be the Gromov boundary of $X$ and let $\overline{X}^G = X \cup \partial_GX$ denote the Gromov compactification. In this expository section we recall the construction of visual metrics on $\overline{X}^G$. 

\begin{theorem}\label{thm:visible_metric} There exist $C > 1$ and $\lambda > 0$ such that: For every $x_0 \in X$ there exists a function 
\begin{align*}
d_{x_0} : \overline{X}^G \times \overline{X}^G \rightarrow [0,\infty)
\end{align*}
with the following properties
\begin{enumerate}
\item $d_{x_0}(x,y) = d_{x_0}(y,x)$ for all $x,y \in \overline{X}^G$,
\item $d_{x_0}(x,y) \leq d_{x_0}(x,z) + d_{x_0}(z,y)$ for all $x,y,z \in \overline{X}^G$, and  
\item for all $x,y \in \overline{X}^G$
\begin{align*}
\frac{1}{C} e^{-\lambda d(x_0, \gamma_{x,y})} \leq d_{x_0}(x,y) \leq Ce^{-\lambda d(x_0, \gamma_{x,y})}
\end{align*}
where $\gamma_{x,y}$ is any geodesic in $(X,d)$ joining $x$ to $y$. 
\end{enumerate}
Moreover, $d_{x_0}$ restricts to a metric on $\partial_G X$ which generates the standard topology. 
\end{theorem}

\begin{remark} \ \begin{enumerate}
\item The function $d_{x_0}$ restricted to $\partial_G X$ is often called a \emph{visual metric}.
\item By definition, if $\gamma : [0,\infty) \rightarrow X$ is a geodesic ray, then 
$$
\lim_{t \rightarrow \infty} \gamma(t) \in \partial_G X
$$
exists and equals the equivalence class of $\gamma$. So in condition (3), if $x \in \partial_G X$, then $x= \lim_{t \rightarrow -\infty} \gamma_{x,y}(t)$. Likewise, if $y \in \partial_G X$, then $y= \lim_{t \rightarrow \infty} \gamma_{x,y}(t)$.
\item Condition (3) implies that $d_{x_0}(x,x) = 0$ if and only if $x \in \partial_G X$. Thus $d_{x_0}$ is not a metric on all of $\overline{X}^G$. To obtain a metric, one could define
\begin{align*}
\overline{d}_{x_0}(x,y) = \min\{ \lambda d(x,y), d_{x_0}(x,y) \}
\end{align*}
where $d(x,y) : = \infty$ when $x$ or $y$ is in $\partial_G X$. For a proof that this works see for instance~\cite[Section  3.6.3]{DSU2017}. 
\item If $(X,d)$ is $\delta$-hyperbolic (in the sense of Definition~\ref{defn:GH}), then any $0 < \lambda \leq \frac{1}{\delta}\log(2)$ satisfies Theorem~\ref{thm:visible_metric}, see the proof of Proposition 3.6.8 in~\cite{DSU2017}. 
\end{enumerate}
\end{remark}

We will sketch the standard construction of $d_{x_0}$. For more details and proofs, see for instance~\cite[Section 3.6.2]{DSU2017}. 

Recall that the \emph{Gromov product of $x,y,z \in X$} is defined to be 
\begin{align*}
(x|y)_z = \frac{1}{2} \left( d(x,z)+d(y,z) - d(x,y) \right).
\end{align*}
In a $\delta$-hyperbolic metric space, the Gromov product is, up to a bounded additive error, an easy to understand geometric quantity. 

\begin{observation} Suppose $\gamma : [a,b] \rightarrow X$ is a geodesic with $\gamma(a) = x$ and $\gamma(b)=y$, then 
\begin{align*}
d(z, \gamma)-2\delta \leq (x|y)_z \leq d(z, \gamma).
\end{align*}
\end{observation}

\begin{remark} The upper bound on $(x|y)_z$ holds for any metric space. \end{remark}

\begin{proof} The second inequality follows from the triangle inequality. To prove the first, pick $w$ in the image of $\gamma$ such that $(x|z)_w = (y|z)_w$. Notice that $(x|y)_w =0$. Since $(X,d)$ is $\delta$-hyperbolic
\begin{align*}
(x|z)_w=(y|z)_w=\min\{ (x|z)_w, (y|z)_w\} \leq \delta + (x|y)_w = \delta.
\end{align*}
A calculation shows that 
\begin{align*}
d(z,w) & = (x|y)_z +(x|z)_w+(y|z)_w - (x|y)_w \\
& = (x|y)_z +(x|z)_w+(y|z)_w
\end{align*}
and so
\begin{equation*}
d(z,\gamma) \leq d(z,w) = (x|y)_z +(x|z)_w+(y|z)_w \leq (x|y)_z + 2 \delta.  \qedhere
\end{equation*}
\end{proof}

Next we extend the Gromov product by taking limits. For $x_0 \in X$ and $(x,y) \in \overline{X}^G\times\overline{X}^G - X\times X$ we define
\begin{align*}
(x|y)_{x_0} := \begin{cases} 
\liminf_{x_n \rightarrow x} (x_n|y)_{x_0} & \text{if } x \in \partial_G X, \, y \in X \\
\liminf_{y_n \rightarrow y} (x|y_n)_{x_0} & \text{if } x \in X, \, y \in \partial_G X \\
\liminf_{x_n \rightarrow x, y_n \rightarrow y} (x_n|y_n)_{x_0} & \text{if } x,y \in \partial_G X
\end{cases}.
\end{align*}

This extension has the following properties. 

\begin{proposition}\label{prop:GP_basic_properties} Assume $x_0 \in X$. 
\begin{enumerate}
\item If $x,y \in \overline{X}^G$, then $(x|y)_{x_0} = \infty$ if and only if $x \in \partial_G X$ and $x=y$.
\item If $x \in \partial_GX$, then the sets 
$$
U_n(x,x_0) = \left\{ y \in \overline{X}^G : (x|y)_{x_0} > n\right\} \quad n=1,2,\dots 
$$
form a neighborhood basis of $x$. 
\item If $x,y \in X$ and $z \in \overline{X}^G$, then 
$$
\abs{ (x|z)_{x_0} - (y|z)_{x_0}} \leq d(x,y). 
$$
\end{enumerate}
\end{proposition}

\begin{proof} Parts (1) and (2) follow from the standard model of the Gromov boundary as equivalence classes of escaping sequences, see~\cite[Section 3.4.2]{DSU2017} or~\cite[Section 2]{KB2002}. Part (3) follows from the triangle inequality. 
\end{proof}

For $\lambda > 0$ sufficiently small define $\rho_{x_0} : \overline{X}^G \times \overline{X}^G \rightarrow [0,\infty)$ by
\begin{align*}
\rho_{x_0}(x,y) = \exp \left( -\lambda (x|y)_{x_0} \right).
\end{align*}
Finally the function $d_{x_0}$ is defined by 
\begin{align*}
d_{x_0}(x,y) = \inf \left\{ \sum_{j=1}^N \rho_{x_0}(x_j,x_{j+1}) : N > 0; x_1, \dots, x_{N+1} \in \overline{X}^G; x_1 = x, x_{N+1}=y\right\}.
\end{align*}
Miraculously, this yields a function which satisfies Theorem~\ref{thm:visible_metric}, see~\cite[Section 3.6.2]{DSU2017} for details. 

We end this discussion with some observations.

\begin{observation}\label{prop:continuity of visual metric} If $(x_n,y_n) \rightarrow (x,y)$ in $\overline{X}^G \times \overline{X}^G$, then 
$$
d_{x_0}(x,y) = \lim_{n \rightarrow \infty} d_{x_0}(x_n, y_n).
$$
\end{observation} 

\begin{proof} Notice that 
$$
\abs{d_{x_0}(x_n,y_n) - d_{x_0}(x,y)} \leq \abs{d_{x_0}(x_n,y_n) - d_{x_0}(x_n,y)}+\abs{d_{x_0}(x_n,y) - d_{x_0}(x,y)}. 
$$
We first prove that $ \abs{d_{x_0}(x_n,y_n) - d_{x_0}(x_n,y)}$ converges to zero. 

\medskip
\noindent \textbf{Case 1:} Assume $y \in X$. Then we can assume that $y_n \in X$ for all $n$. By the mean value theorem and Proposition~\ref{prop:GP_basic_properties}  part (3), we have 
$$
\abs{\rho_{x_0}(z,y_n) -\rho_{x_0}(z,y)} \leq \abs{(z|y_n)_{x_0} -(z|y)_{x_0}}  \leq d(y_n,y)
$$
for all $z \in \overline{X}^G$. Thus
$$
\lim_{n \rightarrow \infty} \abs{d_{x_0}(x_n,y_n) - d_{x_0}(x_n,y)} \leq \lim_{n \rightarrow \infty} d(y_n,y) = 0.
$$

\noindent \textbf{Case 2:} Assume $y \in \overline{X}^G$. Let $C >1$, $\lambda >0$ be the constants from Theorem~\ref{thm:visible_metric}. Then
$$
\lim_{n \rightarrow \infty}\abs{d_{x_0}(x_n,y_n) - d_{x_0}(x_n,y)} \leq \lim_{n \rightarrow \infty} d_{x_0}(y_n, y) \leq  \lim_{n \rightarrow \infty} Ce^{-\lambda (y_n|y)_{x_0} } = 0
$$ 
by Proposition~\ref{prop:GP_basic_properties}  part (2).

Thus in all cases
$$
\lim_{n \rightarrow \infty}\abs{d_{x_0}(x_n,y_n) - d_{x_0}(x_n,y)}=0.
$$
The same argument shows that 
$$
\lim_{n \rightarrow \infty}\abs{d_{x_0}(x_n,y) - d_{x_0}(x,y)}=0
$$
and hence the proof is complete.

\end{proof} 

As an immediate corollary we obtain:

\begin{observation}\label{obs:the sets V are open} If $\xi \in \partial_G X$ and $r > 0$, then the set 
$$
V_{x_0}(\xi;r) :=\left\{ x \in \overline{X}^G : d_{x_0}(\xi,x) < r\right\}
$$
is an open neighborhood of $\xi$ in $\overline{X}^G$. 
\end{observation}

\section{Visual metrics and normalizing maps}\label{sec:vm_and_normalizing}

For the rest of the section, let $\Omega \subset \Cb^d$ be a $\Cb$-properly convex domain with Gromov hyperbolic Kobayashi metric. Then fix some $z_0 \in \Omega$ and some $R > \norm{z_0}$.

Let $d_{z_0}$ denote the function constructed in  Theorem~\ref{thm:visible_metric} for the metric space $(\Omega, K_\Omega)$. Using Theorem~\ref{thm:compactification} we can view $d_{z_0}$ as a function on $\overline{\Omega}^{\End} \times \overline{\Omega}^{\End}$. Let $C_v > 1$ and $\lambda > 0$ be constants such that: for all $x,y \in \overline{\Omega}^{\End}$
\begin{align*}
\frac{1}{C_v} \exp \Big(-\lambda K_\Omega(z_0, \gamma_{x,y})\Big) \leq d_{z_0}(x,y) \leq C_v\exp \Big( -\lambda K_\Omega(z_0, \gamma_{x,y}) \Big)
\end{align*}
when $\gamma_{x,y}$ is a geodesic in $(\Omega, K_\Omega)$ joining $x$ to $y$. Then for $\xi \in \partial \overline{\Omega}^{\End}$ and $r > 0$ define 
\begin{align*}
V_{z_0}(\xi;r) := \left\{ z \in \overline{\Omega}^{\End} : d_{z_0}(\xi, z) < r\right\}.
\end{align*}
The goal of this section is to relate these sets to the normalizing maps constructed in Section~\ref{sec:normalizing_maps}. To that end,  we make the following definitions. 

\begin{definition}\label{defn:visual_metric_normalizing_map}
For $\xi \in \partial \Omega$ and $\epsilon \in (0,1)$, let $q_{\xi, \epsilon} \in [z_0, \xi)$ denote the unique point where
\begin{align*}
K_\Omega(q_{\xi, \epsilon},z_0) = \frac{1}{\lambda}\log\frac{1}{\epsilon}
\end{align*}
and 
\begin{align*}
K_\Omega(q^\prime,z_0) > \frac{1}{\lambda}\log\frac{1}{\epsilon}
\end{align*}
for every $q^\prime \in (q_{\xi, \epsilon},\xi)$. Then let $A_{\xi,\epsilon}$ denote an affine map satisfying Theorem~\ref{thm:normalizing} with $A_{\xi,\epsilon}(q_{\xi, \epsilon})=0$ and $A_{\xi,\epsilon}(\xi)=e_1$.
\end{definition}

In this section we will establish the following four propositions about these normalizing maps and their relationship with the visual metric. We will list the propositions in order of importance, but prove them in a different order. 

\begin{proposition}\label{prop:comparison} There exist $\epsilon_0 \in (0,1]$ and an increasing function $\tau : (0,\infty) \rightarrow (0,\infty)$ with
\begin{align*}
\lim_{r \searrow 0} \tau(r) = 0
\end{align*}
such that: if $\xi\in \partial \Omega \cap \Bb_d(0;R)$, $r > 0$, and $\epsilon \in (0,\epsilon_0/r) \cap (0,1)$, then
\begin{align*}
\overline{\Omega} \cap A^{-1}_{\xi,\epsilon} \Bb_d(e_1;r) \subset V_{z_0}(\xi; \tau(r) \epsilon)
\end{align*}
and
\begin{align*}
V_{z_0}(\xi; r\epsilon) \subset \overline{\Omega} \cap A^{-1}_{\xi,\epsilon} \Bb_d(e_1;\tau(r))
\end{align*}
for every $r > 0$. 
\end{proposition}

\begin{proposition}\label{prop:doubling_proposition} There exists $L \geq 1$ such that: If $S \geq 1$, $\xi \in \partial \Omega \cap \Bb_d(0;R)$, and $\epsilon \in \left(0,\frac{\epsilon_0}{S}\right)$, then 
\begin{align*}
V_{z_0}(\xi;S\epsilon) \subset \xi+ L\tau(S)\Big( V_{z_0}(\xi;\epsilon) - \xi\Big).
\end{align*}
\end{proposition}

\begin{proof}[Proof of Proposition~\ref{prop:doubling_proposition} assuming Proposition~\ref{prop:comparison}] Fix $r \in (0,\tau(1)]$ with $\tau(r) \leq 1$. Notice that $r \leq 1$: if $\tau(1) \leq 1$, then $r \leq \tau(1) \leq 1$ and if $\tau(1) > 1$, then $r < 1$ since $\tau(r) \leq 1$ and $\tau$ is increasing. Then let $L = \frac{1}{r}$. 

Fix $S \geq 1$ and $\epsilon \in \left(0,\frac{\epsilon_0}{S}\right)$. Since $\frac{\tau(S)}{r} \geq \frac{\tau(1)}{r} \geq 1$, $\Omega$ is convex, and $\xi \in \overline{\Omega}$, we have 
$$
\overline{\Omega} \subset \xi + \frac{\tau(S)}{r} \left( \overline{\Omega} - \xi\right).
$$
Since $r \leq 1 \leq S$, we have $\epsilon \in  \left(0,\frac{\epsilon_0}{r}\right)$. Then
\begin{align*}
V_{z_0}(\xi;S\epsilon)&  \subset  \overline{\Omega} \cap A^{-1}_{\xi,\epsilon} \Bb_d(e_1;\tau(S)) \subset \xi + \frac{\tau(S)}{r}\Big(\overline{\Omega} \cap A^{-1}_{\xi,\epsilon} \Bb_d(e_1;r)-\xi\Big) \\
& \subset \xi+ \frac{\tau(S)}{r}\Big( V_{z_0}(\xi;\tau(r)\epsilon) - \xi\Big) \subset \xi+ \frac{\tau(S)}{r}\Big( V_{z_0}(\xi;\epsilon) - \xi\Big).
\end{align*}

\end{proof}

\begin{proposition}\label{prop:visual_metrics_along_line} There exist $\alpha \geq 1$, $B \geq 1$ such that: if $\xi\in \partial \Omega \cap \Bb_d(0;R)$ and $\epsilon \in (0,1)$, then
\begin{align*}
\frac{1}{B} \epsilon^{2/\lambda} \leq \delta_\Omega(q_{\xi,\epsilon}) \leq B\epsilon^{2/(\alpha \lambda)}.
\end{align*}
Moreover, if $q \in [z_0, \xi)$, then 
\begin{align*}
q \in V_{z_0}\left(\xi;B\norm{q-\xi}^{\lambda/2}\right).
\end{align*}
\end{proposition}

\begin{remark} In the special case when $\partial \Omega$ is a $C^{2}$ hypersurface, one can choose $\alpha=1$.  \end{remark}

\begin{proposition}\label{prop:normalizing_maps_visual_metric} There exist $r_0 \in (0,1)$, $m_1 >0 $, $C_0 >0$ such that: if $\xi\in \partial \Omega \cap \Bb_d(0;R)$ and $\epsilon \in (0,1)$, then
\begin{align*}
A_{\xi,\epsilon}\Omega \in \Kb_d(r_0)
\end{align*}
and 
\begin{align*}
\norm{A_{\xi,\epsilon}(z_1) - A_{\xi,\epsilon}(z_2)} \geq \frac{C_0}{\epsilon^{1/m_1}} \norm{z_1-z_2}
\end{align*}
for all $z_1,z_2 \in \Cb^d$.
\end{proposition}

Proposition~\ref{prop:doubling_proposition} should be compared to~\cite[Proposition 2.5]{M1994} and Proposition~\ref{prop:normalizing_maps_visual_metric} should be compared to~\cite[Equation (2.7)]{M1994}.

\subsection{Proof of Proposition~\ref{prop:visual_metrics_along_line}} Let $\delta_0:=\delta_\Omega(z_0)$. If $\xi \in  \partial \Omega \cap \Bb_d(0;R)$ and $q \in [z_0, \xi)$, then 
\begin{align}
\label{eq:one_in_prop_metrics_along_line}
\frac{\delta_0}{2R}\norm{q-\xi} \leq \delta_\Omega(q) \leq \norm{q-\xi}
\end{align}
since $\overline{\Omega}$ contains the convex hull of $\Bb_d(z_0;\delta_0)$ and $\xi$.

By Proposition~\ref{prop:quasi_geodesic} there exist $\alpha_0 \geq 1$, $\beta_0 \geq 0$ such that: if $\xi \in \partial \Omega \cap \Bb_d(0;R)$, then the curve $\sigma_\xi: [0,\infty) \rightarrow \Omega$ given by 
\begin{align*}
\sigma_\xi(t) = \xi + e^{-2t}\left( z_0-\xi \right)
\end{align*}
is an $(\alpha_0,\beta_0)$-quasi-geodesic.

\begin{lemma}\label{lem:dist_est_normalizing_maps}  There exist $\alpha \geq 1$, $\beta > 0$ such that: if $\xi \in \partial\Omega \cap \Bb_d(0;R)$ and $q \in [z_0, \xi)$, then 
\begin{align*}
-\beta + \frac{1}{2} \log \frac{1}{\delta_\Omega(q)} \leq K_\Omega(q,z_0) \leq \beta + \frac{\alpha}{2} \log \frac{1}{\delta_\Omega(q)}.
\end{align*}
\end{lemma}

\begin{remark} The proof below shows that $\alpha = \alpha_0$ satisfies the lemma, however this may not be the optimal choice. \end{remark}

\begin{proof} Fix $\xi \in \partial \Omega \cap \Bb_d(0;R)$ and $q \in [z_0,\xi)$. Then $q = \sigma_\xi(t)$ where 
\begin{align*}
t = \frac{1}{2} \log \frac{\norm{z_0-\xi}}{\norm{q-\xi}}.
\end{align*}
So 
\begin{align*}
K_\Omega(z_0, q) =K_\Omega(\sigma_\xi(0), \sigma_\xi(t)) \leq \alpha_0 t + \beta_0 \leq \beta_0  + \frac{\alpha_0}{2} \log(2R) + \frac{\alpha_0}{2} \log \frac{1}{\norm{q-\xi}}.
\end{align*}
Thus Equation~\eqref{eq:one_in_prop_metrics_along_line} implies that 
\begin{align*}
K_\Omega(z_0, q) \leq  \beta_0  + \frac{\alpha_0}{2} \log(2R) + \frac{\alpha_0}{2} \log \frac{1}{\delta_\Omega(q)}.
\end{align*}
For the lower bound, Lemma~\ref{lem:hyperplanes_along_lines} and Equation~\eqref{eq:one_in_prop_metrics_along_line} imply
\begin{align*}
K_\Omega(z_0, q) \geq  \frac{1}{2} \log \frac{\norm{z_0-\xi}}{\norm{q-\xi}} \geq \frac{1}{2}\log \frac{\delta_0^2}{2R} + \frac{1}{2} \log \frac{1}{\delta_\Omega(q)}.
\end{align*}
So $\alpha = \alpha_0$ and
$$
\beta = \max\left\{ -\frac{1}{2}\log \frac{\delta_0^2}{2R}, \beta_0  + \frac{\alpha_0}{2} \log(2R) \right\}
$$
suffice. 
\end{proof}

\begin{proof}[Proof of Proposition~\ref{prop:visual_metrics_along_line}] 

Since
\begin{align*}
K_\Omega(q_{\xi,\epsilon},z_0) = \frac{1}{\lambda}\log\frac{1}{\epsilon},
\end{align*}
the last lemma implies that 
\begin{align*}
e^{-2\beta} \epsilon^{2/\lambda} \leq \delta_\Omega(q_{\xi,\epsilon}) \leq e^{2\beta/\alpha} \epsilon^{2/(\alpha \lambda)}.
\end{align*}
This proves the first part of the Proposition. 

Now fix some $q \in [z_0,\xi)$. Then $q = \sigma_\xi(t_0)$ where 
\begin{align*}
t_0 = \frac{1}{2} \log \frac{\norm{z_0-\xi}}{\norm{q-\xi}} .
\end{align*} 
Fix a sequence $t_0 < t_1 < t_2 < \dots$ converging to $\infty$ and for each $n$ let $\gamma_n :[0,b_n] \rightarrow \Omega$ be a geodesic joining $q=\sigma_\xi(t_0)$ to $\sigma_\xi(t_n)$. Then by Theorem~\ref{thm:morse_lemma} there exists $M > 0$, which does depend on $q$, such that 
\begin{align}
\label{eq:haus_dist_geodesics}
\max\left\{ \max_{t \in [t_0,t_n]} K_\Omega\left(\sigma_\xi(t), \gamma_n\right), \max_{t \in [0,b_n]} K_\Omega\left(\gamma_n(t), \sigma_\xi|_{[t_0,t_n]}\right) \right\} \leq M
\end{align}
for all $n \geq 1$. 

Using the Arzel\`a-Ascoli theorem and passing to a subsequence we can suppose that $\gamma_n$ converges to a geodesic ray $\gamma: [0,\infty) \rightarrow \Omega$. By the definition of the Gromov boundary and Theorem~\ref{thm:compactification}, we have
\begin{align*}
\lim_{t \rightarrow \infty} \gamma(t) = \xi.
\end{align*}

Equation~\eqref{eq:haus_dist_geodesics} implies that 
\begin{align*}
\max\left\{ \sup_{t \geq t_0} K_\Omega(\sigma_\xi(t), \gamma), \sup_{t \geq 0} K_\Omega\left(\gamma(t), \sigma_\xi|_{[t_0,\infty)}\right) \right\} \leq M.
\end{align*}
Hence
\begin{align*}
K_\Omega(z_0, \gamma) \geq -M + K_\Omega\left(z_0, \sigma_{\xi}|_{[t_0,\infty)}\right).
\end{align*}
But by Lemma~\ref{lem:hyperplanes_along_lines}
\begin{align*}
K_\Omega(z_0, \sigma_{\xi}(t)) \geq \frac{1}{2} \log \frac{ \norm{z_0-\xi}}{\norm{\sigma_\xi(t) - \xi}} = t
\end{align*}
for all $t \geq 0$. And so 
\begin{align*}
K_\Omega(z_0, \gamma) &\geq -M +t_0 = -M + \frac{1}{2} \log \frac{\norm{z_0-\xi}}{\norm{q-\xi}} \\
& \geq -M + \frac{1}{2} \log \delta_0 + \frac{1}{2} \log \frac{1}{\norm{q-\xi}}.
\end{align*}
Then 
\begin{equation*}
d_{z_0}(q, \xi) \leq C_v \exp \left( -\lambda K_\Omega(z_0, \gamma) \right) \leq C_v \exp\left( \lambda M\right) \delta_0^{-\lambda/2}  \norm{q-\xi}^{\lambda/2}. 
\end{equation*}
Thus $B= \max\left\{ e^{2\beta}, C_v\exp\left( \lambda M\right) \delta_0^{-\lambda/2}\right\}$ suffices. 
\end{proof}

\subsection{Proof of Proposition~\ref{prop:normalizing_maps_visual_metric}}

Fix some $\xi\in \partial \Omega \cap \Bb_d(0;R)$ and $\epsilon \in (0,1)$. Then
\begin{align*}
r_0 \norm{\xi-z_0} \leq \delta_\Omega(z_0)
\end{align*}
where $r_0 := \delta_\Omega(z_0)/(2R)$. So $A_{\xi,\epsilon} \Omega \in \Kb_d(r_0)$ by Theorem~\ref{thm:normalizing} part (1).

By Corollary~\ref{cor:m_convex_Gromov_hyp}, there exist $m_0 > 0$ and $c_0 > 0$ such that 
\begin{align*}
\delta_\Omega(z;v)  \leq c_0 \delta_\Omega(z)^{1/m_0}
\end{align*}
for every $z \in \Omega \cap \Bb_d(0;R)$ and $v \in \Cb^d$ non-zero. Since $q_{\xi,\epsilon} \in (z_0,\xi)$ and $z_0, \xi \in \Bb_d(0;R)$ we see that $q_{\xi,\epsilon} \in \Bb_d(0;R)$. So by Theorem~\ref{thm:normalizing} part (5) 
\begin{align*}
\norm{A(z_1) - A(z_2)} \geq \frac{r_0}{\sqrt{2} c_0 \delta_\Omega(q_{\xi,\epsilon})^{1/m_0}} \norm{z_1-z_2}
\end{align*}
for all  $z_1, z_2 \in \Cb^d$. Hence by Proposition~\ref{prop:visual_metrics_along_line}
\begin{align*}
\norm{A(z_1) - A(z_2)} \geq \frac{r_0}{\sqrt{2} c_0 \delta_\Omega(q_{\xi,\epsilon})^{1/m_0}} \norm{z_1-z_2} \geq \frac{r_0}{\sqrt{2}c_0B^{1/m_0} \epsilon^{2/(\alpha \lambda m_0)}} \norm{z_1-z_2}
\end{align*}
for all  $z_1, z_2 \in \Cb^d$. So $C_0 = \frac{r_0}{\sqrt{2}c_0B^{1/m_0}}$ and
\begin{align}
\label{eq:m_1_equality}
m_1 = \frac{\alpha \lambda m_0}{2}
\end{align}
suffice.

\subsection{Proof of Proposition~\ref{prop:comparison}} We begin by defining $\epsilon_0 \in (0,1]$. If $\Omega$ is bounded, let $\epsilon_0 = 1$. If $\Omega$ is unbounded, define $\epsilon_0$ to be the minimum of 1 and
\begin{align*}
\frac{1}{2} \min\left\{ d_{z_0}(\xi,\eta) : \xi \in \partial \Omega \cap \overline{\Bb_d(0;R)}, \ \eta \in \overline{\Omega}^{\End} \setminus \Cb^d \right\}
\end{align*}
(notice that this number exists by Proposition~\ref{prop:continuity of visual metric}). Then 
\begin{equation}
\label{eqn:Vz0 is in C}
\overline{V_{z_0}(\xi, \epsilon_0)} \subset \overline{\Omega} \subset \Cb^d
\end{equation}
for all $\xi \in \partial\Omega \cap \overline{\Bb_d(0;R)}$. 

The proposition will follow from a series of lemmas. 

\begin{lemma}\label{lem:D1} For any $r > 0$ there exists $D_1(r) < \infty$ such that:  if $\xi \in \partial\Omega \cap \Bb_d(0;R)$, $\epsilon \in (0,1)$, and $\gamma: [a,b] \rightarrow \Omega$ is a geodesic with $\gamma(a),\gamma(b) \in A^{-1}_{\xi,\epsilon}\Bb_d(e_1;r)$, then 
\begin{align*}
\gamma \subset A^{-1}_{\xi,\epsilon}\Bb_d(e_1;D_1(r)). 
\end{align*}
\end{lemma}

\begin{remark} This lemma says that a geodesic segment that starts and ends close to $e_1$ in $A_{\xi,\epsilon}\Omega$ stays close to $e_1$. \end{remark}

\begin{proof}Suppose for a contradiction that such a $D_1(r)$ does not exist for some $r > 0$. Then for each $n \in \Nb$ there exist $\xi_n \in \partial \Omega \cap \Bb_d(0;R)$, $\epsilon_n \in (0,1)$, a geodesic $\gamma_n: [a_n,b_n] \rightarrow \Omega$, and $t_n \in [a_n,b_n]$ where $\gamma_n(a_n),\gamma_n(b_n) \in A^{-1}_{\xi_n,\epsilon_n}\Bb_d(e_1;r)$ and
\begin{align*}
n \leq \norm{A_{\xi_n,\epsilon_n}\gamma_n(t_n)-e_1}.
\end{align*}
By Proposition~\ref{prop:normalizing_maps_visual_metric} each $\Omega_n : =A_{\xi_n, \epsilon_n}\Omega$ is in $\Kb_d(r_0)$, so by passing to a subsequence we can suppose that $\Omega_n$ converges to some $\Omega_\infty \in \Kb_d(r_0)$. Then Corollary~\ref{cor:simple_bd_implies_visibility} implies that $\Omega_n$ is a visibility sequence. 

Consider the geodesics $\wh{\gamma}_{n,1}: = A_{\xi_n, \epsilon_n}\gamma_n|_{[a_n,t_n]}$ and $\wh{\gamma}_{n,2}: = A_{\xi_n, \epsilon_n}\gamma_n|_{[t_n,b_n]}$. Notice that $\wh{\gamma}_{n,1}(a_n), \wh{\gamma}_{n,2}(b_n) \in \Bb_d(e_1;r)$ and 
\begin{align*}
\norm{\wh{\gamma}_{n,1}(t_n)}=\norm{\wh{\gamma}_{n,2}(t_n)}=\norm{A_{\xi_n,\epsilon_n}\gamma_n(t_n)}\geq n-1. 
\end{align*}
So using the fact that $\Omega_n$ is a visibility sequence, we can pass to subsequences and find $T_{n,1} \in [a_n, t_n]$ and $T_{n,2} \in [t_n,b_n]$ such that  $\wh{\gamma}_{n,1}(T_{n,1}) \rightarrow z_1 \in \Omega_\infty$ and $\wh{\gamma}_{n,2}(T_{n,2})\rightarrow z_2 \in \Omega_\infty$. Since 
\begin{align*}
\lim_{n \rightarrow \infty} \norm{\wh{\gamma}_{n,2}(t_n)}=\infty,
\end{align*}
the ``in particular'' part of Observation~\ref{obs:AA_for_geod} implies that  
\begin{align*}
 \lim_{n \rightarrow \infty} T_{n,2}-T_{n,1} \geq \lim_{n \rightarrow \infty} T_{n,2}-t_n =\infty.
\end{align*}
Then Proposition~\ref{prop:convergence_of_kob} implies that
\begin{align*}
\infty > K_{\Omega_\infty}(z_1,z_2) & = \lim_{n \rightarrow \infty} K_{\Omega_n}(\wh{\gamma}_{n,1}(T_{n,1}),\wh{\gamma}_{n,2}(T_{n,2})) = \lim_{n \rightarrow \infty} K_{\Omega}(\gamma_n(T_{n,1}), \gamma_n(T_{n,2})) \\
& = \lim_{n \rightarrow \infty} T_{n,2}-T_{n,1} = \infty.
\end{align*}
So we have a contradiction. \end{proof}

\begin{lemma}\label{lem:D1part2} We can assume that $D_1$ is an increasing function with
\begin{align*}
\lim_{r \searrow 0} D_1(r) = 0.
\end{align*}

\end{lemma} 

\begin{proof} For $r > 0$ fixed, let $\wt{D}_1(r)$ be the infimum of all numbers satisfying Lemma~\ref{lem:D1}.  Notice that $\wt{D}_1(r)$ itself may not satisfy the lemma and so we define $D_1(r): = r + \wt{D}_1(r)$. Then, by definition, $\wt{D}_1$ is non-decreasing and so $D_1$ is increasing. Further, $D_1$ satisfies Lemma~\ref{lem:D1}. 

Suppose that $\lim_{r \searrow 0} D_1(r)$ does not equal zero. Then there exists $D_0 > 0$ such that: for each $n \in \Nb$ there exist $\xi_n \in \partial \Omega \cap \Bb_d(0;R)$, $\epsilon_n \in (0,1)$, a geodesic $\gamma_n: [a_n,b_n] \rightarrow \Omega$, and $t_n \in [a_n,b_n]$ where $\gamma_n(a_n),\gamma_n(b_n) \in A^{-1}_{\xi_n,\epsilon_n}\Bb_d(e_1;1/n)$ and
\begin{align*}
D_0 \leq \norm{A_{\xi_n,\epsilon_n}\gamma_n(t_n)-e_1}.
\end{align*}
By definition
\begin{align*}
\norm{A_{\xi_n,\epsilon_n}\gamma_n(t_n)-e_1} \leq D_1(1/n) \leq D_1(1). 
\end{align*}

Now $\Omega_n : =A_{\xi_n, \epsilon_n}\Omega$ is in $\Kb_d(r_0)$, so by passing to a subsequence we can suppose that $\Omega_n$ converges to some $\Omega_\infty \in \Kb_d(r_0)$.  Then Corollary~\ref{cor:simple_bd_implies_visibility} implies that $\Omega_n$ is a visibility sequence. By passing to another subsequence we can suppose that 
\begin{align*}
\lim_{n \rightarrow \infty} A_{\xi_n,\epsilon_n}\gamma_n(t_n) = \eta \in \overline{\Omega}_\infty \cap \overline{\Bb_d(e_1;D_1(1))} \setminus \Bb_d(e_1;D_0).
\end{align*}
We divide the proof into two cases based on the location of $\eta$. 

\medskip

\noindent \textbf{Case 1:} $\eta \in \partial \Omega_\infty$. Consider the geodesics $\wh{\gamma}_{n,1}: = A_{\xi_n, \epsilon_n}\gamma_n|_{[a_n,t_n]}$ and $\wh{\gamma}_{n,2}: = A_{\xi_n, \epsilon_n}\gamma_n|_{[t_n,b_n]}$. Notice that 
\begin{align*}
\lim_{n \rightarrow \infty} \wh{\gamma}_{n,1}(a_n)=e_1= \lim_{n \rightarrow \infty} \wh{\gamma}_{n,2}(b_n)
\end{align*}
and 
\begin{align*}
\lim_{j \rightarrow \infty} \wh{\gamma}_{n,1}(t_n)=\eta= \lim_{n \rightarrow \infty} \wh{\gamma}_{n,2}(t_n).
\end{align*}
Since $\norm{\eta-e_1}\geq D_0$ and $\Omega_n$ is a visibility sequence, we can pass to a subsequence and find $T_{n,1} \in [a_n, t_n]$ and $T_{n,2} \in [t_n,b_n]$ such that  $\wh{\gamma}_{n,1}(T_{n,1}) \rightarrow z_1 \in \Omega_\infty$ and $\wh{\gamma}_{n,2}(T_{n,2})\rightarrow z_2 \in \Omega_\infty$. Since $\eta \in \partial \Omega_\infty$, the ``in particular'' part of Observation~\ref{obs:AA_for_geod} implies that  
\begin{align*}
 \lim_{n \rightarrow \infty} T_{n,2}-T_{n,1} \geq \lim_{n \rightarrow \infty} T_{n,2}-t_n=\infty.
\end{align*}
Then Proposition~\ref{prop:convergence_of_kob} implies
\begin{align*}
\infty > K_{\Omega_\infty}(z_1,z_2) & = \lim_{n \rightarrow \infty} K_{\Omega_j}(\wh{\gamma}_{n,1}(T_{n,1}) ,\wh{\gamma}_{n,2}(T_{n,2}) ) = \lim_{n \rightarrow \infty} K_{\Omega}(\gamma_n(T_{n,1}), \gamma_n(T_{n,2})) \\
& = \lim_{n \rightarrow \infty} T_{n,2}-T_{n,1} = \infty. 
\end{align*}
So we have a contradiction. 

\medskip

\noindent \textbf{Case 2:} $\eta \in \Omega_\infty$. Using Observation~\ref{obs:AA_for_geod}, Lemma~\ref{lem:limits_of_geod}, and passing to a subsequence, we can assume that the geodesics $A_{\xi_n,\epsilon_n}\gamma_n(\cdot + t_n)$ converges locally uniformly to a geodesic $\wh{\gamma} : \Rb \rightarrow \Omega_\infty$ where
\begin{align*}
\lim_{t \rightarrow -\infty} \wh{\gamma}(t) = \lim_{n \rightarrow \infty} A_{\xi_n,\epsilon_n}\gamma_n(a_n) = e_1
\end{align*}
and 
\begin{align*} 
\lim_{t \rightarrow \infty} \wh{\gamma}(t) = \lim_{n \rightarrow \infty} A_{\xi_n,\epsilon_n}\gamma_n(b_n) = e_1.
\end{align*}
Since 
\begin{align*}
\Omega_\infty \in \overline{\Aff(\Cb^d) \cdot \Omega} \cap \Xb_d,
\end{align*}
Theorem~\ref{thm:prior_nec} implies that $(\Omega_\infty, K_{\Omega_\infty})$ is Gromov hyperbolic. However, then by Theorem~\ref{thm:compactification} and the definition of the Gromov boundary the geodesic rays $t \rightarrow \wh{\gamma}(t)$ and $t \rightarrow \wh{\gamma}(-t)$ are in the same equivalence class. But then 
\begin{align*}
\infty > \limsup_{t \rightarrow \infty} K_{\Omega_\infty}(\wh{\gamma}(t), \wh{\gamma}(-t)) = \limsup_{t \rightarrow \infty} 2t = \infty.
\end{align*}
So we have a contradiction. 

Thus  $\lim_{r \searrow 0} D_1(r) = 0$. 
\end{proof}

\begin{lemma}\label{lem:D2} For any $r > 0$ there exists $D_2(r) < \infty$ such that:  if $\xi \in \partial\Omega \cap \Bb_d(0;R)$, $\epsilon \in (0,1)$, and $\gamma: [a,b] \rightarrow \Omega$ is a geodesic with $\gamma(a) \in A^{-1}_{\xi,\epsilon}\Bb_d(e_1;r)$ and $\gamma(b) \notin A^{-1}_{\xi,\epsilon} \Bb_d(e_1;2r)$, then
\begin{align*}
K_\Omega(q_{\xi,\epsilon}, \gamma) \leq D_2(r).
\end{align*}
\end{lemma}

\begin{remark} This lemma says that a geodesic in $A_{\xi,\epsilon}\Omega$ that starts close to $e_1$ and ends far from $e_1$ must pass close to $0=A_{\xi,\epsilon}(q_{\xi,\epsilon})$. \end{remark}

\begin{proof} 
Suppose for a contradiction that such a $D_2(r)$ does not exist for some $r > 0$. Then for each $n \in \Nb$ there exist $\xi_n \in \partial \Omega \cap \Bb_d(0;R)$, $\epsilon_n \in (0,1)$, and a geodesic $\gamma_n: [a_n,b_n] \rightarrow \Omega$ where $\gamma_n(a_n) \in A^{-1}_{\xi_n,\epsilon_n}\Bb_d(e_1;r)$, $\gamma_n(b_n) \notin A^{-1}_{\xi,\epsilon_n} \Bb_d(e_1;2r)$, and
\begin{align*}
K_\Omega(q_{\xi_n,\epsilon_n}, \gamma_n) \geq n. 
\end{align*}
By Proposition~\ref{prop:normalizing_maps_visual_metric} each $\Omega_n : =A_{\xi_n, \epsilon_n}\Omega$ is in $\Kb_d(r_0)$, so by passing to a subsequence we can suppose that $\Omega_n$ converges to some $\Omega_\infty \in \Kb_d(r_0)$. Then Corollary~\ref{cor:simple_bd_implies_visibility} implies that $\Omega_n$ is a visibility sequence. 

Consider the geodesics $\wh{\gamma}_n: = A_{\xi_n, \epsilon_n}\gamma_n$. Then $\wh{\gamma}_n(a_n) \in \Bb_d(e_1;r)$ and $\wh{\gamma}_n(b_n) \notin \Bb_d(e_1;2r)$. So 
\begin{align*}
\norm{\wh{\gamma}_n(a_n)-\wh{\gamma}_n(b_n)} > r.
\end{align*}
Since $\Omega_n$ is a visibility sequence, we can pass to a subsequence and find $T_{n} \in [a_n, b_n]$ such that  $\wh{\gamma}_{n}(T_{n}) \rightarrow z \in \Omega_\infty$. Then 
\begin{align*}
\infty > K_{\Omega_\infty}(0, z) & = \lim_{n \rightarrow \infty} K_{\Omega_n}(0, \wh{\gamma}_n(T_n)) = \lim_{n \rightarrow \infty} K_{\Omega}(q_{\xi_n,\epsilon_n}, \gamma_n(T_n)) \\
& \geq \lim_{n \rightarrow \infty} K_\Omega(q_{\xi_n,\epsilon_n}, \gamma_n) = \infty. 
\end{align*}
So we have a contradiction. Hence for each $r > 0$, there exists some $D_2(r) > 0$ with the desired property. 
\end{proof}

\begin{lemma}\label{lem:D3} For any $r > 0$ there exists $D_3(r) < \infty$ such that:  if $\xi \in \partial\Omega \cap \Bb_d(0;R)$, $\epsilon \in (0,1)$, and $\gamma: [a,b] \rightarrow \Omega$ is a geodesic with $\gamma(a) \in A^{-1}_{\xi,\epsilon}\Bb_d(e_1;r)$ and $\gamma(b) = z_0$, then
\begin{align*}
K_\Omega(q_{\xi,\epsilon}, \gamma) \leq D_3(r).
\end{align*}
Moreover, we can assume that $D_3$ is an increasing function. 
\end{lemma}

\begin{remark} This lemma is similar to Lemma~\ref{lem:D2}, however the increasing condition on $D_3$ (which may not hold for $D_2$) will be important for later estimates. 
\end{remark}

\begin{proof} Define $\wt{D}_3(r) \in (0,\infty]$ to be the smallest number satisfying the first part of the lemma (since the inequality is not strict, there does indeed exist a smallest number). 

We claim that $\wt{D}_3(r) < \infty$ for every $r > 0$. Suppose that $\xi \in \partial\Omega \cap \Bb_d(0;R)$, $\epsilon \in (0,1)$, and $\gamma: [a,b] \rightarrow \Omega$ is a geodesic with $\gamma(a) \in A^{-1}_{\xi,\epsilon}\Bb_d(e_1;r)$ and $\gamma(b) = z_0$. 

If $z_0 \notin A^{-1}_{\xi,\epsilon}\Bb_d(e_1;2r)$, then 
\begin{align*}
K_\Omega(q_{\xi,\epsilon}, \gamma) \leq D_2(r)
\end{align*}
by Lemma~\ref{lem:D2}. 

Next consider the case when $z_0 \in A^{-1}_{\xi,\epsilon}\Bb_d(e_1;2r)$. Since $A_{\xi,\epsilon}(q_{\xi,\epsilon}) = 0$, $A_{\xi,\epsilon}(\xi) = e_1$, and $q_{\xi,\epsilon} \in (z_0, \xi)$ we see that 
\begin{align*}
2r+1 \geq \norm{A_{\xi,\epsilon}(z_0)} = \frac{1}{\norm{q_{\xi,\epsilon}-\xi}}\norm{q_{\xi,\epsilon}-z_0} \geq \frac{1}{\norm{q_{\xi,\epsilon}-\xi}}\left( \norm{\xi-z_0}-\norm{\xi-q_{\xi,\epsilon}}\right).
\end{align*}
So 
\begin{align*}
\norm{q_{\xi,\epsilon}-\xi} \geq \frac{1}{2r+2} \norm{\xi-z_0} \geq \frac{\delta_\Omega(z_0)}{2r+2}.
\end{align*}
Then by Lemma~\ref{lem:dist_est_normalizing_maps} 
\begin{align*}
K_\Omega(q_{\xi,\epsilon}, \gamma) \leq K_\Omega(q_{\xi,\epsilon}, z_0) \leq \beta + \frac{\alpha}{2} \log \frac{1}{\norm{q_{\xi,\epsilon}-\xi}} \leq \beta + \frac{\alpha}{2} \log  \frac{2r+2}{\delta_\Omega(z_0)}.
\end{align*}

Thus 
\begin{align*}
\wt{D}_3(r) \leq \max\left\{ \beta + \frac{\alpha}{2} \log  \frac{2r+2}{\delta_\Omega(z_0)}, D_2(r) \right\}
\end{align*}
is finite. 

Finally, by definition $\wt{D}_3$ is non-decreasing and so $D_3(r) := \wt{D}_3(r) + r$ is increasing and satisfies the lemma.
\end{proof}

For $r > 0$, let $\wt{\tau}_1(r) \in (0,\infty]$ be the infimum of all numbers $\tau > 0$ such that 
\begin{align*}
\overline{\Omega} \cap A^{-1}_{\xi,\epsilon} \Bb(e_1;r) \subset V_{z_0}(\xi; \tau \epsilon)
\end{align*}
for all $\xi\in \partial \Omega \cap \Bb_d(0;R)$ and $\epsilon \in (0,1)$. Then define $\tau_1(r) : = r + \wt{\tau}_1(r)$. Notice that 
\begin{align*}
\overline{\Omega} \cap A^{-1}_{\xi,\epsilon} \Bb(e_1;r) \subset V_{z_0}(\xi; \tau_1(r) \epsilon)
\end{align*}
for all $\xi\in \partial \Omega \cap \Bb_d(0;R)$ and $\epsilon \in (0,1)$. 

\begin{lemma} $\tau_1(r) < \infty$ for every $r > 0$, $\tau_1$ is increasing,  and $\lim_{r \searrow 0} \tau_1(r) = 0$. \end{lemma}

\begin{proof} We first prove that $\tau_1(r) < \infty$ for every $r > 0$. Fix $r>0$, $\xi\in \partial \Omega \cap \Bb_d(0;R)$, $\epsilon \in (0,1)$, and $y \in \overline{\Omega} \cap A^{-1}_{\xi,\epsilon} \Bb_d(e_1;r)$. Let $\gamma: (-\infty, b) \rightarrow \Omega$ be a geodesic such that 
\begin{align*}
\lim_{t \rightarrow -\infty} \gamma(t) = \xi \text{ and } \lim_{t \rightarrow b} \gamma(t) = y.
\end{align*}
(notice that $b < \infty$ when $y \in \Omega$ and $b=\infty$ when $y \in \partial \Omega$). Then by Lemma~\ref{lem:D1}
\begin{align*}
\gamma \subset A^{-1}_{\xi,\epsilon} \Bb_d(e_1;D_1(r)).
\end{align*} 
Let $T = K_\Omega(z_0, \gamma)$. Then there exists a geodesic $\sigma:[0,T] \rightarrow \Omega$ with $\sigma(0) = z_0$ and $\sigma(T) \in \gamma$. Then 
\begin{align}
\label{eq:inclusion_sigma}
\sigma(T) \in \gamma \subset A^{-1}_{\xi,\epsilon}\Bb_d(e_1;D_1(r)). 
\end{align}
Hence, if $D_4(r) = D_3(D_1(r))$, then by Lemma~\ref{lem:D3}
\begin{align*}
K_\Omega(\sigma(t_0), q_{\xi,\epsilon}) \leq D_4(r)
\end{align*}
for some $t_0 \in [0,T]$. Then
\begin{align*}
K_\Omega(z_0, \gamma)& = K_\Omega(z_0,\sigma(t_0))+K_\Omega(\sigma(t_0), \sigma(T)) \\
& \geq K_\Omega(z_0,q_{\xi,\epsilon}) +K_\Omega(q_{\xi,\epsilon}, \sigma(T)) - 2D_4(r) \\
& =\frac{1}{\lambda}\log\frac{1}{\epsilon}+K_\Omega(q_{\xi,\epsilon}, \sigma(T)) - 2D_4(r).
\end{align*}
Thus 
\begin{align*}
d_{z_0}(\xi,y) 
&\leq C_v \exp\left( -\lambda K_\Omega(z_0, \gamma) \right) \leq C_v\epsilon \exp\left(2\lambda D_4(r) \right) \exp\left(-\lambda K_\Omega(q_{\xi,\epsilon}, \sigma(T)) \right).
\end{align*}

Next consider the complex hyperplane $H:=e_1+\Span_{\Cb}\{e_2,\dots, e_d\}$. Then $H \cap A_{\xi,\epsilon}\Omega = \emptyset$ since $A_{\xi,\epsilon}\Omega \in \Kb_d(r_0)$. Then Lemma~\ref{lem:hyperplanes} and Equation~\eqref{eq:inclusion_sigma} imply that
\begin{align*}
K_\Omega(q_{\xi,\epsilon}, \sigma(T)) = K_{A_{\xi,\epsilon}\Omega}(A_{\xi,\epsilon}q_{\xi,\epsilon}, A_{\xi,\epsilon} \sigma(T)) \geq  \frac{1}{2} \log \frac{d_{\Euc}(0,H)}{d_{\Euc}(A_{\xi,\epsilon}\sigma(T), H)} \geq \frac{1}{2} \log \frac{1}{D_1(r)}.
\end{align*}
Then, 
\begin{align*}
d_{z_0}(\xi,y) \leq C_v \epsilon \exp\left(2\lambda D_4(r) \right) D_1(r)^{\lambda/2}.
\end{align*}

Since $\xi\in \partial \Omega \cap \Bb_d(0;R)$, $\epsilon \in (0,1)$, and $y \in \overline{\Omega} \cap A^{-1}_{\xi,\epsilon} \Bb_d(e_1;r)$ were arbitrary we have
\begin{align*}
\tau_1(r) \leq r + C_v \exp\left(2\lambda D_4(r) \right) D_1(r)^{\lambda/2} < \infty.
\end{align*}
This proves the first assertion.

By definition, $\wt{\tau}_1$ is non-decreasing and so $\tau_1$ is increasing. Thus the second assertion is true. 

To prove the last assertion, first notice that $D_4$ is increasing by Lemmas~\ref{lem:D3} and~\ref{lem:D1part2}. Then Lemma~\ref{lem:D1part2} implies that
\begin{equation*}
\lim_{r \searrow 0} \tau_1(r)  \leq C_v \exp\left(2\lambda D_4(1) \right) \lim_{r \searrow 0} D_1(r)^{\lambda/2}=0. \qedhere
\end{equation*}
\end{proof}

Next for $r > 0$ let $\tau_2(r) \in (0,\infty]$ be the smallest number such that 
\begin{align*}
V_{z_0}(\xi; r\epsilon) \subset \overline{\Omega} \cap A^{-1}_{\xi,\epsilon} \Bb_d(e_1;\tau_2(r))
\end{align*}
for all $\xi\in \partial \Omega \cap \Bb_d(0;R)$ and $\epsilon \in (0,\epsilon_0/r) \cap (0,1)$. Observation~\ref{obs:the sets V are open} and Equation~\eqref{eqn:Vz0 is in C} imply that  $V_{z_0}(\xi; r\epsilon)$ is an open set in $\overline{\Omega}$ and hence $\tau_2(r)$ exists.

\begin{lemma}\label{lem:limit_of_tau2} $\lim_{r \searrow 0} \tau_2(r) = 0$. \end{lemma}

\begin{proof} Suppose not, then there exists $\tau_0 > 0$ such that: for every $n \in \Nb$ there exist $\xi_n\in \partial \Omega \cap \Bb_d(0;R)$, $\epsilon_n \in (0,1)$, and $y_n \in V_{z_0}(\xi_n; \frac{1}{n}\epsilon_n)$ with
\begin{align*}
\tau_0 \leq \norm{A_{\xi_n,\epsilon_n}y_n-e_1}.
\end{align*}
Let $\gamma_n : (-\infty, b_n) \rightarrow \Omega$ be a geodesic with 
\begin{align*}
\lim_{t \rightarrow -\infty} \gamma_n(t) = \xi_n \text{ and } \lim_{t \rightarrow b_n} \gamma_n(t) = y_n
\end{align*}
(notice that $b_n < \infty$ when $y_n \in \Omega$ and $b_n=\infty$ when $y_n \in \partial \Omega$). By Lemma~\ref{lem:D2}, there exists $t_n \in (-\infty, b_n)$ such that 
\begin{align*}
K_\Omega(q_{\xi_n,\epsilon_n}, \gamma_n(t_n)) \leq D_2(\tau_0/2).
\end{align*}
Then
\begin{align*}
K_\Omega(z_0, \gamma_n) \leq K_\Omega(z_0, q_{\xi_n,\epsilon_n}) + K_\Omega(q_{\xi_n,\epsilon_n}, \gamma_n(t_n)) \leq \frac{1}{\lambda} \log \frac{1}{\epsilon_n} + D_2(\tau_0/2).
\end{align*}
So 
\begin{align*}
\frac{1}{n} \epsilon_n &\geq d_{z_0}(\xi_n, y_n) \geq \frac{1}{C_v} \exp( -\lambda K_\Omega(z_0, \gamma_n)) \\
& \geq \frac{1}{C_v} \exp(-\lambda D_2(\tau_0/2))\epsilon_n.
\end{align*}
Then sending $n \rightarrow \infty$ yields a contradiction. Thus $\lim_{r \searrow 0} \tau_2(r) = 0$.
\end{proof} 

\begin{lemma}For any $r> 0$,  $\sup_{s \in (0,r]} \tau_2(s) < \infty$.\end{lemma} 

\begin{proof} Suppose for a contradiction that $\sup_{s \in (0,r]} \tau_2(s) = \infty$ for some $r > 0$. Then for every $n \in\Nb$ there exist $s_n \in (0,r]$, $\xi_n \in \partial \Omega \cap \Bb_d(0;R)$, $\epsilon_n \in (0,\epsilon_0/s_n) \cap (0,1)$, and 
$$
y_n \in V_{z_0}(\xi_n; s_n\epsilon_n) \setminus A^{-1}_{\xi_n,\epsilon_n} \Bb_d(e_1;n+1).
$$
Notice that $y_n \in \overline{\Omega} \subset \Cb^d$ by our choice of $\epsilon_0$, see Equation~\eqref{eqn:Vz0 is in C}. Also
\begin{align*}
\norm{A_{\xi_n,\epsilon_n} y_n } \geq \norm{A_{\xi_n,\epsilon_n} y_n -e_1}-\norm{e_1} \geq n.
\end{align*}

By passing to a subsequence we can suppose that $s_n \rightarrow s_\infty$, $\epsilon_n \rightarrow \epsilon_\infty$, and $\Omega_n:=A_{\xi_n,\epsilon_n}\Omega$ converges to some $\Omega_\infty$ in $\Kb_d(r_0)$. By Lemma~\ref{lem:limit_of_tau2} we must have $s_\infty \neq 0$ and so 
$$
\epsilon_\infty \leq \frac{\epsilon_0}{s_\infty}.
$$
Also, Corollary~\ref{cor:simple_bd_implies_visibility} implies that $\Omega_n$ is a visibility sequence. 

We consider two cases. 

\medskip

\noindent \textbf{Case 1:} $\epsilon_\infty > 0$. Then 
\begin{align*}
\sup_{n \in \Nb} K_\Omega(z_0, q_{\xi_n,\epsilon_n}) =\sup_{n \in \Nb} \frac{1}{\lambda} \log \frac{1}{\epsilon_n}< \infty.
\end{align*}
So we can pass to a subsequence such that $q_{\xi_n,\epsilon_n} \rightarrow q \in \Omega$. Then $(\Omega, q_{\xi_n,\epsilon_n}) \rightarrow (\Omega,q)$ and $A_{\xi_n,\epsilon_n}(\Omega, q_{\xi_n,\epsilon_n}) \rightarrow (\Omega_\infty,0)$. So by Proposition~\ref{prop:limit_domain_depends_on_sequence}, we can pass to a subsequence where $A_{\xi_n,\epsilon_n} \rightarrow A \in \Aff(\Cb^d)$. Then
\begin{align*}
\lim_{n \rightarrow \infty} \norm{y_n }=\lim_{n \rightarrow \infty} \norm{A^{-1}A_{\xi_n,\epsilon_n} y_n } =\infty.
\end{align*}
By passing to another subsequence we can suppose that $\xi_n \rightarrow \xi \in \partial \Omega \cap \overline{\Bb_d(0;R)}$ and $y_n \rightarrow \eta \in \overline{\Omega}^{\rm End} \setminus \Cb^d$. Then Proposition~\ref{prop:continuity of visual metric} implies that
$$
d_{z_0}(\eta, \xi) = \lim_{n \rightarrow \infty} d_{z_0}(y_n, \xi_n) \leq s_\infty \epsilon_\infty \leq \epsilon_0
$$
which contradicts the definition of $\epsilon_0$, see Equation~\eqref{eqn:Vz0 is in C}.

\medskip

\noindent \textbf{Case 2:} $\epsilon_\infty = 0$. Then 
\begin{align}
\label{eq:epsilon_n_goes_to_zero}
\lim_{n \rightarrow \infty} \norm{\xi_n-q_{\xi_n,\epsilon_n}} = 0.
\end{align}
We first show that $\overline{\Omega}_\infty$ is one-ended. By construction $A_{\xi_n,\epsilon_n}(z_0) = t_n e_1$ for some $t_n \leq 0$. Since $z_0$, $q_{\xi_n,\epsilon_n}$, and $\xi_n$ are co-linear 
\begin{align*}
-t_n & = \abs{t_n} =\norm{A_{\xi_n,\epsilon_n}(z_0) - A_{\xi_n,\epsilon_n}(q_{\xi_n,\epsilon_n})} =  \frac{\norm{z_0 - q_{\xi_n,\epsilon_n}}}{\norm{\xi_n-q_{\xi_n,\epsilon_n}}}.
\end{align*}
Then, since 
\begin{align*}
\liminf_{n \rightarrow \infty}  & \norm{z_0-q_{\xi_n,\epsilon_n}} \geq \liminf_{n \rightarrow \infty}  \norm{z_0-\xi_n} - \norm{\xi_n-q_{\xi_n,\epsilon_n}}\\
& =\liminf_{n \rightarrow \infty}  \norm{z_0-\xi_n} \geq \delta_\Omega(z_0),
\end{align*} 
Equation~\eqref{eq:epsilon_n_goes_to_zero} implies that $t_n \rightarrow -\infty$. So $-e_1 \in {\rm AC}(\Omega_\infty)$. Since $\Omega_\infty \in \Kb_d(r_0)$, we have 
\begin{align*}
\Big(e_1 + \Span_{\Cb}\{e_2,\dots, e_d\} \Big)\cap \Omega_\infty = \emptyset
\end{align*}
and so $e_1 \notin {\rm AC}(\Omega_\infty)$. Thus $\overline{\Omega}_\infty$ is one-ended by Observation~\ref{obs:asymptotic_cone_3}. 

Now let $\gamma_n : (-\infty, b_n) \rightarrow \Omega$ be a geodesic with 
\begin{align*}
\lim_{t \rightarrow -\infty} \gamma_n(t) = \xi \text{ and } \lim_{t \rightarrow b_n} \gamma_n(t) = y_n.
\end{align*}
(notice that $b_n < \infty$ when $y_n \in \Omega$ and $b_n=\infty$ when $y_n \in \partial \Omega$). Next consider the geodesics $\wh{\gamma}_n = A_{\xi_n,\epsilon_n}\gamma_n:(-\infty, b_n) \rightarrow \Omega_n$. Since $\Omega_n$ is a visibility sequence, after passing to a subsequence there exists a sequence $T_n \in (-\infty, b_n)$ such that $\wh{\gamma}_n(T_n)$ converges to a point in $\Omega_\infty$. Passing to a further subsequence, we can assume that
$$
b : = \lim_{n \rightarrow \infty} b_n-T_n \in [0,\infty]
$$ 
exists and, by Observation~\ref{obs:AA_for_geod}, that $\wh{\gamma}_n(\cdot +T_n)$ converges locally uniformly to a geodesic $\wh{\gamma} : (-\infty, b] \cap \Rb \rightarrow \Omega_\infty$.  Since 
$$
\lim_{n \rightarrow \infty} \norm{A_{\xi_n,\epsilon_n}y_n}=\infty,
$$
the ``in particular'' part of Observation~\ref{obs:AA_for_geod} implies that $b=\infty$. Then, by Lemma~\ref{lem:limits_of_geod}
\begin{align*}
\lim_{t \rightarrow \infty} \norm{\wh{\gamma}(t)} = \lim_{n \rightarrow \infty} \norm{A_{\xi_n,\epsilon_n}y_n}=\infty. 
\end{align*}

Next let $\sigma_n :[0,c_n] \rightarrow \Omega$ be a sequence of geodesics with $\sigma_n(c_n) = z_0$ and $\sigma_n(0) = q_{\xi_n,\epsilon_n}$. Notice that 
\begin{align*}
c_n =K_\Omega(q_{\xi_n,\epsilon_n}, z_0) =  \frac{1}{\lambda} \log \frac{1}{\epsilon_n} \rightarrow \infty.
\end{align*}
Consider the geodesic $\wh{\sigma}_n = A_{\xi_n,\epsilon_n} \sigma_n : [0,c_n] \rightarrow \Omega_n$. Then $\wh{\sigma}_n(0) = 0$, so using Observation~\ref{obs:AA_for_geod} we can pass to a subsequence such that $\wh{\sigma}_n$ converges locally uniformly to a geodesic $\wh{\sigma} : [0,\infty) \rightarrow \Omega_\infty$.  By Lemma~\ref{lem:limits_of_geod}
 \begin{align*}
\lim_{t \rightarrow \infty} \norm{\wh{\sigma}(t)} =  \lim_{n \rightarrow \infty} \norm{\wh{\sigma}_n(c_n) }= \lim_{n \rightarrow \infty} \norm{A_{\xi_n, \epsilon_n} z_0}=  \lim_{n \rightarrow \infty} \abs{t_n} = \infty. 
\end{align*}

Since 
\begin{align*}
\Omega_\infty \in \overline{\Aff(\Cb^d) \cdot \Omega} \cap \Xb_d,
\end{align*}
Theorem~\ref{thm:prior_nec} implies that $(\Omega_\infty, K_{\Omega_\infty})$ is Gromov hyperbolic. Then, since $\overline{\Omega}_\infty$ is one-ended, Theorem~\ref{thm:compactification} implies that $\wh{\gamma}|_{[0,\infty)}$ and $\wh{\sigma}$ are in the same equivalence class of rays in $\partial_G \Omega_\infty$. So
\begin{align*}
M:=\sup_{t \geq 0} K_{\Omega_\infty}( \wh{\sigma}(t), \wh{\gamma}(t)) < \infty.
\end{align*}

Now fix some 
\begin{align*}
T > M+1 + \frac{1}{\lambda} \log ( r  C_v).
\end{align*}
Notice that $T+T_n < b_n$ for $n$ sufficiently large since $b=\infty$. Then for $n$ sufficiently large, Proposition~\ref{prop:convergence_of_kob} implies that
\begin{align*}
K_\Omega(\sigma_n(T),& \gamma_n(T+T_n)) = K_{\Omega_n}( \wh{\sigma}_n(T), \wh{\gamma}_n(T+T_n)) \\
& \leq 1+K_{\Omega_\infty}( \wh{\sigma}(T), \wh{\gamma}(T))\leq 1+ M.
\end{align*}
Then for $n$ sufficiently large
\begin{align*}
K_\Omega(z_0, \gamma_n) & \leq K_\Omega(z_0, \gamma_n(T+T_n))\leq K_\Omega(z_0, \sigma_n(T)) + K_\Omega(\sigma_n(T), \gamma_n(T+T_n))\\
&  \leq c_n - T + 1+M  = \frac{1}{\lambda} \log \frac{1}{\epsilon_n} - T + 1+M
\end{align*}
and
\begin{align*}
d_{z_0}(\xi_n, y_n) 
&\geq \frac{1}{C_v} \exp ( -\lambda K_\Omega(z_0, \gamma_n)) \geq \frac{\epsilon_n}{C_v}\exp( \lambda T-\lambda M-\lambda)  \\
& > r \epsilon_n.
\end{align*}
Thus $y_n \notin V_{z_0}(\xi_n; r\epsilon_n)$ for $n$ sufficiently large and hence we have a contradiction.

\end{proof}

Finally we can finish the proof of Proposition~\ref{prop:comparison} by setting 
\begin{align*}
\tau(r)=\tau_1(r)+\sup_{s \in (0,r]} \tau_2(s).
\end{align*}

\section{Plurisubharmonic functions on normalized domains}\label{sec:psh_on_normal_domains}

In this section we construct special plurisubharmonic functions on normalized domains. This construction is similar to the proof of~\cite[Proposition 3.1]{M1994}.

\begin{proposition}\label{prop:plurisubharmonic_on_normal_domain} For any $d \in \Nb$ and $a,r >0$ there exist $C,b > 0$ such that: if $\Omega \in \Kb_d(r)$, then there exists a $C^\infty$ plurisubharmonic function $F: \Omega \rightarrow [0,1]$ with
\begin{align*}
i\partial \bar{\partial} F(z) \geq C i\partial \bar{\partial} \norm{z}^2 \text{ on } \Bb_d(e_1;a) \cap \Omega 
\end{align*}
and 
\begin{align*}
\supp(F) \subset \Bb_d(e_1;b) \cap \Omega.
\end{align*}
\end{proposition}

The rest of the section is devoted to the proof of the Proposition. 

\begin{definition} Given $\Omega \in \Kb_d(r)$ we say that a list of vectors $(v_1, \dots, v_d)$ is \emph{$\Omega$-supporting} if 
\begin{align*}
e_j + \Span_{\Cb}\{e_{j+1}, \dots, e_d\} \subset \{ z \in \Cb^d : {\rm Re} \ip{z,v_j} = 1\}
\end{align*}
and 
\begin{align*}
\Omega \subset  \{ z \in \Cb^d : {\rm Re} \ip{z,v_j} < 1\}
\end{align*}
for all $j \in \{1,\dots, d\}$.
\end{definition}

\begin{lemma} If  $\Omega \in \Kb_d(r)$, then there exists a list of $\Omega$-supporting vectors. \end{lemma}

\begin{proof}
Since $\Omega$ is convex and 
\begin{align*}
(e_j + \Span_{\Cb}\{e_{j+1}, \dots, e_d\}) \cap \Omega = \emptyset
\end{align*}
there exists a real hyperplane $H_j$ such that $H_j \cap \Omega = \emptyset$ and 
\begin{align*}
e_j + \Span_{\Cb}\{e_{j+1}, \dots, e_d\} \subset H_j.
\end{align*}
Since $0 \in \Omega$, for each $j$ we can pick $v_j \in \Cb^d$ such that $H_j = \{ z \in \Cb^d : {\rm Re} \ip{z,v_j} = 1\}$ and $\Omega \subset  \{ z \in \Cb^d : {\rm Re} \ip{z,v_j} < 1\}$. 
\end{proof}

\begin{lemma}\label{lem:supp_vec_estimate} If $\Omega \in \Kb_d(r)$ and $(v_1,\dots, v_d)$ is $\Omega$-supporting, then
\begin{enumerate}
\item $1 \leq \abs{v_{1,1}} \leq r^{-1}$,
\item  $v_{j,j}=1$ when $j > 1$, 
\item $v_{j,\ell}=0$ when $\ell > j$,
\item $\abs{v_{j,1}} \leq r^{-1}$ for $1 \leq j \leq d$,
\item $\abs{v_{j,\ell}} \leq 1$ for $1 < \ell \leq j$.
\end{enumerate}
In particular, 
\begin{align*}
\norm{v_j} \leq \sqrt{ r^{-2}+(j-1)}
\end{align*}
for $1 \leq j \leq d$. 
\end{lemma}

\medskip

\begin{proof} Since
\begin{align*}
r\Db \cdot e_1 \subset \Omega \subset \{ z \in \Cb^d : {\rm Re} \ip{z,v_j} < 1\}
\end{align*}
 we must have $ \abs{v_{j,1}} \leq r^{-1}$ for $1 \leq j \leq d$. This proves (4).
 
When $1 < \ell \leq d$, 
\begin{align*}
\Db \cdot e_\ell \subset \Omega \subset \{ z \in \Cb^d : {\rm Re} \ip{z,v_j} < 1\}
\end{align*}
and so $\abs{v_{j,\ell}} \leq 1$. This proves (5). 

Since 
\begin{align*}
e_j + \Span_{\Cb}\{e_{j+1}, \dots, e_d\} \subset \{ z \in \Cb^d : {\rm Re} \ip{z,v_j} = 1\}
\end{align*}
we must have ${ \rm Re}(v_{j,j})=1$ and $v_{j,\ell}=0$ when $\ell > j$. This proves (3) and when combined with (5) (respectively (4)) implies (2) (respectively (1)). 
 \end{proof}

\begin{lemma} For any $d \in \Nb$, $r \in (0,1]$, and $a > 0$ there exist $\alpha,b,C > 0$ with the following property:  If $\Omega \in \Kb_d(r)$, $(v_1,\dots, v_d)$ is $\Omega$-supporting, and $h : \Omega \rightarrow (-\infty, 1)$ is defined by 
\begin{align*}
h(z) = \frac{1}{d}\sum_{j=1}^d e^{2{ \rm Re} \ip{z,v_j} -2} + \sum_{j=1}^d \ln \abs{ \frac{1}{2- \ip{z,v_j}}},
\end{align*}
then 
\begin{enumerate}
\item $-\alpha \leq h(z)$ on $\Bb_d(e_1;a) \cap \Omega$,
\item $h(z) \leq -2\alpha$ on $\Omega \setminus \Bb_d(e_1;b)$, 
\item $h$ is strictly plurisubharmonic on $\Omega$, and
\item $i\partial \bar{\partial} h(z) \geq C i\partial \bar{\partial} \norm{z}^2$ on $\Bb_d(e_1;a) \cap \Omega$. 
\end{enumerate}
\end{lemma}

\begin{proof} If $\Omega \in \Kb_d(r)$ and $(v_1,\dots, v_d)$ is $\Omega$-supporting,  then
$$
2{ \rm Re} \ip{z,v_j} -2 < 0 \quad \text{and} \quad \abs{\frac{1}{2- \ip{z,v_j}}} \leq  \frac{1}{2- { \rm Re} \ip{z,v_j} } < 1
$$
for all $1 \leq j \leq d$ and $z \in \Omega$. So $h$ does indeed map $\Omega$ into $(-\infty, 1)$. 

The existence of some $\alpha > 0$ satisfying Part (1) follows from Lemma~\ref{lem:supp_vec_estimate}. 

Lemma~\ref{lem:supp_vec_estimate} also implies that there exists $\epsilon > 0$ such that: if $\Omega \in \Kb_d(r)$ and $(v_1,\dots, v_d)$ is $\Omega$-supporting, then 
\begin{equation}
\label{eqn:reproducing_formula}
\max_{1 \leq j \leq d} \abs{\ip{z,v_j}} \geq \epsilon \norm{z}
\end{equation}
for all $z \in \Cb^d$. Hence, if $\Omega \in \Kb_d(r)$, $(v_1,\dots, v_d)$ is $\Omega$-supporting, and $z \in \Omega$ then
\begin{align*}
 \sum_{j=1}^d \ln \abs{ \frac{1}{2- \ip{z,v_j}}}  < \ln \frac{1}{2+\epsilon\norm{z}}. 
\end{align*}
So there exists some $b > 0$ satisfying Part (2).

Next we show that any such $h$ is strictly plurisubharmonic. Suppose  $\Omega \in \Kb_d(r)$ and $(v_1,\dots, v_d)$ is $\Omega$-supporting. Fix some $X \in \Cb^d$. The second sum in the definition of $h$ is pluriharmonic on $\Omega$, so 
\begin{align*}
\sum_{j,k=1}^d \frac{\partial^2h(z)}{\partial z_j \partial \bar{z}_k} X_j\bar{X}_k = \frac{1}{d}\sum_{j=1}^d e^{2{ \rm Re} \ip{z,v_j} -2}\abs{\ip{X,v_j}}^2.
\end{align*}
Then using Equation~\eqref{eqn:reproducing_formula}
\begin{align}
\label{eqn:lower_bd_on_hessian_of_h}
\sum_{j,k=1}^d \frac{\partial^2h(z)}{\partial z_j \partial \bar{z}_k} X_j\bar{X}_k \geq \frac{e^{-2(1+\norm{z})}}{d}\sum_{j=1}^d \abs{\ip{X,v_j}}^2 \geq  \frac{e^{-2(1+\norm{z})}}{d} \epsilon^2 \norm{X}^2 > 0. 
\end{align}
Hence $h$ is strictly plurisubharmonic on $\Omega$. 

Finally,  Equation~\eqref{eqn:lower_bd_on_hessian_of_h} implies that there exists some $C$ satisfying part (4). 
\end{proof}

\begin{proof}[Proof of Proposition~\ref{prop:plurisubharmonic_on_normal_domain}]

 Let $\chi: \Rb \rightarrow [0,\infty)$ be a convex $C^\infty$ function such that 
\begin{enumerate}
\item $\chi(x) = 0$ on $\left(-\infty, -2\alpha\right]$,
\item $\chi^\prime(x) > 0$ and $\chi^{\prime\prime}(x) > 0$ on $\left(-2\alpha, \infty\right)$, and
\item $\chi(1) = 1$.
\end{enumerate}
Let $\kappa := \min\{ \chi^{\prime\prime}(r) : r \in [-\alpha, 1] \}$. 

Suppose $\Omega \in \Kb_d(r)$, $(v_1,\dots, v_d)$ is $\Omega$-supporting, and let $h :\Omega \rightarrow \Rb$ be the function from the last lemma. Then define $F: \Omega \rightarrow[0, 1]$ by $F=\chi \circ h$. Then by construction $\supp(F) \subset \Bb_d(e_1;b) \cap \Omega$. Moreover
\begin{align*}
i\partial \bar{\partial} F(z)
&= (\chi^{\prime\prime} \circ h)(z)i\partial \bar{\partial}h(z)+(\chi^\prime \circ h(z))^2 i\partial h\wedge \bar{\partial} h
\end{align*}
and so $F$ is plurisubharmonic on $\Omega$. Finally, when $z \in \Omega \cap \Bb_d(e_1;a)$ we have
\begin{equation*}
i\partial \bar{\partial} F(z)
\geq (\chi^{\prime\prime} \circ h)(z)i\partial \bar{\partial}h(z) \geq \kappa C i\partial \bar{\partial} \norm{z}^2 
\end{equation*}
where $C > 0$ is the constant in the last lemma. 
\end{proof}

\section{Plurisubharmonic functions on convex domains}\label{sec:psh_on_general_domains}

In this section we construct functions satisfying the hypothesis of Theorem~\ref{thm:straube}.  This construction uses ideas from the proofs of~\cite[Propositions 3.1, 3.2]{M1994} and~\cite[Theorem 2]{S1997}.

\begin{theorem}\label{thm:final_plurisubharmonic_fcn}
Suppose $\Omega \subset \Cb^d$ is a $\Cb$-properly convex domain and $(\Omega, K_\Omega)$ is Gromov hyperbolic. If $\xi_0 \in \partial \Omega$, then there exist $C > 0$, $m_2 > 2$, a neighborhood $U$ of $\xi_0$, and a bounded continuous plurisubharmonic function $G : U \cap \Omega \rightarrow [0,1]$ 
such that 
\begin{align*}
i\partial \bar{\partial}G(z) \geq \frac{1}{\delta_{\Omega}(z)^{2/m_2}} i\partial \bar{\partial} \norm{z}^2 \text{ on } U \cap \Omega.
\end{align*}
\end{theorem}

For the rest of the section fix $\Omega \subset \Cb^d$ a $\Cb$-properly convex domain where $(\Omega, K_\Omega)$ is Gromov hyperbolic. Then fix some $z_0 \in \Omega$ and $\xi_0 \in \partial \Omega$. Finally, fix some $R > 0$ with $z_0, \xi_0 \in \Bb_d(0;R)$. 

As in Section~\ref{sec:vm_and_normalizing}, let $d_{z_0}$ denote the function constructed in  Theorem~\ref{thm:visible_metric} for the metric space $(\Omega, K_\Omega)$. Using Theorem~\ref{thm:compactification} we can view $d_{z_0}$ as a function on $\overline{\Omega}^{\End} \times \overline{\Omega}^{\End}$. Let $C_v > 1$ and $\lambda > 0$ be constants such that: for all $x,y \in \overline{\Omega}^{\End}$
\begin{align*}
\frac{1}{C_v} \exp \Big(-\lambda K_\Omega(z_0, \gamma_{x,y})\Big) \leq d_{z_0}(x,y) \leq C_v\exp \Big( -\lambda K_\Omega(z_0, \gamma_{x,y}) \Big)
\end{align*}
when $\gamma_{x,y}$ is a geodesic in $(\Omega, K_\Omega)$ joining $x$ to $y$. As before, for $\xi \in \overline{\Omega}^{\End}$ and $r > 0$ define 
\begin{align*}
V_{z_0}(\xi;r) := \left\{ z \in \overline{\Omega}^{\End} : d_{z_0}(\xi, z) < r\right\}.
\end{align*}

\begin{lemma}\label{lem:plurisub_one} There exist $c_1, \epsilon_1 \in (0,1)$ and $m_1 >0$ such that: For any $\xi \in \partial \Omega \cap \Bb_d(0;R)$ and $\epsilon \in (0,\epsilon_1)$ there is a smooth plurisubharmonic function $F_{\xi,\epsilon} : \Omega \rightarrow [0,1]$ with
\begin{align*}
i\partial \bar{\partial} F_{\xi,\epsilon} (z) \geq \frac{c_1}{\epsilon^{2/m_1}} i\partial \bar{\partial} \norm{z}^2 \text{ on } V_{z_0}(\xi; 2\epsilon) \cap \Omega\end{align*}
and 
\begin{align*}
\supp(F_{\xi,\epsilon} ) \subset V_{z_0}\left(\xi; \frac{\epsilon}{c_1}\right).
\end{align*}
\end{lemma}

\begin{remark}\label{rmk:m_1_equals} The $m_1$ in Lemma~\ref{lem:plurisub_one} can be taken to be the $m_1$ from  Proposition~\ref{prop:normalizing_maps_visual_metric}. \end{remark}

\begin{proof} For $\xi \in \partial \Omega \cap \Bb_d(0;R)$ and $\epsilon \in (0,1)$, let $A_{\xi,\epsilon} \in \Aff(\Cb^d)$ be the affine map from Definition~\ref{defn:visual_metric_normalizing_map}. By Proposition~\ref{prop:normalizing_maps_visual_metric} there exist $r_0, C_0, m_1 > 0$ (which do not depend on $\xi$ or $\epsilon$) such that $A_{\xi,\epsilon} \Omega \in \Kb_d(r_0)$ and 
\begin{align}
\label{eq:A_lower_bound_in_proof}
\norm{A_{\xi,\epsilon}(z_1) - A_{\xi,\epsilon}(z_2)} \geq \frac{C_0}{\epsilon^{1/m_1}} \norm{z_1-z_2}
\end{align}
for all $z_1, z_2 \in \Cb^d$. Then let $\epsilon_0 \in (0,1]$ and $\tau: (0,\infty) \rightarrow (0,\infty)$ be the constant and function from Proposition~\ref{prop:comparison}. Also, let $C_1>0, b >1$ be the constants in Proposition~\ref{prop:plurisubharmonic_on_normal_domain} associated to $r=r_0$ and $a = \tau(2)$. Finally let 
$$
\epsilon_1 = \frac{\epsilon_0}{\max\{2,b\}}.
$$
 
Fix $\xi \in \partial \Omega \cap \Bb_d(0;R)$ and $\epsilon \in (0,\epsilon_1)$. By Proposition~\ref{prop:plurisubharmonic_on_normal_domain} there exists a smooth plurisubharmonic function $F : A_{\xi,\epsilon}\Omega \rightarrow [0,1]$ such that 
\begin{align*}
i\partial \bar{\partial} F(z) \geq C_1 i\partial \bar{\partial} \norm{z}^2 \text{ on } \Bb_d(e_1;\tau(2)) \cap A_{\xi,\epsilon}\Omega 
\end{align*}
and 
\begin{align*}
\supp(F) \subset \Bb_d(e_1;b) \cap A_{\xi,\epsilon}\Omega.
\end{align*}
Then define $F_{\xi,\epsilon} = F \circ A_{\xi,\epsilon} : \Omega \rightarrow [0,1]$. Then 
\begin{align*}
\supp(F_{\xi,\epsilon}) \subset \Omega \cap A_{\xi,\epsilon}^{-1} \Bb_d(e_1;b) \subset V_{z_0}(\xi; \tau(b)\epsilon).
\end{align*}
Moreover, if $A_{\xi,\epsilon}(\cdot) = b_0 + g(\cdot)$ where $z_0 \in \Cb^d$ and $g \in \GL_d(\Cb)$, then Equation~\eqref{eq:A_lower_bound_in_proof} implies that 
\begin{align*}
\norm{gz} \geq \frac{C_0}{\epsilon^{1/m_1}}\norm{z}
\end{align*}
for all $z \in \Cb^d$. 

Since $\epsilon < \frac{\epsilon_0}{2}$,  Proposition~\ref{prop:comparison} implies
$$
V_{z_0}(\xi; 2\epsilon) \subset \overline{\Omega} \cap A_{\xi,\epsilon}^{-1} \Bb_d(e_1;\tau(2)).
$$
So if $z \in V_{z_0}(\xi; 2\epsilon) \cap \Omega$ and $X \in \Cb^d$, we have
\begin{align*}
i\partial \bar{\partial} F_{\xi,\epsilon} (z)(X,\bar{X}) & = i\partial \bar{\partial} F(A_{\xi,\epsilon}z)(gX, \overline{gX}) \geq C_1\norm{gX}^2 \\
& \geq \frac{C_1C_0^2}{\epsilon^{2/m_1}} \norm{X}^2.
\end{align*}
Hence $c_1 = \min\left\{ C_1C_0^2, \frac{1}{\tau(b)} \right\}$ satisfies the lemma. 
\end{proof}

Next define 
\begin{align*}
V_\epsilon : = \cup \left\{ V_{z_0}(\xi;\epsilon) : \xi \in \partial\Omega \cap \Bb_d(0;R)\right\}.
\end{align*}

\begin{lemma}\label{lem:plurisub_two} There exist $c_2 \in (0,1)$ and $\epsilon_2 \in (0,\epsilon_1)$ such that: for any $\epsilon \in (0,\epsilon_2)$ there is a plurisubharmonic function $F_{\epsilon} : \Omega \rightarrow [0,1]$ with
\begin{align*}
i\partial \bar{\partial} F_{\epsilon} (z) \geq \frac{c_2}{\epsilon^{2/m_1}} i\partial \bar{\partial} \norm{z}^2 \text{ on } V_{\epsilon} \cap \Omega.
\end{align*}
\end{lemma}

\begin{proof} By Proposition~\ref{prop:doubling_proposition} there exist $\epsilon_2 \in (0,\epsilon_1)$ and $M > 0$ such that 
\begin{align*}
V_{z_0}\left(\xi; \frac{2\epsilon}{c_1}+\frac{\epsilon}{2}\right) \subset \xi + M \cdot\left( V_{z_0}\left(\xi;\frac{\epsilon}{2}\right) - \xi\right)
\end{align*}
for all $\xi \in \partial \Omega \cap \Bb_d(0;R)$ and $\epsilon \in (0,\epsilon_2)$. 

Fix $\epsilon \in (0,\epsilon_2)$. Let $\{ \xi_j : j \in J\} \subset \partial \Omega \cap \Bb_d(0;R)$ be a maximal set such that the sets $V_{z_0}(\xi_j;\epsilon/2)$ are pairwise disjoint. We claim that 
\begin{align*}
V_\epsilon \subset \cup_{j\in J} V_{z_0}(\xi_j;2\epsilon).
\end{align*}
If not, there exist $\xi \in \partial \Omega \cap \Bb_d(0;R)$ and $z \in V_{z_0}(\xi;\epsilon)$ such that 
\begin{align*}
d_{z_0}(z,\xi_j) > 2\epsilon 
\end{align*}
for all $j \in J$. Then 
\begin{align*}
d_{z_0}(\xi,\xi_j) \geq \min_{j=1,\dots, n} d_{z_0}(z,\xi_j)-d_{z_0}(z,\xi) > \epsilon
\end{align*}
for all $j \in J$. Hence $V_{z_0}(\xi;\epsilon/2)$ is disjoint from each $V_{z_0}(\xi_j;\epsilon/2)$. This contradicts the maximality. 

\medskip

\noindent \textbf{Claim:} If $z \in \Omega$, then 
$$
0 \leq \#\left\{ j : z \in V_{z_0}\left(\xi_j; \frac{\epsilon}{c_1}\right)\right\} \leq M^{2d}.
$$ 
\medskip

\noindent \emph{Proof of Claim:} This is just the proof of the Claim on page 124 in~\cite{M1994}: Suppose that
\begin{align*}
z \in \cap_{k=1}^\ell V_{z_0}\left(\xi_{j_k}; \frac{\epsilon}{c_1}\right)
\end{align*}
and 
\begin{align*}
\mu \left( V_{z_0}\left(\xi_{j_1}; \frac{\epsilon}{2}\right)\right) \leq \dots \leq \mu \left( V_{z_0}\left(\xi_{j_\ell}; \frac{\epsilon}{2}\right) \right)
\end{align*}
where $\mu$ is the Lebesgue measure on $\Cb^d$ (recall that these sets are open in $\overline{\Omega}$ by Observation~\ref{obs:the sets V are open}). Then 
\begin{align*}
  \mu\left(V_{z_0}\left(\xi_{j_1}; \frac{\epsilon}{2}\right)\right) & \leq \frac{1}{\ell}\sum_{k=1}^\ell \mu\left(V_{z_0}\left(\xi_{j_k}; \frac{\epsilon}{2}\right)\right) = \frac{1}{\ell}\mu\left( \cup_{k=1}^\ell V_{z_0}\left(\xi_{j_k}; \frac{\epsilon}{2}\right)\right)\\
 &  \leq \frac{1}{\ell}\mu \left( V_{z_0}\left(\xi_{j_1}; \frac{2\epsilon}{c_1}+\frac{\epsilon}{2}\right)\right) \leq \frac{M^{2d}}{\ell}\mu \left( V_{z_0}\left(\xi_{j_1}; \frac{\epsilon}{2}\right)\right).
\end{align*}
So $\ell \leq M^{2d}$. \hspace*{\fill}$\blacktriangleleft$

\medskip

Now by the previous lemma, for each $j \in J$ there exists $F_j : \Omega \rightarrow [0,1]$ such that 
\begin{align*}
i\partial \bar{\partial} F_{j} (z) \geq \frac{c_1}{\epsilon^{2/m_1}} i\partial \bar{\partial}\norm{z}^2 \text{ on } V_{z_0}(\xi_j; 2\epsilon) \cap \Omega
\end{align*}
and 
\begin{align*}
\supp(F_j ) \subset V_{z_0}\left(\xi_j; \frac{\epsilon}{c_1}\right).
\end{align*}

Finally we define 
\begin{align*}
F_{\epsilon} = \frac{1}{M^{2d}} \sum_{j \in J} F_j.
\end{align*}
Then $F_{\epsilon}$ is a smooth plurisubharmonic function, maps into $[0,1]$, and 
\begin{align*}
i\partial \bar{\partial} F_{\epsilon} (z) \geq \frac{c_2}{ \epsilon^{2/m_1}} i\partial \bar{\partial}\norm{z}^2 \text{ on } V_{\epsilon}
\end{align*}
where $c_2 = c_1 M^{-2d}$. 
\end{proof}

For $\delta > 0$ define 
\begin{align*}
S_\delta := \{ z \in \Omega : \exists \xi \in \partial \Omega \cap \Bb_d(0;R) \text{ such that } z \in [z_0,\xi) \text{ and } \norm{z-\xi} < \delta\}.
\end{align*}

\begin{lemma}\label{lem:neighborhood_inclusions} There exist $B > 0$ and a neighborhood $U$ of $\xi_0$ such that 
\begin{enumerate}
\item $S_{\delta} \subset V_{B\delta^{\lambda/2}}$ for all $\delta > 0$,
\item if $z \in U \cap \Omega$ and $\delta_\Omega(z) \leq \delta$, then $z \in S_{B\delta}$. 
\end{enumerate}
\end{lemma}

\begin{proof} By Proposition~\ref{prop:visual_metrics_along_line} there exists $B_0 \geq 1$ such that: if $q \in [z_0, \xi)$, then 
\begin{align*}
q \in V_{z_0}\left(\xi;B_0\norm{q-\xi}^{\lambda/2}\right).
\end{align*}
So 
\begin{align*}
S_{\delta} \subset V_{B_0\delta^{\lambda/2}}
\end{align*}
for all $\delta > 0$. 

Let $\delta_0 := \delta_\Omega(z_0)$ and pick $U$ a sufficiently small neighborhood of $\xi_0$ such that: if $z \in U \cap \Omega$, then there exists some $\xi \in  \partial \Omega \cap \Bb_d(0;R)$ with $z \in [z_0,\xi)$. 

Fix $\delta > 0$ and $z \in U \cap \Omega$ with $\delta_\Omega(z) \leq \delta$. Then there exists $\xi \in  \partial \Omega \cap \Bb_d(0;R)$ with $z \in [z_0,\xi)$. Since $\overline{\Omega}$ contains the convex hull of $\Bb_d(z_0;\delta_0)$ and $\xi$, we have 
\begin{align*}
\frac{\delta_0}{2R} \norm{z-\xi} \leq \delta_\Omega(z) \leq \delta.
\end{align*}
So $z \in S_{B_1\delta}$ where $B_1 = \frac{2R}{\delta_0}$. 

Then $B = \max\{ B_0, B_1\}$ satisfies the conclusion of the lemma. 
\end{proof}

\begin{proof}[Proof of Theorem~\ref{thm:final_plurisubharmonic_fcn}] Define
\begin{align*}
\delta_1 =\frac{1}{B^{2/\lambda}}\epsilon_2^{2/\lambda}.
\end{align*}
By Lemmas~\ref{lem:plurisub_two} and~\ref{lem:neighborhood_inclusions}, for each $\delta \in (0,\delta_1)$ there exists a smooth plurisubharmonic function $F_{\delta} : \Omega \rightarrow [0,1]$ such that  
\begin{align*}
i\partial \bar{\partial} F_{\delta} (z) \geq \frac{c_3}{\delta^{2/\ell}} i\partial \bar{\partial} \norm{z}^2 \text{ on } S_{\delta}\end{align*}
where $c_3 = c_2B^{-2/m_1}$ and $\ell  = 2m_1/\lambda$. 

Now we use the argument on page 464 in~\cite{S1997}: Pick $k_0 \in \Nb$ such that $2^{-k_0} < \delta_1$. Then pick any
\begin{align}
\label{eq:defn_of_m2}
m_2 > \max\left\{\frac{2m_1}{\lambda}, 2\right\}
\end{align}
and define
\begin{align*}
F(z) = \sum_{k=k_0}^{\infty} 2^{-2k(1/\ell-1/m_2)} F_{2^{-k}}.
\end{align*}
Since each $F_{2^{-k}}$ is bounded in absolute value by 1, the sum is uniformly convergent. Thus $F$ is a bounded continuous function. Since each $F_{2^{-k}}$ is plurisubharmonic, $F$ is as well. By decreasing $U$, we can assume that: if $z \in U \cap \Omega$, then $B\delta_\Omega(z) < 2^{-k_0}$. Now fix some $z \in U \cap \Omega$. Then there exists some $K \geq k_0$ such that
\begin{align*}
\frac{1}{2^{K+1}} \leq B\delta_\Omega(z) \leq \frac{1}{2^K}.
\end{align*}
Then $z \in S_{2^{-k}}$ for all $k_0 \leq k \leq K$. Hence there exists $c_4 > 0$ (independent of $z$) such that
\begin{align*}
i\partial \bar{\partial} F (z)
& \geq \sum_{k=k_0}^K  \frac{c_32^{2k/\ell}}{ 2^{2k(1/\ell-1/m_2)} } i\partial \bar{\partial} \norm{z}^2 \\
& \geq c_4 2^{2(K+1)/m_2} i\partial \bar{\partial} \norm{z}^2 \geq \frac{c_4}{B^{2/m_2}\delta_\Omega(z)^{2/m_2}}i\partial \bar{\partial} \norm{z}^2.
\end{align*}

Then let $G = \frac{1}{\sum_{k=k_0}^{\infty} 2^{-2k(1/\ell-1/m_2)} } F$. 
\end{proof}

\begin{remark}\label{rmk:value_of_ell} When $d \geq 2$, we always have $\max\left\{ \frac{2m_1}{\lambda},2\right\} =\frac{2m_1}{\lambda}$ in Equation~\eqref{eq:defn_of_m2}. To see this, first observe that Equation~\eqref{eq:m_1_equality} implies that
\begin{align*}
m_1  = \frac{\alpha \lambda m_0}{2}
\end{align*}
where $\alpha \geq 1$ is the constant in Lemma~\ref{lem:dist_est_normalizing_maps} and $m_0 > 0$ is the constant from Corollary~\ref{cor:m_convex_Gromov_hyp}. Remark~\ref{rmk:m_convex_values} implies that $m_0 \geq 2$. Thus 
\begin{align*}
 \frac{2m_1}{\lambda} = \frac{2}{\lambda}\frac{\alpha \lambda m_0}{2} = \alpha m_0 \geq 2. 
\end{align*}
\end{remark}

\section{Proof of Theorem~\ref{thm:intersection}}\label{sec:pf_of_thm_intersection}

In this section we prove the following strengthening of Theorem~\ref{thm:intersection}. 

\begin{theorem}\label{thm:intersection_body} Suppose $\Omega_1,\dots, \Omega_n \subset \Cb^d$ are $\Cb$-properly convex domains and each $(\Omega_j, d_{\Omega_j})$ is Gromov hyperbolic. If $\Omega := \cap_{j=1}^n \Omega_j$ is bounded and non-empty, then $\Omega$ satisfies a subelliptic estimate.  \end{theorem}

For the rest of the section fix $\Omega = \cap_{j=1}^n \Omega_j$ as in the statement of Theorem~\ref{thm:intersection_body}.

\begin{lemma} For every $\xi \in \partial \Omega$, there is a neighborhood $W$ of $\xi$, $C>0$, $m > 2$, and a bounded continuous plurisubharmonic function $G : W \cap \Omega \rightarrow [0,1]$ 
such that 
\begin{align*}
i\partial \bar{\partial} G(z) \geq \frac{C}{\delta_{\Omega}(z)^{2/m}} i\partial \bar{\partial} \norm{z}^2 \text{ on } W \cap \Omega.
\end{align*}
\end{lemma}

\begin{proof} By relabeling we can suppose that $\xi \in \partial \Omega_j$ for $1 \leq j \leq \ell$ and $\xi \in \Omega_j$ for $\ell+1 \leq j \leq n$. Then there exists a neighborhood $U_0$ of $\xi$ such that: if $z \in U_0 \cap \Omega$, then
\begin{align*}
\delta_\Omega(z) = \min_{1 \leq j \leq \ell} \delta_{\Omega_j}(z).
\end{align*}
By Theorem~\ref{thm:final_plurisubharmonic_fcn}, for each $1 \leq j \leq \ell$, there exist constants $C_j>0$, $m_j > 2$, a neighborhood $U_j$ of $\xi$, and a bounded continuous plurisubharmonic function $G_j : U_j \cap \Omega_j \rightarrow [0,1]$ 
such that 
\begin{align*}
i\partial \bar{\partial} G_j(z) \geq \frac{C_j}{\delta_{\Omega_j}(z)^{2/m_j}} i\partial \bar{\partial} \norm{z}^2 \text{ on } U_j \cap \Omega_j.
\end{align*}
Then $G= \frac{1}{\ell}\sum_{j=1}^\ell G_j$ satisfies the conclusion of the lemma with $W = \cap_{j=0}^\ell U_j$, $C = \frac{1}{\ell}\min_{1\leq j \leq \ell} C_j$, and $m = \max_{1 \leq j \leq \ell} m_j$. 
\end{proof}

So by Straube's theorem (Theorem~\ref{thm:straube} above) for each $\xi \in \partial \Omega$ there exist constants $C_\xi >0, m_\xi > 2$ and a neighborhood $V_\xi$ of $\xi$ in $\Cb^d$ such that 
\begin{align*}
\norm{u}_{\frac{1}{m_\xi}, V_\xi \cap \Omega} \leq C_\xi ( \|\bar{\partial} u\|_0 + \|\bar{\partial}^* u\|_0)
\end{align*}
for all $u \in L^2_{(0,q)}(\Omega) \cap { \rm dom}(\bar{\partial}) \cap { \rm dom}(\bar{\partial}^*)$.  Since $\partial \Omega$ is compact, we can find $\xi_1,\dots, \xi_N \in \partial \Omega$ such that if $V_j := V_{\xi_j}$, then
\begin{align*}
\partial \Omega \subset \cup_{1 \leq j \leq N} V_{j}.
\end{align*} 
Let $C_j = C_{\xi_j}$ and $m = \max_{1 \leq j \leq N} m_{\xi_j}$.

Next fix a relatively compact open set $V_0 \subset \Omega$ where $\overline{\Omega} \subset \cup_{j=0}^N V_j$. Using standard interior estimates, see for instance Proposition 5.1.1 and Equation (4.4.6) in~\cite{CS2001}, we have the following estimate.

\begin{lemma} There exists $C_0 > 0$ such that: 
\begin{align*}
\norm{u}_{\frac{1}{m}, V_0} \leq C_0 ( \|\bar{\partial} u\|_0 + \|\bar{\partial}^* u\|_0)
\end{align*}
for every $u \in L^2_{(0,q)}(\Omega) \cap { \rm dom}(\bar{\partial}) \cap { \rm dom}(\bar{\partial}^*)$.
\end{lemma}

Let $\Vc = \cup_{j=0}^N V_j$. Then let $\chi_0,\dots, \chi_N$ be a smooth partition of unity subordinate to the open cover $\Vc = \cup_{j=0}^N V_j$, that is:
\begin{enumerate}
\item each $\chi_j :  \Vc \rightarrow [0,1]$ is smooth and ${\rm supp}(\chi_j) \subset V_j$,
\item $\sum_{j=0}^N \chi_j = 1$ on $\Vc$.
\end{enumerate}
Since $\overline{\Omega} \subset \Vc$, there exists some constant $B>0$ such that: if $0 \leq j \leq N$ and $u|_{V_j \cap \Omega} \in W^{2,1/m}_{0,q}(V_j \cap \Omega)$, then 
\begin{align*}
\norm{\chi_j u}_{\frac{1}{m}, \Omega} \leq B\norm{u}_{\frac{1}{m}, V_{j} \cap \Omega}.
\end{align*}
Finally, if $u \in L^2_{(0,q)}(\Omega) \cap { \rm dom}(\bar{\partial}) \cap { \rm dom}(\bar{\partial}^*)$, then
\begin{align*}
\norm{u}_{\frac{1}{m}} \leq \sum_{j=0}^N \norm{\chi_j u}_{\frac{1}{m}} \leq B \sum_{j=0}^N\norm{u}_{\frac{1}{m}, V_{j} \cap \Omega} \leq B\left( \max_{0 \leq j \leq N} C_j \right)(N+1) \left(   \|\bar{\partial} u\|_0 + \|\bar{\partial}^* u\|_0 \right).
\end{align*}

\section{The order of subelliptic estimate}\label{sec:optimal_estimate}

In this section we describe the order of subelliptic estimate obtained by our argument in the special case of a bounded convex domain with Gromov hyperbolic Kobayashi metric. 

For a bounded convex domain $\Omega \subset \Cb^d$, define 
\begin{align*}
m_\star(\Omega):= \inf\{ m \geq 1: \Omega \text{ is $m$-convex} \} \in [1,\infty].
\end{align*}
By Remark~\ref{rmk:m_convex_values}, if $d=1$, then $m_\star(\Omega)=1$ and if $d \geq 2$, then $m_\star(\Omega) \geq 2$. Further, by Corollary~\ref{cor:m_convex_Gromov_hyp}, if $(\Omega, K_\Omega)$ is Gromov hyperbolic, then $m_\star(\Omega) < \infty$. 

We say that a bounded convex domain $\Omega \subset \Cb^d$ is \emph{$\alpha$-regular} if for any $z_0 \in \Omega$ there exists some $B=B(\alpha,z_0) > 0$ such that 
\begin{align*}
K_\Omega(q, z_0) \leq B + \frac{\alpha}{2} \log \frac{1}{\delta_\Omega(q)} 
\end{align*}
for all $q \in \Omega$. Then define 
\begin{align*}
\alpha_\star(\Omega):= \inf\{ \alpha >0: \Omega \text{ is $\alpha$-regular} \} \in [1,\infty)
\end{align*}
Lemma~\ref{lem:hyperplanes} implies that $\alpha_\star(\Omega) \geq 1$ and Proposition~\ref{prop:quasi_geodesic} implies that $\alpha_\star(\Omega) < \infty$. 

\begin{theorem}\label{thm:optimal_order} Suppose $d \geq 2$, $\Omega \subset \Cb^d$ is a bounded convex domain, and $(\Omega, K_\Omega)$ is Gromov hyperbolic. If 
\begin{align*}
\epsilon <  \frac{1}{\alpha_\star(\Omega) m_\star(\Omega)},
\end{align*}
then a subelliptic estimate of order $\epsilon$ holds on $\Omega$. 
\end{theorem}

Before proving Theorem~\ref{thm:optimal_order} we calculate $\alpha_\star$ and $m_\star$ for some classes of domains. 

\begin{proposition}\label{prop:alpha_star_reg} Suppose $\Omega \subset \Cb^d$ is a bounded convex domain and $\partial \Omega$ is $C^{1}$. If $\epsilon > 0$ and $z_0 \in \Omega$, then there exists 
$B =B(\epsilon, z_0)> 0$ such that 
\begin{align*}
K_\Omega(q, z_0) \leq B + \frac{1+\epsilon}{2} \log \frac{1}{\delta_\Omega(q)} 
\end{align*}
for all $q \in \Omega$. In particular, $\alpha_\star(\Omega) = 1$. 
\end{proposition}

\begin{proof} For $\xi \in \partial \Omega$ let $n_{\xi} \in \Cb^d$ denote the inward pointing unit normal vector of $\partial \Omega$ at $\xi$. 

Fix $\epsilon > 0$. Since $\partial \Omega$ is $C^1$, there exists $t_0=t_0(\epsilon) > 0$ such that 
$$
\delta_\Omega(\xi + t n_{\xi}) \geq \frac{1}{1+\epsilon} t
$$ 
for all $\xi \in \partial \Omega$ and $t \in (0,t_0)$. 

For $\xi \in \partial \Omega$ let $\gamma_\xi: [0,\infty) \rightarrow \Omega$ denote the curve 
$$
\gamma_\xi(t) = \xi+ t_0e^{-2t}n_\xi.
$$
Then for $0 \leq t$ we have
\begin{align*}
K_\Omega(\gamma_{\xi}(t), \gamma_{\xi}(0)) &\leq \int_0^t k_\Omega( \gamma_{\xi}(r); \gamma_{\xi}^\prime(r))dr \leq \int_0^t \frac{\norm{ \gamma_{\xi}^\prime(r)}}{\delta_\Omega( \gamma_{\xi}(r))}dr \\
&  \leq (1+\epsilon)  \int_0^t dr = (1+\epsilon)t.
\end{align*}

Now fix $q \in \Omega$. Then $q = \xi + \delta_\Omega(q) n_\xi$ where $\xi \in \partial \Omega$ is a point in $\partial \Omega$ closest to $q$. If $\delta_\Omega(q) \geq t_0$, then 
$$
K_\Omega(q,z_0) \leq B_0 + \frac{1+\epsilon}{2} \log \frac{1}{\delta_\Omega(q)} 
$$
where 
$$
B_0 = \frac{1+\epsilon}{2} \log \frac{1}{t_0} +\max\{ K_\Omega(z,z_0) : z \in \Omega \text{ and } \delta_\Omega(z) \geq t_0\}.
$$
If $\delta_\Omega(q) \leq t_0$, then $q = \gamma_\xi(t)$ where $t = \frac{1}{2} \log \frac{t_0}{\delta_\Omega(q)} \geq 0$. Then 
\begin{align*}
K_\Omega(q, z_0) & \leq K_\Omega( q, \gamma_\xi(0)) + K_\Omega(\gamma_\xi(0),z_0) \\
& \leq  B_1+ \frac{1+\epsilon}{2} \log \frac{1}{\delta_\Omega(q)}
\end{align*} 
where 
$$
B_1 = \frac{1+\epsilon}{2} \log t_0 + \max \{ K_\Omega(\gamma_\xi(0),z_0)  : \xi \in \partial \Omega\}.
$$
Since $B = \max\{B_1,B_2\}$ does not depend on $q$ this completes the proof. 

\end{proof}

Next we compute $m_\star(\Omega)$ in the special case when $\partial \Omega$ is $C^\infty$. To do this we need to define the line type at a boundary point. Given a function $f: \Cb \rightarrow \Rb$ with $f(0)=0$ let $\nu(f)$ denote the order of vanishing of $f$ at $0$. Suppose that $D \subset \Cb^d$ is a domain with $C^\infty$ boundary and 
\begin{align*}
D = \{ z \in \Cb^d : r(z) < 0\}
\end{align*}
 where $r$ is a $C^\infty$ function with $\nabla r \neq 0$ near $\partial D$. The \emph{line type of a boundary point} $\xi \in \partial D$  is defined to be
\begin{align*}
\ell(D,\xi)=\sup \{ \nu( r \circ \psi) | \ \psi : \Cb \rightarrow \Cb^d & \text{ is a non-constant complex affine map }\\
& \text{ with $\psi(0)=\xi$} \}.
\end{align*}
Notice that $\nu(r\circ \psi) \geq 2$ if and only if $\psi(\Cb)$ is tangent to $D$. McNeal~\cite{M1992} proved that if $D$ is convex  then $\xi \in \partial \Omega$ has finite line type if and only if it has finite type in the sense of D'Angelo (also see~\cite{BS1992}). 

\begin{proposition} Suppose $d \geq 2$, $\Omega \subset \Cb^d$ is a bounded convex domain, and $\partial \Omega$ is $C^{\infty}$. Then 
\begin{align*}
m_\star(\Omega) = \max_{\xi \in \partial \Omega} \ell(\Omega, \xi).
\end{align*}
\end{proposition}

\begin{proof} This is a straightforward calculation, see for instance~\cite[Section 9]{Z2016}. \end{proof}

\subsection{Proof of Theorem~\ref{thm:optimal_order}} This is simply a matter of tracking the constants in the proof of 
Theorem~\ref{thm:intersection}. 

Fix 
\begin{align*}
\epsilon <  \frac{1}{\alpha_\star(\Omega) m_\star(\Omega)}
\end{align*}
and let $m := \epsilon^{-1}$. Then there exist $m_0 \geq 2$, $\alpha \geq 1$, and $z_0 \in \Omega$ such that 
\begin{enumerate}
\item $m > m_0 \alpha$,
\item $\Omega$ is $m_0$-convex,
\item $\Omega$ is $\alpha$-regular. 
\end{enumerate}
Fix $z_0 \in \Omega$ and let $\lambda$ be the constant associated to $d_{z_0}$ in Sections~\ref{sec:vm_and_normalizing} and~\ref{sec:psh_on_general_domains}. 

Notice that, by definition, $\alpha$ satisfies Lemma~\ref{lem:dist_est_normalizing_maps} and so by Equation~\eqref{eq:m_1_equality}
\begin{align*}
m_1 : = \frac{\alpha \lambda m_0}{2}
\end{align*}
satisfies the conclusion of Proposition~\ref{prop:normalizing_maps_visual_metric}. Hence $m_1$ also satisfies the conclusion of Lemmas~\ref{lem:plurisub_one} and~\ref{lem:plurisub_two} (see Remark~\ref{rmk:m_1_equals}). Then by Equation~\eqref{eq:defn_of_m2} and Remark~\ref{rmk:value_of_ell}, any
\begin{align*}
m_2 > \frac{2m_1}{\lambda} = \alpha m_0
\end{align*}
satisfies the conclusion of Theorem~\ref{thm:final_plurisubharmonic_fcn}. In particular, $m$ does. Then Straube's theorem (Theorem~\ref{thm:straube} above) implies that a local subelliptic estimate of order $\epsilon=\frac{1}{m}$ holds at every boundary point. Then by the ``local to global'' proof in Section~\ref{sec:pf_of_thm_intersection} we see that a subelliptic estimate of order $\epsilon=\frac{1}{m}$ holds on $\Omega$.

\part{Examples}

\section{The Hilbert distance}\label{sec:intro_to_hilbert_metric}

In this expository section we recall the definition of the Hilbert distance and then state some of its properties. 

Suppose $\Omega \subset \Rb^d$ is a convex domain. Given $x,y \in \Omega$ distinct let $L_{x,y}$ be the real line containing them  and let $a,b \in \partial \Omega \cup \{\infty\}$ be the endpoints of $\overline{\Omega} \cap L_{x,y}$ with the ordering $a,x,y,b$. Then define the \emph{Hilbert pseudo-distance between $x,y$} to be
\begin{align*}
H_\Omega(x,y) = \frac{1}{2} \log \frac{ \norm{x-b}\norm{y-a}}{ \norm{y-b}\norm{x-a}}
\end{align*}
where we define 
\begin{align*}
\frac{ \norm{x-\infty}}{\norm{y-\infty}}=\frac{ \norm{y-\infty}}{\norm{x-\infty}}=1.
\end{align*}
In the case when $\Omega$ does not contain any affine real lines, we see that $H_\Omega(x,y) > 0$ for all $x,y \in \Omega$ distinct. This motivates the following definition. 

\begin{definition} A convex domain $\Omega \subset \Rb^d$ is called \emph{$\Rb$-properly} convex if $\Omega$ does not contain any affine real lines. \end{definition}

\begin{theorem}\label{thm:Hilbert_basic_properties} \ 
\begin{enumerate} \item If $\Omega \subset \Rb^d$ is a $\Rb$-properly convex domain, then $(\Omega, H_\Omega)$ is a proper geodesic metric space. For $x,y \in \Omega$ distinct, there exists a geodesic line $\gamma : \Rb \rightarrow \Omega$ whose image is $L_{x,y} \cap \Omega$. 
\item If $\Omega \subset \Rb^d$ is a convex domain and $V \subset \Rb^d$ is an affine subspace intersecting $\Omega$, then 
\begin{align*}
H_\Omega(x,y) = H_{\Omega \cap V}(x,y)
\end{align*}
for all $x,y \in \Omega \cap V$. 
\item If $\Omega \subset \Rb^d$ is a convex domain and $A \in \Aff(\Rb^d)$ is an affine automorphism of $\Rb^d$, then 
\begin{align*}
H_\Omega(x,y) = H_{A\Omega}(Ax,Ay)
\end{align*}
for all $x,y \in \Omega$. 
\end{enumerate}
\end{theorem}

Properties (2) and (3) in Theorem~\ref{thm:Hilbert_basic_properties} are immediate from the definition and a proof of Property (1) can be found in~\cite[Section 28]{BK1953}.

We also can define an infinitesimal Hilbert pseudo-metric. Given $x \in \Omega$ and a non-zero $v \in \Rb^d$ let $a,b \in \partial \Omega \cup \{\infty\}$ be the endpoints of $\overline{\Omega} \cap (x+ \Rb \cdot v)$. Then define the \emph{Hilbert norm of $v$ at $x$} to be
\begin{align*}
h_\Omega(x;v) = \frac{\norm{v}}{2} \left( \frac{1}{\norm{x-a}} + \frac{1}{\norm{x-b}} \right). 
\end{align*}
Given a piecewise $C^1$ curve $\sigma : [0,1] \rightarrow \Omega$ we define the \emph{Hilbert length of $\sigma$} to be 
\begin{align*}
\ell_{H,\Omega}(\sigma) : = \int_0^1 h_\Omega(\sigma(t); \sigma^\prime(t)) dt.
\end{align*} 

It is fairly straightforward to establish the following. 

\begin{proposition} If $\Omega \subset \Rb^d$ is a properly convex domain, then 
\begin{align*}
H_\Omega(x,y) = \inf\left\{ \ell_{H,\Omega}(\sigma) : \sigma : [0,1] \rightarrow \Omega \text{ is piecewise $C^1$}, \sigma(0)=x, \sigma(1) = y\right\}.
\end{align*}
\end{proposition}

We will also use the following result of Karlsson and Noskov.

\begin{theorem}[Karlsson-Noskov~\cite{KN2002}]\label{thm:KN} Suppose $\Omega \subset \Rb^d$ is a $\Rb$-properly convex domain. If $(\Omega, H_\Omega)$ is Gromov hyperbolic, then 
\begin{enumerate}
\item $\Omega$ is strictly convex (that is, $\partial \Omega$ does not contain any line segments of positive length),
\item $\partial \Omega$ is a $C^1$ hypersurface. 
\end{enumerate}
\end{theorem}

Next we consider the space of $\Rb$-properly convex domains. 

\begin{definition} \ \begin{enumerate}
\item Let $\Yb_d$ denote the space of $\Rb$-properly convex domains in $\Rb^d$ endowed with the local Hausdorff topology. 
\item Let $\Yb_{d,0} = \{ (\Omega, x) : \Omega \in \Yb_d, x \in \Omega\} \subset \mathbb{Y}_d \times \Rb^d$. 
\end{enumerate}
\end{definition}

As in the complex case, the group $\Aff(\Rb^d)$ acts co-compactly on $\Yb_{d,0}$. 

\begin{theorem}[Benz\'{e}cri~\cite{B1959}]\label{thm:real_affine_cocompact}
The group $\Aff(\Rb^d)$ acts co-compactly on $\Yb_{d,0}$, that is, there exists a compact set $K \subset \Yb_{d,0}$ such that $\Aff(\Rb^d) \cdot K = \Yb_{d,0}$. 
\end{theorem}

\begin{remark} To be precise, Benz\'{e}cri established a real projective variant of the above result which easily implies Theorem~\ref{thm:real_affine_cocompact}. A direct proof can also be found in~\cite{F1991}.
\end{remark}

Using the definition of the Hilbert distance it is not difficult to observe that the Hilbert distance is continuous on $\Yb_d$. 

\begin{observation}\label{obs:Hilbert_metric_conv} Suppose $\Omega_n \subset \Rb^d$ is a sequence of convex domains converging to a convex domain $\Omega$ in the local Hausdorff topology. Then 
\begin{align*}
H_\Omega = \lim_{n \rightarrow \infty} H_{\Omega_n}
\end{align*}
locally uniformly on $\Omega \times \Omega$. 
\end{observation}

As a consequence of Theorem~\ref{thm:KN} and Observation~\ref{obs:Hilbert_metric_conv} we have the following. 

\begin{corollary}\label{cor:hilbert_gromov_necessary} Suppose $\Omega \subset \Rb^d$ is a $\Rb$-properly convex domain and $(\Omega, H_\Omega)$ is Gromov hyperbolic. Then
\begin{enumerate}
\item if $\mathcal{D} \in \overline{\Aff(\Rb^d) \cdot \Omega} \cap \Yb_d$, then $(\mathcal{D},H_\mathcal{D})$ is Gromov hyperbolic, 
\item every domain in $\overline{\Aff(\Rb^d) \cdot \Omega} \cap \Yb_d$ is strictly convex,
\item every domain in $\overline{\Aff(\Rb^d) \cdot \Omega} \cap \Yb_d$ has $C^1$ boundary. 
\end{enumerate}
\end{corollary}

Recently, Benoist completely characterized the convex domains which have Gromov hyperbolic Hilbert metric in terms of the derivatives of local defining functions. To state his result we need some definitions. 

\begin{definition} Suppose $\mathcal{U} \subset \Rb^d$ is an open set and $F: \mathcal{U} \rightarrow \Rb$ is a $C^1$ function. Then for $x,x+h \in \mathcal{U}$ define
\begin{align*}
 D_x(h) :=F(x+h)-F(x)-F^\prime(x) \cdot h.
\end{align*}
Then $F$ is said to be \emph{quasi-symmetric} if there exists $H \geq 1$ so that 
\begin{align*}
D_x(h) \leq H D_x(-h)
\end{align*}
whenever $x,x+h,x-h \in \mathcal{U}$.
\end{definition}

\begin{definition}\label{defn:quasi_symmetric}
Suppose $\Omega \subset \Rb^d$ is a bounded convex domain. Then $\Omega$ is said to have \emph{quasi-symmetric boundary} if its boundary is $C^1$ and is everywhere locally the graph of a quasi-symmetric function. 
\end{definition}

\begin{theorem}[{Benoist~\cite[Theorem 1.4]{B2003}}]
 Suppose $\Omega \subset \Rb^d$ is a bounded convex domain. Then the following are equivalent:
\begin{enumerate}
 \item $(\Omega,H_{\Omega})$ is Gromov hyperbolic,
\item $\Omega$ has quasi-symmetric boundary.
\end{enumerate}
\end{theorem}

\section{Proof of Corollary~\ref{cor:Hilbert_metric}}\label{sec:pf_of_cor_hilbert_metric}

In this section we prove Corollary~\ref{cor:Hilbert_metric}. For the rest of the section suppose that $\Omega \subset \Cb^d$ is a bounded convex domain and $(\Omega, H_\Omega)$ is Gromov hyperbolic. Suppose for a contradiction that  $(\Omega, K_\Omega)$ is not Gromov hyperbolic. 

Since $(\Omega, K_\Omega)$ is not Gromov hyperbolic, Theorem~\ref{thm:main_equivalence}  implies that there exist affine maps $A_n \in \Aff(\Cb^d)$ such that $A_n \Omega \rightarrow \Omega_\infty$ in $\Xb_d$ and $\partial \Omega_\infty$ has non-simple boundary. Then by Proposition~\ref{prop:affine_vs_holomorphic_disks}, $\partial \Omega_\infty$ contains an affine disk. Then without loss of generality we can assume that $0 \in \Omega_\infty$ and $e_1 + \Db \cdot e_2 \subset \partial \Omega_\infty$. Pick $\lambda \in \Cb$ such that $-\lambda e_2 \in \partial \Omega_\infty$ and $\norm{0-\lambda e_2} = \delta_{\Omega_\infty}(0;e_2)$. By rotating $\Omega_\infty$ we can assume, in addition, that $\lambda \in \Rb_{>0}$. 

Let $V = \Span_{\Rb}\{e_1,e_2\}$ and $C=V \cap \Omega_\infty$. 

\medskip

\noindent \textbf{Claim:} $C$ is a $\Rb$-properly convex domain in $V$ and $(C,H_C)$ is not Gromov hyperbolic. 

\medskip

\noindent \emph{Proof of Claim:}
By construction $e_1 + [-1,1] \cdot e_2 \subset \partial C$ which implies by convexity that 
\begin{align}
\label{eq:boundary_line}
(e_1 + \Rb \cdot e_2) \cap C = \emptyset.
\end{align}
Further $\lambda e_2 \in \partial C$. We claim that $C$ is $\Rb$-properly convex. Suppose that $a+\Rb \cdot v \subset C$ for some $a,v \in \Span_{\Rb}\{e_1,e_2\}$. Since $0 \in C$, the real analogue of Observation~\ref{obs:asymptotic_cone_1} implies that $\Rb \cdot v \subset C$. If $v=v_1e_1+v_2e_2$, then Equation~\eqref{eq:boundary_line} implies that $v_1 = 0$. Then, since $\lambda e_2 \in \partial C$, we must have $v_2 = 0$. So $v=0$ and hence $C$ is $\Rb$-properly convex. Finally, since $e_1 + [-1,1] \cdot e_2 \subset \partial C$, Theorem~\ref{thm:KN} implies that $(C,H_C)$ is not Gromov hyperbolic. \hspace*{\fill}$\blacktriangleleft$

\medskip

For a convex domain $\mathcal{D} \subset \Cb^d$ and $x,y,z \in \mathcal{D}$ define the Gromov product associated to $H_\mathcal{D}$ by
\begin{align*}
(x|y)_z^{H,\mathcal{D}} : = \frac{1}{2} \left( H_\mathcal{D}(x,z) + H_\mathcal{D}(y,z) - H_\mathcal{D}(x,y) \right). 
\end{align*}
Since $(\Omega, H_\Omega)$ is Gromov hyperbolic, there exists $\delta > 0$ such that 
\begin{align*}
(x|z)_w^{H,\Omega}  \geq \min\left\{ (x|y)_w^{H,\Omega}, (y|z)_w^{H,\Omega}\right\} - \delta
\end{align*}
for every $x,y,z,w \in \Omega$. So by Theorem~\ref{thm:Hilbert_basic_properties} part (3) and Observation~\ref{obs:Hilbert_metric_conv} 
\begin{align*}
(x|z)_w^{H,\Omega_\infty}  \geq \min\left\{ (x|y)_w^{H,\Omega_\infty}, (y|z)_w^{H,\Omega_\infty}\right\} - \delta
\end{align*}
for every $x,y,z,w \in \Omega_\infty$ (notice that $\Omega_\infty$ may not be $\Rb$-properly convex and so $H_{\Omega_\infty}$ may not be a distance on $\Omega_\infty$, but this doesn't matter). So by Theorem~\ref{thm:Hilbert_basic_properties} part (2) \begin{align*}
(x|z)_w^{H,C}  \geq \min\left\{ (x|y)_w^{H,C}, (y|z)_w^{H,C}\right\} - \delta
\end{align*}
for every $x,y,z,w \in C$. But then $(C,H_C)$ is Gromov hyperbolic which contradicts the claim.

\section{Tube domains}\label{sec:tube_domains}

In this section we establish Corollary~\ref{cor:tube_domains} by proving Propositions~\ref{prop:one_direction} and~\ref{prop:the_other_direction} below. 

\begin{proposition}\label{prop:one_direction} Suppose $d \geq 2$, $C \subset \Rb^d$ is a $\Rb$-properly convex domain, and $\Omega = C + i\Rb^d$. If $(\Omega, K_\Omega)$ is Gromov hyperbolic, then $(C,H_C)$ is Gromov hyperbolic and $C$ is unbounded. \end{proposition}

Before proving the proposition we establish two lemmas. 

\begin{lemma}\label{lem:QI_embedding} Suppose $C \subset \Rb^d$ is a $\Rb$-properly convex domain and $\Omega = C + i\Rb^d$. Then 
\begin{align*}
K_\Omega(c_1,c_2) \leq H_C(c_1,c_2) \leq 2K_\Omega(c_1,c_2)
\end{align*}
for all $c_1,c_2 \in C$. 
\end{lemma}

\begin{remark} When $C$ is bounded, Pflug and Zwonek~\cite[Proposition 15]{PZ2018} proved that $K_\Omega(c_1,c_2) \leq H_C(c_1,c_2)$ for $c_1, c_2 \in C$. 
\end{remark}

\begin{proof} Using Proposition~\ref{prop:convergence_of_kob} and Observation~\ref{obs:Hilbert_metric_conv} it suffices to prove the lemma in the case when $C$ is bounded. Then by a result of Pflug and Zwonek~\cite[Proposition 15]{PZ2018} we have
\begin{align*}
K_\Omega(c_1,c_2) \leq H_C(c_1,c_2)
\end{align*}
for all $c_1,c_2 \in C$. 

For $c \in C$ and $v \in \Rb^d$ non-zero define
\begin{align*}
\delta_C(c;v) = \min\{ \norm{\xi-c} : \xi \in (c+\Rb \cdot v) \cap \partial C \}
\end{align*}
and define $\delta_C(c;0) = \infty$. Then, by definition, 
\begin{align}
\label{eq:hilbert_inf_est}
\frac{\norm{v}}{2\delta_C(c;v)} \leq h_C(c;v) \leq \frac{\norm{v}}{\delta_C(c;v)}
\end{align}
for all $c \in C$ and $v \in \Rb^d$. Then let $P: \Rb^d + i \Rb^d \rightarrow \Rb^d$ be the projection $P(x+iy) = x$. Notice that
\begin{align}
\label{eq:real_vs_complex_dist_to_bd_2}
\delta_\Omega(z;v) \leq \delta_C(P(z);P(v))
\end{align}
for all $z \in \Omega$ and non-zero $v \in \Cb^d$.

Fix $c_1, c_2 \in C$ and let $\sigma: [0,1] \rightarrow \Omega$ be a piecewise $C^1$ curve with $\sigma(0)=c_1$ and $\sigma(1) = c_2$. Then by Equation~\eqref{eq:hilbert_inf_est}, Equation~\eqref{eq:real_vs_complex_dist_to_bd_2}, and Lemma~\ref{lem:kob_inf_bound}
\begin{align*}
\ell_{H,C}(P\circ \sigma)& = \int_0^1 h _C(P\sigma(t); P\sigma^\prime(t)) dt \leq \int_0^1 \frac{\norm{P(\sigma^\prime(t))}}{\delta_C(P(\sigma(t)); P(\sigma^\prime(t)))} dt \\
& \leq  \int_0^1 \frac{\norm{\sigma^\prime(t)}}{\delta_\Omega(\sigma(t); \sigma^\prime(t))} dt \leq 2 \int_0^1 k_\Omega(\sigma(t); \sigma^\prime(t)) dt =2\ell_{\Omega}(\sigma).
\end{align*}
So
\begin{align*}
H_C(c_1,c_2) \leq 2 \ell_\Omega(\sigma).
\end{align*}
Then taking the infimum over all such curves we see that 
\begin{equation*}
H_C(c_1, c_2) \leq 2 K_\Omega(c_1,c_2). \qedhere
\end{equation*}
\end{proof}

\begin{lemma}\label{lem:QI_flat_embedding} Suppose $C \subset \Rb^d$ is a bounded convex domain and $\Omega = C + i\Rb^d$. If $c_0 \in C$, then there exists $A=A(c_0) \geq 1$ such that 
\begin{align*}
\frac{1}{A}\norm{y_1-y_2} \leq K_\Omega(c_0+iy_1,c_0+iy_2) \leq  A\norm{y_1-y_2}
\end{align*}
for all $y_1,y_2 \in \Rb^d$. 
\end{lemma}

\begin{proof}

Since $C$ is bounded, there exists $A_1 > 0$ such that 
\begin{align*}
\delta_\Omega(z;v) \leq A_1
\end{align*}
for all $z \in C$ and $v \in \Cb^d$ non-zero. Since $\Omega$ is invariant under translations of the form $z \rightarrow z+iy$ with $y \in \Rb^d$, this implies that 
\begin{align*}
\delta_\Omega(z;v) \leq A_1
\end{align*}
for all $z \in \Omega$ and $v \in \Cb^d$ non-zero. Then by Lemma~\ref{lem:kob_inf_bound}
\begin{align*}
K_\Omega(z_1,z_2) \geq \frac{1}{2A_1} \norm{z_1-z_2}
\end{align*}
for all $z_1, z_2 \in \Omega$. 

Next, since $\Omega$ is invariant under translations of the form $z \rightarrow z+iy$ with $y \in \Rb^d$, we see that 
\begin{align*}
\delta_\Omega(c_0+iy) = \delta_\Omega(c_0)
\end{align*}
for every $y \in \Rb^d$. Now fix $y_1, y_2 \in \Rb^d$ and define $\sigma : [0,1] \rightarrow \Omega$ by $\sigma(t) = (1-t)(c_0+iy_1) + t(c_0+iy_2)$. Then Lemma~\ref{lem:kob_inf_bound} implies that
\begin{align*}
K_\Omega(c_0+iy_1,c_0+iy_2) \leq \int_0^1 k_\Omega(\sigma(t); \sigma^\prime(t)) dt  \leq \int_0^1 \frac{\norm{y_2-y_1}}{ \delta_\Omega(c_0)} dt =  \frac{\norm{y_2-y_1}}{ \delta_\Omega(c_0)}.
\end{align*}
So the Lemma is true with 
\begin{equation*}
A: = \max\{ 2A_1, \delta_\Omega(c_0)^{-1}\}. \qedhere
\end{equation*}
\end{proof}

\begin{proof}[Proof of Proposition~\ref{prop:one_direction}] By Lemma~\ref{lem:QI_embedding}, the inclusion map $(C,H_C) \hookrightarrow (\Omega, K_\Omega)$ is a quasi-isometric embedding. Then $(C,H_C)$ is Gromov hyperbolic, see~\cite[Chapter III.H, Theorem 1.9]{BH1999}. 

If $C$ is bounded and $c_0 \in C$, then Lemma~\ref{lem:QI_flat_embedding} implies that the map 
\begin{align*}
y \in (\Rb^d,d_{\Euc}) \rightarrow c_0+iy \in (\Omega, K_\Omega)
\end{align*}
is an quasi-isometric embedding. But since $(\Omega, K_\Omega)$ is Gromov hyperbolic and $d \geq 2$, this is impossible. So $C$ must be unbounded. 

\end{proof}

\begin{proposition}\label{prop:the_other_direction} Suppose $C \subset \Rb^d$ is a $\Rb$-properly convex domain and $\Omega = C + i\Rb^d$. If $(C,H_C)$ is Gromov hyperbolic and $C$ is unbounded, then $(\Omega, K_\Omega)$ is Gromov hyperbolic. \end{proposition}

We will need one lemma before proving the proposition. 

\begin{lemma}\label{lem:affine_orbits_equal} Suppose $C \subset \Rb^d$ is a $\Rb$-properly convex domain and $\Omega = C + i\Rb^d$. Then 
\begin{align*}
\overline{\Aff(\Cb^d) \cdot \Omega} \cap \Xb_d = \Aff(\Cb^d) \cdot \Big( \overline{\Aff(\Rb^d) \cdot C} \cap \Yb_d + i\Rb^d \Big).
\end{align*}
In particular, the following are equivalent 
\begin{enumerate}
\item every domain in $\overline{\Aff(\Rb^d) \cdot C} \cap \Yb_{d}$ is strictly convex
\item every domain in $\overline{\Aff(\Cb^d) \cdot \Omega} \cap \Xb_{d}$ has simple boundary. 
\end{enumerate}
\end{lemma}

\begin{proof} Since every map $A \in \Aff(\Rb^d)$ extends to a map in $\Aff(\Cb^d)$ we see that 
\begin{align*}
 \Aff(\Cb^d) \cdot \Big( \overline{\Aff(\Rb^d) \cdot C} \cap \Yb_d + i\Rb^d \Big) \subset \overline{\Aff(\Cb^d) \cdot \Omega} \cap \Xb_d.
\end{align*}

For the other inclusion, suppose that $A_n \in \Aff(\Cb^d)$ and $A_n\Omega$ converges to some $\mathcal{D}$ in $\Xb_d$. Fix some $z_0 \in \mathcal{D}$. Then, after passing to a subsequence, we can suppose that $z_0 \in A_n \Omega$ for all $n$. Let $z_n = A_n^{-1}z_0$. Then $A_n(\Omega, z_n) \rightarrow (\mathcal{D},z_0)$ in $\Xb_{d,0}$.

Suppose $z_n = x_n + iy_n \in \Rb^d + i\Rb^d$. Then let $T_n \in \Aff(\Cb^d)$ denote the translation $T_n(z) = z- iy_n$. Next, by Theorem~\ref{thm:real_affine_cocompact}, we can pass to a subsequence and find $B_n \in \Aff(\Rb^d)$ such that $B_n(C, x_n)$ converges to some $(C_\infty, x_\infty)$ in $\Yb_{d,0}$. Then extending each $B_n$ to an affine automorphism of $\Cb^d$, 
\begin{align*}
B_n T_n (\Omega, z_n) \rightarrow (C_\infty+i\Rb^d, x_\infty)
\end{align*}
in $\Xb_{d,0}$. But then, by Proposition~\ref{prop:limit_domain_depends_on_sequence}, there exists some $A \in \Aff(\Cb^d)$ such that 
\begin{align*}
\mathcal{D} = A (C_\infty+i\Rb^d) \in  \Aff(\Cb^d) \cdot \Big( \overline{\Aff(\Rb^d) \cdot C} \cap \Yb_d + i\Rb^d \Big).
\end{align*}
Thus 
\begin{align*}
\overline{\Aff(\Cb^d) \cdot \Omega} \cap \Xb_d \subset \Aff(\Cb^d) \cdot \Big( \overline{\Aff(\Rb^d) \cdot C} \cap \Yb_d + i\Rb^d \Big).
\end{align*}

Finally, the in particular part follows from the main assertion and Proposition~\ref{prop:affine_vs_holomorphic_disks}.

\end{proof}

\begin{proof}[Proof of Proposition~\ref{prop:the_other_direction}] By Corollary~\ref{cor:hilbert_gromov_necessary}, every domain in $\overline{\Aff(\Rb^d) \cdot C} \cap \Yb_{d}$ is strictly convex. So by Lemma~\ref{lem:affine_orbits_equal} every domain in $\overline{\Aff(\Cb^d) \cdot \Omega} \cap \Xb_{d}$ has simple boundary. Since $C$ is unbounded, ${\rm AC}(\Omega)$ is not totally real and hence $(\Omega, K_\Omega)$ is Gromov hyperbolic by Theorem~\ref{thm:main_equivalence_general_version}.
\end{proof}

\section{The squeezing function}\label{sec:squeezing}

In this section we construct Example~\ref{ex:non_strongly_pseudoconvex} by showing that an example of Forn{\ae}ss and Wold satisfies all the desired conditions. Their example was constructed to be a counterexample to a natural question concerning the squeezing function.

Given a domain $\Omega \subset \Cb^d$ biholomorphic to a bounded domain, let $s_\Omega : \Omega \rightarrow (0,1]$ denote the \emph{squeezing function on $\Omega$}, that is 
\begin{align*}
s_\Omega(z) = \sup\{ r : & \text{ there exists a holomorphic embedding } \\
& f: \Omega \hookrightarrow \Bb_d \text{ with } f(z)=0 \text{ and } r\Bb_d \subset f(\Omega) \}.
\end{align*}
The quantity $s_\Omega(z)$ can be seen as a measure of how close the complex geometry of $\Omega$ at $z$ is to the complex geometry of the unit ball. 

For strongly pseudoconvex domains, Diederich, Forn{\ae}ss, and Wold~\cite[Theorem 1.1]{DFW2014} and Deng, Guan, and Zhang~\cite[Theorem 1.1]{DGZ2016} proved the following. 

\begin{theorem}\cite{DFW2014, DGZ2016}\label{thm:sq_on_str}
If $\Omega \subset \Cb^d$ is a bounded strongly pseudoconvex domain with $C^2$ boundary, then $\lim_{z \rightarrow \partial \Omega} s_\Omega(z) = 1$.
\end{theorem}

 Based on the above theorem, it seems natural to ask if the converse holds.

\begin{question}\label{question:squeezing} (Forn{\ae}ss and Wold~\cite[Question 4.2]{FW2018}) Suppose $\Omega  \subset \Cb^d$ is a bounded pseudoconvex domain with $C^k$ boundary for some $k > 2$. If $\lim_{z \rightarrow \partial \Omega} s_\Omega(z) = 1$, is $\Omega$ strongly pseudoconvex?
\end{question}

In the convex case the answer is yes when $k>2$~\cite{Z2019} and no when $k=2$.

\begin{example}[Forn{\ae}ss and Wold~\cite{FW2018}]\label{ex:FW} For any $d \geq 2$ there exists a bounded convex domain $\Omega \subset \Cb^d$ with $C^2$ boundary such that $\Omega$ is not strongly pseudoconvex and $\lim_{z \rightarrow \partial \Omega} s_\Omega(z) = 1$.
\end{example}

The next theorem shows that the domains in Example~\ref{ex:FW} satisfy the claims in Example~\ref{ex:non_strongly_pseudoconvex}. 

\begin{theorem}\label{thm:squeezing} Suppose $d \geq 2$, $\Omega \subset \Cb^d$ is a bounded convex domain, $\partial\Omega$ is $C^{1}$, and $\lim_{z \rightarrow \partial\Omega} s_\Omega(z) = 1$. Then a subelliptic estimate of order $\epsilon$ holds for every $\epsilon \in (0,1/2)$. 
\end{theorem}

The theorem will require several lemmas. 

\begin{lemma}\label{lem:SQ_1}Suppose $\Omega \subset \Cb^d$ is a bounded convex domain and $\lim_{z \rightarrow \partial\Omega} s_\Omega(z) = 1$. If $z_n \in \Omega$ is a sequence with 
\begin{align*}
\lim_{n \rightarrow \infty} d_{\Euc}(z_n, \partial \Omega) = 0
\end{align*}
and $A_n \in \Aff(\Cb^d)$ are affine maps such that  $A_n(\Omega,z_n)$ converges to $(\mathcal{U},u)$ in $\Xb_{d,0}$, then $\mathcal{U}$ is biholomorphic to $\Bb_d$. 
\end{lemma}

\begin{proof} The function 
\begin{align*}
(\mathcal{D},z) \in \Xb_{d,0} \rightarrow s_\mathcal{D}(z)
\end{align*}
is upper semi-continuous (see for instance~\cite[Proposition 7.1]{Z2018}). So
\begin{align*}
1 \geq s_\mathcal{U}(u) \geq \lim_{n \rightarrow \infty} s_{A_n\Omega}(A_nz_n) = \lim_{n \rightarrow \infty} s_{\Omega}(z_n) =1.
\end{align*}
Hence $s_\mathcal{U}(u)=1$. Then by~\cite[Theorem 2.1]{DGZ2012}, $\mathcal{U}$ is biholomorphic to $\Bb_d$. 
\end{proof}

The proof of the next lemma uses the following result.

\begin{proposition}\cite[Proposition 2.1]{Z2019}\label{prop:asymptotic_shape} Suppose $\Omega \subset \Cb^d$ is a convex domain with 
\begin{enumerate}
\item $\Omega \cap (e_1+\Span_{\Cb}\{e_2,\dots, e_d\})=\emptyset$,
\item $\Omega \cap \Cb \cdot e_1 = \{ (z,0,\dots,0) \in \Cb^d:  { \rm Re}(z) < 1\}$, and
\item $\Omega$ is biholomorphic to $\Bb_d$. 
\end{enumerate}
If $v \in \Span_{\Cb}\{e_2,\dots, e_d\}$, then
\begin{align*}
\frac{1}{2} = \lim_{t \rightarrow \infty} \frac{1}{t} \log \delta_\Omega(-e^te_1; v). 
\end{align*}
\end{proposition}

\begin{remark} The theorem says that $\Omega$ asymptotically ``looks'' like the domain
\begin{align*}
\left\{ (z_1,\dots, z_d) : { \rm Re}(z_1) < 1 - \sum_{j=2}^d \abs{z_j}^2 \right\}
\end{align*}
which is biholomorphic to $\Bb_d$. 
\end{remark}

\begin{lemma}\label{lem:SQ_2}Suppose $\Omega \subset \Cb^d$ is a bounded convex domain, $\partial\Omega$ is $C^{1}$, and $\lim_{z \rightarrow \partial\Omega} s_\Omega(z) = 1$. Then $\Omega$ is $(2+a)$-convex for every $a > 0$. 
\end{lemma}

\begin{proof} Without loss of generality we may assume $0 \in \Omega$. Then, as in Section~\ref{sec:local_m_convexity}, for $z \in \Omega \setminus \{0\}$ let $\pi_\Omega(z) \in \partial \Omega$ be defined by
\begin{align*}
\{ \pi_\Omega(z) \} = \partial \Omega \cap  \Rb_{>0} \cdot z.
\end{align*}
Also, for $z \in \Omega \setminus \{0\}$ let $r_\Omega(z) = \norm{z-\pi_\Omega(z)}$ and let $T_\Omega(z)$ denote the set of unit vectors $v \in \Cb^d$ where 
\begin{align*}
( \pi_\Omega(z) + \Cb \cdot v) \cap \Omega = \emptyset. 
\end{align*}
Since $\Omega$ is convex and $\partial \Omega$ is $C^1$, the set $T_\Omega(z)$ coincides with a complex hyperplane intersected with the unit sphere. Also, if $z \in \Omega \setminus\{0\}$, then  $\overline{\Omega}$ contains the convex hull of $\Bb_d(0;\delta_\Omega(0))$ and $\pi_\Omega(z)$. Hence
\begin{equation}
\label{eqn : sq fcn estimate on r Omega}
r_\Omega(z) \leq \frac{\norm{\pi_\Omega(z)}}{\delta_{\Omega}(0)} \delta_\Omega(z) \leq  \frac{\max_{w \in \partial \Omega} \norm{w}}{\delta_{\Omega}(0)} \delta_\Omega(z)
\end{equation}
for all $z \in \Omega \setminus\{0\}$.

Fix $a >0$. We claim that  $\Omega$ is $(2+a)$-convex. Using Equation~\eqref{eqn : sq fcn estimate on r Omega} and the proof of Lemma~\ref{lem:Omega_R_any_direction}, it is enough to show that there exists $C > 0$ such that
\begin{align*}
\delta_\Omega(z;v) \leq C r_\Omega(z)^{1/(2+a)}
\end{align*}
for every $z \in \Omega \setminus \{0\}$ and $v \in T_\Omega(z)$. Suppose not, then there are sequences $z_m \in \Omega \setminus \{0\}$ and $v_m \in T_\Omega(z_m)$ such that 
\begin{align*}
\delta_\Omega(z_m;v_m) = C_m r_\Omega(z_m)^{1/(2+a)}
\end{align*}
and $C_m \geq m$. 

Since $\Omega$ is bounded, the quantity
\begin{align*}
M:=\sup \left\{ \delta_{ \Omega}(z;v) : z \in \Omega, v \in \Cb^d \setminus\{0\} \right\} 
\end{align*}
is finite. Then, since $C_m \geq m$, we must have 
\begin{align}
\label{eq:r_goes_to_zero_2convex}
\lim_{m \rightarrow \infty} r_\Omega(z_m) = 0.
\end{align}

Since $\Omega_{m}$ is convex, the function $f_m:[0,1]\rightarrow \Rb$ defined by
\begin{align*}
f_m(t) = \frac{\norm{\pi_{\Omega}(z_m)-tz_m}^{1/(2+a)}}{\delta_{\Omega}(tz_m;v_m)}
\end{align*}
is continuous. Let $t_m \in [0,1]$ be a minimum point of $f_m$. Notice that $f_m(1) = \frac{1}{C_m} \leq \frac{1}{m}$ and
$$
f_m(0) = \frac{ \norm{\pi_{\Omega}(z_m)}^{1/(2+a)}}{ \delta_{\Omega}(0;v_m)} \geq \frac{ \delta_\Omega(0)^{1/(2+a)}}{M}.
$$
So for $m$ sufficiently large, $f_m(1) < f_m(0)$ and hence $t_m \in (0,1]$. So after possibly passing to a tail of the sequence, replacing $z_m$ with $t_mz_m$, and increasing $C_m$, we can further assume that each $z_m$ has the following extremal property: 
\begin{align}
\label{eq:maximal_choice_squeezing}
\delta_{ \Omega}(tz_m;v_m) \leq C_m r_{\Omega}(tz_m)^{1/(2+a)}
\end{align}
for all $t \in (0,1]$. Finally, by replacing $v_m$ by some $e^{i\theta_m}v_m$ where $\theta_m \in \Rb$, we can assume that 
\begin{align*}
z_m + C_m r_{\Omega}(z_m)^{1/(2+a)} v_m \in \partial \Omega.
\end{align*}
Notice that $v_m$ is still contained in $T_{\Omega}(z_m)$. 

Let 
\begin{align*}
a_m := \pi_{\Omega}(z_m) \in \partial \Omega
\end{align*}
and 
\begin{align*}
b_m:=z_m + C_m r_{\Omega}(z_m)^{1/(1+a)} v_m \in \partial \Omega.
\end{align*}
Then let $B_m \in \Aff(\Cb^d)$ be an affine map such that $B_m(z_m)=0$, $B_m(a_m)=e_1$, and $B_m(b_m) = e_2$. 

For $r > 0$ and $\theta \in (0,\pi/2)$ let 
\begin{align*}
\mathcal{C}(r,\theta) = \{ x+iy \in \Cb: -r < x < 1, \abs{y} < \tan(\theta)(1-x) \}.
\end{align*}
Then $\mathcal{C}(r,\theta) \subset \Cb$ is a truncated cone based at $1$ in $\Cb$. Since $\partial\Omega$ is $C^1$ and $z_m$ converges towards the boundary, there exist sequences $r_m \rightarrow \infty$ and $\theta_m \rightarrow \pi/2$ such that 
\begin{align}
\label{eq:squeezing_cone_inclusion}
\mathcal{C}(r_m,\theta_m)\cdot e_1 \subset B_m \Omega.
\end{align}
In particular, there exists some $r  \in (0,1)$ such that 
\begin{align*}
r\Db \cdot e_1 \subset B_m\Omega
\end{align*}
for all $m$. Further, since $v_m \in T_{\Omega}(z_m)$, we see that 
\begin{align*}
B_m \Omega \cap (e_1 + \Cb \cdot e_2) = \emptyset.
\end{align*}
By construction $e_2=B_m(b_m) \in \partial B_m \Omega$ and since $\delta_{\Omega}(z_m;v_m) = \norm{b_m-z_m}$ we see that 
\begin{align*}
\Db \cdot e_2 \subset B_m\Omega.
\end{align*}
Thus 
\begin{align*}
B_m \Omega \cap \Span_{\Cb}\{e_1, e_2\} \in \Kb_2(r).
\end{align*}
So by Proposition~\ref{prop:compact_on_slices},  we can assume that $B_m \Omega \in \Kb_d(r)$. Then, since $\Kb_d(r)$ is compact, we can pass to a subsequence so that $B_m (\Omega,z_m)  \rightarrow (\mathcal{D},0)$ in $\Xb_{d,0}$. 

Lemma~\ref{lem:SQ_1} implies that $\mathcal{D}$ is biholomorphic to $\Bb_d$. We will use Proposition~\ref{prop:asymptotic_shape} to derive a contradiction. First, since $\mathcal{D} \in \Kb_d(r)$ we have
\begin{align*}
\mathcal{D}\cap (e_1+\Span_{\Cb}\{e_2,\dots, e_d\})=\emptyset.
\end{align*}
Next, Equation~\eqref{eq:squeezing_cone_inclusion} implies that 
\begin{align*}
\{ (z,0,\dots,0) \in \Cb^d: { \rm Re}(z) < 1\} \subset \mathcal{D}.
\end{align*}
Then, since $e_1 \in \partial \mathcal{D}$ and $\mathcal{D}$ is convex, we must have 
\begin{align*}
\{ (z,0,\dots,0) \in \Cb^d: { \rm Re}(z) < 1\} = \mathcal{D} \cap \Cb \cdot e_1.
\end{align*}
Finally we obtain a contradiction by verifying the following claim.

\medskip

\noindent \textbf{Claim:} $\delta_\mathcal{D}(-te_1; e_2) \leq (1+t)^{1/(2+a)}$ for every $t > 0$. 

\medskip

\noindent \emph{Proof of Claim:} Fix $t > 0$. Then for $m$ sufficiently large
\begin{align*}
B_m^{-1}(-te_1) \in (0,z_m)
\end{align*}
and
\begin{align*}
r_{\Omega}( B_m^{-1}(-te_1)) = (1+t)r_\Omega(z_m).
\end{align*}
Then by Equation~\eqref{eq:maximal_choice_squeezing}
\begin{align*}
\delta_{ \Omega}(B_m^{-1}(-te_1);v_m) \leq C_m (1+t)^{1/(2+a)} r_{\Omega}(z_m)^{1/(2+a)}.
\end{align*}
Then
\begin{align*}
\delta_{ B_m\Omega}(-te_1;e_2) = \frac{1}{C_mr_{\Omega}(z_m)^{1/(2+a)}} \delta_\Omega(B_m^{-1}(-te_1);v_m) \leq (1+t)^{1/(2+a)}.
\end{align*}
So
\begin{align*}
\delta_\mathcal{D}(-te_1; e_2) =\lim_{m \rightarrow \infty} \delta_{ B_m\Omega}(-te_1;e_2)  \leq (1+t)^{1/(2+a)}.
\end{align*}
This proves the claim. \hspace*{\fill}$\blacktriangleleft$

Now we have a contradiction: Proposition~\ref{prop:asymptotic_shape} implies that 
\begin{align*}
\frac{1}{2} = \lim_{t \rightarrow \infty} \frac{1}{t} \log \delta_\mathcal{D}(-e^te_1; e_2),
\end{align*}
while the claim implies that this limit is bounded above by $\frac{1}{2+a} < \frac{1}{2}$. 
\end{proof}

\begin{lemma} Suppose $\Omega \subset \Cb^d$ is a bounded convex domain and $\lim_{z \rightarrow \partial\Omega} s_\Omega(z) = 1$. Then $(\Omega, K_\Omega)$ is Gromov hyperbolic. 
\end{lemma}

\begin{proof} Using Theorem~\ref{thm:main_equivalence} we need to show that every domain in
\begin{align*}
\overline{\Aff(\Cb^d) \cdot \Omega} \cap \Xb_d
\end{align*}
has simple boundary. 

If 
$$
\mathcal{D} \in \overline{\Aff(\Cb^d) \cdot \Omega} \cap \Xb_d - \Aff(\Cb^d) \cdot \Omega,
$$
then $\mathcal{D}$ is biholomorphic to $\Bb_d$ by Lemma~\ref{lem:SQ_1}. So, in this case, $(\mathcal{D},K_\mathcal{D})$ is Gromov hyperbolic and hence $\mathcal{D}$ has simple boundary by Theorem~\ref{thm:main_equivalence}. So it suffices to show that $\Omega$ has simple boundary. However, if $\Omega$ has non-simple boundary, then there exists some 
$$
\mathcal{D}^\prime \in \overline{\Aff(\Cb^d) \cdot \Omega} \cap \Xb_d - \Aff(\Cb^d) \cdot \Omega,
$$
with non-simple boundary, see for instance~\cite[Proposition A.9]{GZ2020}, and we just showed that this is impossible. 
\end{proof}

\begin{proof}[Proof of Theorem~\ref{thm:squeezing}] Since $(\Omega, K_\Omega)$ is Gromov hyperbolic, Theorem~\ref{thm:optimal_order} says that a subelliptic estimate of order $\epsilon$ holds for all
\begin{align*}
\epsilon < \frac{1}{\alpha_\star(\Omega)m_\star(\Omega)}.
\end{align*}
Further $m_\star(\Omega) =2$ by Lemma~\ref{lem:SQ_2} and $\alpha_\star(\Omega)=1$ by Proposition~\ref{prop:alpha_star_reg}.
\end{proof}

\section{Miscellaneous Examples}\label{sec:misc_examples}

\subsection{The failure of the converse to Theorem~\ref{thm:main}}\label{sec:converse_example} In Example~\ref{ex:intersections} we constructed strongly convex domains $\Omega_1, \dots, \Omega_d$ such that 
\begin{align*}
\Omega: = \cap_{j=1}^d \Omega_j
\end{align*} 
is non-empty and $(\Omega, d_\Omega)$ is not Gromov hyperbolic. However, each $(\Omega_j, d_{\Omega_j})$ is Gromov hyperbolic by Corollary~\ref{cor:finite_type} and so $\Omega$ satisfies a subelliptic estimate by Theorem~\ref{thm:intersection}.

\subsection{Example~\ref{ex:local_cone_point}}\label{sec:ex_local_cone_point} In~\cite[Theorem 1.8]{Z2017} we proved that the Kobayashi metric on the convex cone
\begin{align*}
\mathcal{C} = \{ (z_0,z) \in \Cb \times \Cb^d : {\rm Im}(z_0) > \norm{z} \}
\end{align*}
is Gromov hyperbolic. Then by Theorem~\ref{thm:intersection_body} a subelliptic estimate holds on 
\begin{align*}
\Omega = \Bb_{d+1}(0;r) \cap \mathcal{C}
\end{align*}
for any $r > 0$.

\subsection{Example~\ref{ex:ae_flat}}\label{sec:ex_ae_flat} To construct Example~\ref{ex:ae_flat} we need to recall some facts about convex divisible domains in $\Pb(\Rb^d)$. 

\begin{definition} \ \begin{enumerate}
\item A domain $\Omega \subset \Pb(\Rb^d)$ is \emph{properly convex} if there exists an affine chart of $\Pb(\Rb^d)$ which contains $\Omega$ as a bounded convex domain. 
\item Two domains $\Omega_1, \Omega_2 \subset \Pb(\Rb^d)$ are \emph{projectively equivalent} if there exists some $g \in \PGL_d(\Rb)$ such that $g\Omega_1 = \Omega_2$. 
\item The projective automorphism group of a domain $\Omega \subset \Pb(\Rb^d)$ is 
\begin{align*}
\Aut_{\Pb}(\Omega) = \{ g \in \PGL_{d}(\Rb) : g \Omega = \Omega \}.
\end{align*}
\item A properly convex domain $\Omega \subset \Pb(\Rb^d)$ is called \emph{divisible} if there exists a discrete group $\Gamma \leq \Aut_{\Pb}(\Omega)$ which acts properly discontinuously, freely, and co-compactly on $\Omega$. 
\end{enumerate}
\end{definition}

Given a properly convex domain $\Omega \subset \Pb(\Rb^d)$, one can define the Hilbert distance on $\Omega$ by fixing an affine chart that contains $\Omega$ as a bounded convex domain and taking the Hilbert metric there. Using the projective invariance of the cross ratio, one can show that this definition does not depend on the choice of affine chart. 

The fundamental example of a properly convex divisible domain is the unit ball 
\begin{align*}
\Bc = \left\{ [1:x_1:\dots : x_{d-1}] \in \Pb(\Rb^d): \sum_{i=1}^{d-1} x_i^2 < 1\right\}.
\end{align*} 
Then $(\Bc, H_{\Bc})$ is the Klein-Beltrami model of real hyperbolic $(d-1)$-space and any real hyperbolic manifold can be identified with a quotient $\Gamma \backslash \Bc$ for some discrete group $\Gamma \leq \Aut_{\Pb}(\Bc)$ which acts properly discontinuously on $\Bc$. Since compact real hyperbolic manifolds exist in any dimension, this implies that $\Bc$ is divisible.  

It turns out that $\Bc$ is not the only example of a properly convex divisible domain.

\begin{theorem}[{Benoist~\cite[Corollary 2.10]{B2004}, Kapovich~\cite{K2007}}] For any $d \geq 3$ there exists a properly convex divisible domain $\Omega \subset \Pb(\Rb^d)$ such that $\Omega$ is not projectively equivalent to $\Bc$ and $(\Omega, H_\Omega)$ is Gromov hyperbolic.  \end{theorem}

Benoist~\cite{B2004} proved a number of results about these domains. To state his results we need one definition. 

\begin{definition}\label{defn:curvature_on_meas_zero} Suppose $\Omega \subset \Rb^d$ is a bounded convex domain with $C^1$ boundary. For $x \in \partial \Omega$, let $n_\Omega(x)$ be the inward pointing unit normal vector at $x$. Then \emph{the curvature of $\partial \Omega$ is concentrated on a set of measure zero} if the pull back of the Lebesgue measure on $\mathbb{S}^{d-1}$ under $n_\Omega$ is singular to the volume induced by some (hence any) Riemannian metric on $\partial \Omega$.  
\end{definition}

\begin{theorem}[{Benoist~\cite[Theorem 1.1, Theorem 1.2, Theorem 1.3]{B2004}}] Suppose $\Omega \subset \Pb(\Rb^d)$ is a properly convex divisible domain with $(\Omega, H_\Omega)$ Gromov hyperbolic. If $\Omega$  is not projectively equivalent to $\Bc$, then 
\begin{enumerate}
\item $\partial \Omega$ is $C^{1,\alpha}$ for some $\alpha > 0$ but not $C^{1,1}$,
\item $\Omega$ is strictly convex, and
\item the curvature of $\partial \Omega$ is concentrated on a set of measure zero
\end{enumerate}
\end{theorem}

Then the existence of Example~\ref{ex:ae_flat} follows from the previous two theorems.

\bibliographystyle{alpha}
\bibliography{SCV}

\end{document}